\documentclass[final,1p,times]{amsart}

\usepackage{amssymb}
\usepackage{amsthm}
\usepackage{amscd}
\usepackage{amsmath}
\usepackage{amsfonts}
\usepackage{amssymb}
\usepackage{graphicx}
\usepackage{color}
\usepackage{framed}
\usepackage{comment}
\usepackage{chngcntr}
\numberwithin{equation}{section}
\newtheorem{theorem}{Theorem}[section]
\newtheorem{proposition}[theorem]{Proposition}
\newtheorem{lemma}[theorem]{Lemma}
\newtheorem{corollary}[theorem]{Corollary}
\newtheorem{definition}[theorem]{Definition}
\newtheorem{remark}[theorem]{Remark}

\usepackage{mathrsfs}
\usepackage{titletoc}
\usepackage{appendix}
\usepackage{tikz}
\usepackage[breaklinks,colorlinks]{hyperref}
\usepackage[pagewise,mathlines]{lineno}
\usepackage{geometry}
\usepackage{bm}
\usepackage[shortlabels]{enumitem}
\usepackage{mathtools}
\usepackage{cite}

\newcommand{\D}{\Delta}
\newcommand{\ra}{\rightarrow}
\newcommand{\p}{\partial}
\newcommand{\f}{\frac}

\newcommand{\be}{\begin{equation}}
	\renewcommand{\ra}{\rightarrow}
	\newcommand{\ee}{\end{equation}}
\newcommand{\bea}{\begin{eqnarray}}
	\newcommand{\eea}{\end{eqnarray}}
\newcommand{\bna}{\begin{eqnarray*}}
	\newcommand{\ena}{\end{eqnarray*}}
\renewcommand{\o}{\omega}
\renewcommand{\O}{\Omega}
\renewcommand{\le}{\left}
\newcommand{\ri}{\right}

\newcommand{\Rmnum}[1]{\expandafter\@slowromancap\romannumeral #1@}
\makeatother

\usepackage{bbding}

\begin{document}
		
	\title{Nonexistence results for semilinear parabolic and hyperbolic equations on metric graphs}
	
	\author{Yang Liu}
	\address{Yau Mathematical Sciences Center, Tsinghua University, Beijing, 100084, China}
	\email{dliuyang@tsinghua.edu.cn}
	
	\author{Yong Lin}
	\address{Yau Mathematical Sciences Center; Department of Mathematical Sciences, Tsinghua University, Beijing, 100084, China}
	\email{yonglin@tsinghua.edu.cn}
	
	\author{Haohang Zhang}
	\address{Yau Mathematical Sciences Center; Department of Mathematical Sciences, Tsinghua University, Beijing, 100084, China}
	\email{zhanghh22@mails.tsinghua.edu.cn}
	\thanks{Corresponding author: Haohang Zhang}
	
	
	
	\begin{abstract}		
	This paper investigates the nonexistence of solutions to semilinear parabolic and hyperbolic inequalities with positive potentials on metric graphs, including both nonnegative solutions and sign-changing solutions. The Laplacian under consideration is of a nonstandard type, incorporating contributions from both the vertices and edges of the metric graph. We construct a new pseudo-metric and introduce suitable space-time test functions of either coupled or separated type. Under suitable weighted space-time volume growth conditions on the potential, we establish nonexistence results for very weak solutions. More precisely, we show that all such solutions to the inequality must be identically zero.
	\end{abstract}
	
	\maketitle
	
\section{Introduction}
This paper is concerned with the nonexistence of global solutions to semilinear parabolic and hyperbolic inequalities on metric graphs. Specifically, we study the parabolic problem
\begin{equation}\label{s1:1}\left\{\begin{array}{lll}
		u_t(x,t)\geq\Delta_\mathcal{G} u(x,t)+h(x,t) |u(x,t)|^{\sigma}, &(x,t)\in\mathcal{G}\times(0,\infty),&\\[1ex]
		u(x,0)=u_0(x),&x\in\mathcal{G}, &\end{array}\ri.		
\end{equation}
and the hyperbolic problem
\begin{equation}\label{s1:2}\left\{\begin{array}{lll}
		u_{tt}(x,t)\geq\Delta_\mathcal{G} u(x,t)+h(x,t) |u(x,t)|^{\sigma}, &(x,t)\in\mathcal{G}\times(0,\infty),&\\[1ex]
		u(x,0)=u_0(x),&x\in\mathcal{G},&\\[1ex]
		u_t(x,0)=u_1(x),&x\in\mathcal{G}, &\end{array}\ri.		
\end{equation}
where $h:\mathcal{G}\times[0,\infty)\ra\mathbb{R}$ is a positive potential, $\sigma > 1$, and $u_0,u_1:\mathcal{G} \to \mathbb{R}$ are given initial data. Here, $\Delta_\mathcal{G}$ denotes the Laplacian on the metric graph $\mathcal{G}=(\mathcal{V},\mathcal{E})$, incorporating both vertex and edge contributions (see Definition \ref{s0:D1} and Appendix \ref{D}). 

Problems \eqref{s1:1}  and \eqref{s1:2} are posed on a metric graph $\mathcal{G}$, a continuous spatial network, with edges as physical line segments joined at vertices. Unlike combinatorial graphs, therefore, the edges of a metric graph are not assigned directions or weights as abstract auxiliary relations between vertices; instead, they are more appropriately viewed as intervals glued together at vertices. This feature allows dynamical behavior to evolve along the edges. As a result, metric graphs provide a better description of network models arising in various fields; see \cite{R,S-M-A,M-S-V,M,B-K}. Within this framework, dynamical behavior on metric graphs is usually described by PDEs. These equations are defined on the edges and satisfy specific boundary conditions at the vertices, such as the Kirchhoff transmission condition or the homogeneous Neumann boundary condition. Such a setting arises in quantum wires \cite{K-Z,R-S} and has been extensively studied from a mathematical perspective \cite{F-T1,F-T2,Fr}. We will specify later the precise assumptions on the graph $\mathcal{G}$, see \eqref{s3:1}. While PDEs on combinatorial graphs have been extensively studied in the literature, e.g., \cite{B-C,S-T-Z,E-M,G-T,H-L,K-R,M2,S-S-V,L-H-N,G-L-Y-2,Z-L-Y1,H-L-Y,L-Yang,Liu1,G-L-Y-Z,H-L-M,L-W-Z,M-P-S,M-P2}, the setting of metric graphs has undergone substantial development, with key contributions presented in \cite{A-S-T,B-D-S,C-J-S,D-S-T,K-M-N,P-T1,P-T2}.

The analysis of semilinear parabolic equations of the form \eqref{s1:1} has a rich history in Euclidean space and on manifolds. In particular, the global existence and finite-time blow-up of solutions have been extensively investigated in $\mathbb{R}^N$. For $h\equiv1$, it was shown by Fujita in \cite {Fu}, and by Hayakawa in \cite {Ha} for the critical case, that
\begin{itemize}
	\item if $1<\sigma\leq\sigma^*:=1+\frac{2}{N}$, then any nonnegative nontrivial solution blows up in finite time;
	\item if $\sigma>\sigma^*$, then there exists a global solution corresponding to a nonnegative initial value that is sufficiently small in a suitable sense. 
\end{itemize}
Analogous results for the parabolic equation \eqref{s1:1} have also been established on manifolds \cite{B-P-T,G-S-X-X,M-M-P,Z}, as well as in bounded domains of $\mathbb{R}^N$ \cite{Me1,Me2}. In recent years, the blow-up and global existence of solutions to \eqref{s1:1} on combinatorial graphs has been widely studied. In the case of $h\equiv0$, under suitable assumptions on the graph, Huang \cite{Huang} proved that there exists at most one solution to the parabolic problem
 \begin{equation*}\left\{\begin{array}{lll}
 		u_t(x,t)=\Delta_\mathcal{V} u(x,t), &(x,t)\in\mathcal{V}\times(0,\infty),&\\[1ex]
 		u(x,0)=0,&x\in\mathcal{V}, &\end{array}\ri.		
 \end{equation*}
 provided that $u$ satisfies a growth condition: for some $C>0$, $0<c<1/2$, and a diverging
 sequence $\{R_n\}\subset(0,\infty)$, 
 $$\int_0^T\sum_{x\in B_{R_n}(x_0)}\mu_{\mathcal{V}}(x)u^2(x,t)dt\leq Ce^{cR_n\log R_n},\quad\text{for\ any\ } n\ \text{large\ enough}.$$
 Recently, this result was generalized in \cite{H-K-S} under weaker growth conditions.  It is worth mentioning that, using heat kernel estimates, Lin-Wu \cite{L-W} showed the existence and nonexistence of global nonnegative solutions for the semilinear heat equation 
\begin{equation*}\left\{\begin{array}{lll}
		u_t(x,t)=\Delta_\mathcal{V} u(x,t)+u^{1+\alpha}(x,t), &(x,t)\in\mathcal{V}\times(0,\infty),&\\[1ex]
		u(x,0)=a(x),&x\in\mathcal{V}. &\end{array}\ri.		
\end{equation*}
More recently, under a suitable space-time weighted volume growth condition, Monticelli-Punzo-Somaglia \cite{M-P-S3} investigated the nonexistence of nonnegative nontrivial global solutions to the vertex-based equation \eqref{s1:1}  in the case of $d\mu_{\mathcal{E}}=0$ by imposing  that for some $x_0\in \mathcal{V}$, $R_0>1$, $\alpha\in[0,1]$, $C>0$, 
\begin{equation}\label{s1:3}
	\Delta d(x,x_0)\leq \frac{C}{d^\alpha(x,x_0)}, \quad\forall x\in \mathcal{V}\setminus B_{R_0}(x_0).
\end{equation}
Furthermore, Meglioli \cite{Me}  established the uniqueness of solutions to the vertex-based  equations 
\begin{equation*}\left\{\begin{array}{lll}
		\rho(x) u_t(x,t)-\Delta_\mathcal{V} u(x,t)=0, &(x,t)\in\mathcal{V}\times(0,T],&\\[1ex]
		u(x,0)=0,&x\in\mathcal{V}. &\end{array}\ri.		
\end{equation*}
It was then extended to the edge-based equations on metric graphs by Meglioli-Punzo \cite{M-P}. For other relevant works, representative studies are provided in \cite{G-W,M-Q-D,Wu,L-L-W,M-P-S2}. 

We next turn to the hyperbolic problem \eqref{s1:2} on Euclidean spaces or manifolds. Unlike the case of parabolic equations, the theory concerning nonexistence results for hyperbolic problems has developed at a significantly slower pace. To our knowledge, the earliest contributions addressing the blow-up or non-existence of global solutions to hyperbolic equations in $\mathbb{R}^N$ trace back to \cite{K}, with a more general framework subsequently explored in \cite{M-Po,S}. These classical results establish that the Cauchy problem
\begin{equation*}\left\{\begin{array}{lll}
		u_{tt}(x,t)\geq\Delta u(x,t)+ |u(x,t)|^{q}, &(x,t)\in \mathbb{R}^N\times(0,\infty),&\\[1ex]
		u(x,0)=u_0(x),&x\in\mathbb{R}^N,&\\[1ex]
		u_t(x,0)=u_1(x),&x\in\mathbb{R}^N, &\end{array}\ri.		
\end{equation*}
fails to admit any nontrivial global solution under the conditions
$$1<q\leq \frac{n+1}{n-1},\quad\liminf_{R\to\infty}\int_{B_R(0)}u_1(x)dx\geq0.$$
Extensions to Riemannian manifolds have also been considered, such as in \cite{M-P-Sa}, where Monticelli-Punzo-Squassina treat the problem
\begin{equation*}\left\{\begin{array}{lll}
		u_{tt}(x,t)\geq\Delta u(x,t)+ v(x,t)|u(x,t)|^{q}, &(x,t)\in M\times(0,\infty),&\\[1ex]
		u(x,0)=u_0(x),&x\in M,&\\[1ex]
		u_t(x,0)=u_1(x),&x\in M, &\end{array}\ri.		
\end{equation*}
with $M$ a complete noncompact manifold, $\Delta$ the Laplace-Beltrami operator, $q>1$, and $v>0$ locally integrable. Under suitable curvature and volume growth assumptions, a nonexistence result for very weak solutions is established. In a more recent work \cite{M-P-S2}, Monticelli-Punzo-Somaglia extended the manifold setting to combinatorial graphs and derived analogous results. However, to the best of our knowledge, nonexistence results for semilinear hyperbolic equations on metric graphs remain largely unexplored.

In the study of semilinear parabolic and hyperbolic equations on metric graphs, a major technical difficulty arises from the singular behavior of the original distance function. In general, $d(\cdot,x_0)$ is not differentiable at certain points of the graph, which complicates the construction of test functions and the derivation of integral identities. Furthermore, when integration by parts is carried out edge by edge, additional boundary terms appear at the vertices. The treatment of these vertex contributions is often delicate and depends strongly on the underlying graph structure. 

The present paper may be viewed as a continuation and refinement of our previous work \cite{L-L-Z} on elliptic inequalities over metric graphs. In that paper, a mollification procedure was introduced to regularize the distance function and establish the corresponding nonexistence results. While the method was well suited to the elliptic setting, the regularized distance still retained certain singular features inherited from the original distance function. As a consequence, some technical difficulties persisted in the analysis. To address these issues, we develop here a new pseudo-distance construction, which is particularly adapted to the treatment of parabolic and hyperbolic inequalities and leads to a substantially simpler analytical framework. This construction yields a smooth replacement of the distance along the edges of the graph and avoids the singularities present in the original distance function. More importantly, it leads to an integration-by-parts formula in which the boundary contributions at the vertices disappear naturally, allowing the vertex and edge components of the graph to be handled within a unified framework. This feature makes the subsequent analysis applicable to both vertex-based and edge-based Laplacians and plays a crucial role in the derivation of the a priori estimates underlying the test function method.

Our approach allows us to establish nonexistence results for both nonnegative and sign-changing solutions to \eqref{s1:1} and \eqref{s1:2} under suitable space-time growth conditions on the potential $h$. Notably, our results recover previous Liouville-type theorems in the purely vertex-based setting \cite{M-P-S3,M-P-S2} and extend them to the metric graph context. Compared with the approaches in \cite{L-W,L-L-W,Wu}, our method is fundamentally different: for the heat equation \eqref{s1:1}, we do not rely on semigroup theory or heat kernel estimates, but instead develop a framework based on pseudo-distance and carefully constructed test functions, which allows a direct treatment of both vertex and edge contributions. This construction is central to proving the nonexistence results in both the parabolic and hyperbolic settings, particularly when dealing with sign-changing solutions.

The analysis of Problem \eqref{s1:1} in this paper covers two cases: nonnegative solutions (see Theorem \ref{s2:T3}) and general sign-changing solutions (see Theorem \ref{s2:T4}). For nonnegative solutions, the proof relies on a priori estimates derived by choosing suitable time-space coupled test functions. On the other hand, this approach cannot be applied to the case of general sign-changing solutions. For this purpose, we draw inspiration from \cite{M-P-S2} and use time-space separated test functions, which possess certain exponential decay properties with respect to the variable $x$ at infinity and compact support with respect to the variable $t$. Our a priori estimates rely on integration by parts formulas. The modified distance eliminates the outward normal contributions at the boundary, making them similar to vertex-based formulas. 

The derivative estimates for the test functions obtained in \cite{M-P-S,M-P-S2,M-P-S3} rely on a pseudo-metric defined on the vertex set and satisfying a collection of structural properties. Rather than providing an explicit construction of such a function, the authors assume that it enjoys several geometric and analytic properties, including the $q$-intrinsic condition and  a Laplacian estimate of the form \eqref{s1:3}. By contrast, in the present paper we explicitly construct a pseudo-distance on the whole metric graph. This construction allows us to verify the required geometric properties directly and provides a unified framework for the treatment of both vertices and edges. However, working on the entire metric graph also introduces new analytical difficulties, since the estimates for the test functions now depend on the behavior of the pseudo-distance inside the edges, requiring suitable bounds for its derivatives. Finally, by choosing appropriate classes of potentials, we obtain several corollaries. The analysis of Problem \eqref{s1:2} follows essentially the same line of argument as that of Problem \eqref{s1:1}. For this reason, in the proofs of Theorems \ref{s2:T1} and \ref{s2:T2} we focus only on the points that require additional arguments and omit the routine modifications.\\


The remaining parts of this paper are organized as follows: In Section \ref{2}, we state the assumptions on metric graphs and present the main results together with their corollaries. Section \ref{3} is devoted to the proofs for the parabolic equation \eqref{s1:1}, including the nonexistence of nonnegative solutions (Theorem \ref{s2:T3}) and sign-changing solutions (Theorem \ref{s2:T4}). In Section \ref{4}, we mainly consider Problem \eqref{s1:2} and obtain the corresponding results (Theorems \ref{s2:T1} and \ref{s2:T2}). Essential preliminaries for analysis on metric graphs are provided in the Appendix: Section \ref{A} introduces core basic notions and foundational definitions, Section \ref{B} presents two measures based on vertices and edges respectively, Section \ref{C} concerns functional spaces on metric graphs, and Section \ref{D} gives the definition of the Laplacian $\Delta_{\mathcal{G}}$.

\section{Statement of the main results}\label{2}

Throughout this paper, for any $x_0\in\mathcal{V}$ and $R>0$, we denote by \[ B_R(x_0):=\{x\in\mathcal{G}:d(x,x_0)<R\} \] the metric ball centered at $x_0$ with radius $R$. In what follows, we assume that the metric graph $\mathcal{G}$ satisfies the following conditions:
\begin{equation}\label{s3:1}
	\begin{cases}
		\text{ (i) }	\mathcal{G} \text{ is an infinite, connected, locally finite metric graph;}\\
		\text{ (ii) }  \text{the edge lengths are uniformly bounded, i.e., } j:=\sup_{e\in \mathcal{E}}l_e<\infty \text{ and } r:=\inf_{e\in\mathcal{E}}l_e>0;\\
		\text{ (iii) }	\text{there exists } x_0\in\mathcal{V} \text{ and a sequence } \{x_n\}\subset \mathcal{V}  \text{ such that } d(x_0, x_n) \text{ tends to } +\infty; \\
		\text{ (iv) }	\text{ for every } R>0, B_R(x_0) \text{ is finite, consisting of finitely many vertices and edges;} \\
		\text{ (v) }	\text{ there exists a constant } C > 0	\text{  such that for every } x\in\mathcal{V}, \sum_{y\sim x}\o(x,y)\leq C\mu_{\mathcal{V}}(x).
	\end{cases}
\end{equation}

\subsection{Definition of solutions}

For a function $f:\mathcal G\times[0,\infty)\to\mathbb R$, we use $f'$ to denote differentiation with respect to the spatial variable $x$, while $f_t$ and $f_{tt}$ denote the first-order and second-order derivatives with respect to the time variable $t$, respectively.  We begin with the notion of very weak solutions to the parabolic problem \eqref{s1:1}.
\begin{definition}\label{s2:D1}
	We say that $u:\mathcal G\times[0,\infty)\to\mathbb R$ is a very weak solution to \eqref{s1:1} if the following conditions hold:
	\begin{itemize}
		\item  for every $x\in\mathcal{G}$,  $u(x,\cdot)\in L_{\text{loc}}^1([0,\infty))\cap L^\sigma_{\text{loc}}([0,\infty),h(x,\cdot)dt)$;
		\item for every $t\in[0,\infty)$, $u(\cdot,t)\in\mathcal{D}(\mathcal{G})$, and $[\mathcal{K}(u)](x,t)=0$ for all $x\in\mathcal{V}$, where $\mathcal{D}(\mathcal{G})$ and $[\mathcal{K}(u)](\cdot,t)$ are given in \eqref{s2:2} and \eqref{s2:4}, respectively;
		\item for every nonnegative function $\varphi: \mathcal{G} \times [0,\infty) \to \mathbb{R}$ satisfying $\text{supp}(\varphi) \subset A \times [0,T]$ for some $T > 0$ and compact set $A \subset \mathcal{G}$, with $\varphi(x,\cdot) \in C^1([0,\infty))$ for all $x \in \mathcal{G}$, the inequality
		\begin{align}\label{s2:9}
			\nonumber\int_0^\infty &\int_{\mathcal{G}}\varphi(x,t)\Delta_{\mathcal{G}} u(x,t)d\mu_{\mathcal{G}}dt+	\int_0^\infty \int_{\mathcal{G}}h(x,t)|u(x,t)|^\sigma\varphi(x,t)d\mu_{\mathcal{G}} dt\\
			&+\int_0^\infty \int_{\mathcal{G}}u(x,t)\varphi_t(x,t)d\mu_{\mathcal{G}}dt+\int_\mathcal{G}u_0(x)\varphi(x,0)d\mu_{\mathcal{G}}\leq0
		\end{align} 
		holds, where the integral over $\mathcal{G}$ is defined by \eqref{s2:22}.
	\end{itemize}
\end{definition}
The condition $u(\cdot,t)\in\mathcal D(\mathcal G)$ ensures that the Laplacian $\Delta_{\mathcal G}u$ is well defined. The corresponding notion for \eqref{s1:2} is defined analogously. More precisely, we assume that the first two conditions in Definition \ref{s2:D1} are satisfied and that,  for every nonnegative test function $\varphi: \mathcal{G} \times [0,\infty) \to \mathbb{R}$ with $\text{supp}(\varphi) \subset A \times [0,T]$, for some $T > 0$ and compact set $A \subset \mathcal{G}$, and such that $\varphi(x,\cdot) \in C^1([0,\infty))$ for all $x \in \mathcal{G}$, the following inequality holds:
\begin{align}\label{s2:10}
	\nonumber-\int_0^\infty \int_{\mathcal{G}}\varphi_{tt}(x,t) u(x,t)d\mu_{\mathcal{G}}dt&+\int_0^\infty \int_{\mathcal{G}}\varphi(x,t)\Delta_{\mathcal{G}} u(x,t)d\mu_{\mathcal{G}}dt+	\int_0^\infty \int_{\mathcal{G}}h(x,t)|u(x,t)|^\sigma\varphi(x,t)d\mu_{\mathcal{G}} dt\\
	&+\int_\mathcal{G}u_1(x)\varphi(x,0)d\mu_{\mathcal{G}}-\int_\mathcal{G}u_0(x)\varphi_t(x,0)d\mu_{\mathcal{G}}\leq0.
\end{align}

	\subsection{Parabolic problem}
	
We first establish Liouville-type nonexistence results for very weak solutions to \eqref{s1:1} under suitable weighted space-time volume growth assumptions on the potential $h$. Throughout this paper, we set $ R_0:=\max\{2j,1\}$.

\begin{theorem}\label{s2:T3}
	Let $\mathcal{G}$ satisfy \eqref{s3:1}. Assume that the potential $h:\mathcal G\times[0,\infty)\to\mathbb R$ is positive, $\sigma>1$, and 
	\[
	\theta_1\ge2,\quad
	\theta_2\ge1,\quad
	\theta_1\ge\theta_2.
	\] Suppose that there exists a constant $C>0$ such that for every $R\geq R_0$
	\begin{equation}\label{s3:10}
		\int_0^\infty \left\{\sum_{x\in\mathcal{V}}\mu_{\mathcal{V}}(x) h^{-\frac{1}{\sigma-1}}(x,t)\textbf{1}_{E_R}(x,t)\right\}dt\leq CR^{\frac{\sigma}{\sigma-1}},\quad\int_0^\infty\left\{\sum_{e\in \mathcal{E}}\int_{0}^{l_e}h_e^{-\frac{1}{\sigma-1}}(x,t)\textbf{1}_{E_R}(x,t)dx\right\}dt\leq CR^{\frac{\sigma}{\sigma-1}},
	\end{equation} 
	where 
	$$E_R:=\le\{(x,t)\in\mathcal{G}\times[0,\infty):R^{\theta_1}\leq d(x,x_0)^{\theta_1}+t^{\theta_2}\leq 2R^{\theta_1}\ri\}.$$
	Then any nonnegative very weak solution $u:\mathcal{G}\times[0,\infty)\to\mathbb{R}$ to  \eqref{s1:1}  satisfies \[ u\equiv0, \quad\text{on }\mathcal{G}\times[0,\infty). \]
\end{theorem}

Next, we present several direct consequences of Theorem \ref{s2:T3} corresponding to some typical classes of potentials.
\begin{corollary}\label{s2:C1}
	Let $\mathcal{G}$ satisfy \eqref{s3:1}, let $\sigma>1$, and let $ h:\mathcal{G}\times[0,\infty)\to\mathbb{R} $ be a positive function. Assume that $ u:\mathcal{G}\times[0,\infty)\to\mathbb{R} $ is a nonnegative very weak solution to \eqref{s1:1}. Suppose that one of the following assumptions holds.
	\begin{itemize}
		\item The potential satisfies $h(x,t)\geq f(t)g(x)$ for all $(x,t)\in \mathcal{G}\times[0,\infty)$, for some positive functions $f$ and $g$. Moreover, assume that there exist constants $C>0$, $T_0>0$, and $\xi_1,\xi_2\geq0$ such that for every $R\geq R_0$ and $T\geq T_0$, 		\begin{equation}\label{s3:9}
			\int_0^Tf^{-\frac{1}{\sigma-1}}(t)dt\leq CT^{\xi_1},\quad\sum_{x\in\mathcal{V} \cap B_{R}}\mu_{\mathcal{V}}(x)g^{-\frac{1}{\sigma-1}}(x)\leq CR^{\xi_2},\quad \sum_{e\in \mathcal{E}\cap B_R}\int_{0}^{l_e}g_e^{-\frac{1}{\sigma-1}}(x)dx\leq CR^{\xi_2},
		\end{equation} 
		where
		$$\xi_1+\xi_2\leq \frac{\sigma}{\sigma-1}.$$
		
		\item The potential satisfies $h(x,t)\geq g(x)$ for all $(x,t)\in \mathcal{G}\times[0,\infty)$, for some positive function $g$. Assume further that for every $R\geq R_0$,
		\begin{equation}\label{s3:14}
			\sum_{x\in\mathcal{V} \cap B_{R}}\mu_{\mathcal{V}}(x)g^{-\frac{1}{\sigma-1}}(x)\leq CR^{\frac{1}{\sigma-1}},\quad\sum_{e\in \mathcal{E}\cap B_R}\int_{0}^{l_e}g_e^{-\frac{1}{\sigma-1}}(x)dx\leq CR^{\frac{1}{\sigma-1}}.
		\end{equation} 	
	\end{itemize}
	Hence, \eqref{s1:1} admits no nontrivial nonnegative very weak solutions.
\end{corollary}
As an immediate application of Corollary \ref{s2:C1}, we obtain the following result for the autonomous equation with constant potential.
\begin{corollary}\label{s2:C4}
	Let $\mathcal{G}$ satisfy \eqref{s3:1}, and let $\sigma>1$.  If $u$ is a nonnegative very weak solution to  
	\begin{equation*}\left\{\begin{array}{lll}
			u_t(x,t)\geq\Delta_{\mathcal{G}} u(x,t)+u^{\sigma}(x,t), &(x,t)\in\mathcal{G}\times(0,\infty),&\\[1ex]
			u(x,0)=u_0(x),&x\in\mathcal{G}.&\end{array}\ri.		
	\end{equation*} 
	Suppose that there exists a constant $C>0$ such that for every $R\geq R_0$,
	\begin{equation*}
		\mu_\mathcal{V}(B_R)\leq CR^{\frac{1}{\sigma-1}},\quad	\mu_\mathcal{E}(B_R)\leq CR^{\frac{1}{\sigma-1}}.
	\end{equation*} 
	Consequently, every such solution is trivial.
\end{corollary}

\begin{remark}
	\begin{enumerate}
		\item On metric graphs, the geometric structure naturally consists of two different components: the discrete set of vertices and the one-dimensional edge segments connecting them. 
		Consequently, volume estimates on $\mathcal G$ involve both the discrete measure $\mu_{\mathcal V}$ and the one-dimensional edge measure $\mu_{\mathcal E}$. 
		For this reason, throughout the paper the weighted volume growth assumptions are stated separately on vertices and edges.
		\item It is worth noting that condition \eqref{s3:10} is only used for sufficiently large values of $R$ in the proof of Theorem \ref{s2:T3}. More precisely, the inclusion $\tilde E_R\subset M_R$,
		established in Lemma \ref{s4:L1}, is required only when 
		$$R\ge \max\{4j,1\}.$$ 
		Hence the assumption may be imposed only for all sufficiently large $R$. The same conclusion applies to Corollaries \ref{s2:C1} and \ref{s2:C4}.
	\end{enumerate}
\end{remark}

The restriction to nonnegative solutions in the above results is mainly technical, as it allows one to apply comparison and integral estimates. The later theorems relax this restriction by working in suitable weighted spaces with exponential decay. For later use, given $\delta>0$ and $R\geq R_0$, we introduce the weighted Lebesgue space 
\begin{equation}\label{s3:16} 
	X_\alpha := L^1_{e^{-\alpha d(x,x_0)}}(\mathcal{G}), \qquad \alpha=\frac{\delta}{R},
\end{equation} 
where $ L^1_{e^{-\alpha d(x,x_0)}}(\mathcal{G}) $ is defined in \eqref{s2:13}.

The proof of Theorem \ref{s2:T3} relies on suitable compactly supported space-time coupled test functions. However, in the study of sign-changing solutions, the above argument is no longer applicable. Instead, we employ a different class of test functions which are separated in the space and time variables. In particular, these test functions are globally supported with respect to the spatial variable and exhibit suitable exponential decay at infinity. This highlights the delicate interplay between the geometry of the metric graph and the volume growth of the potential. As a consequence, additional assumptions are required. More precisely, besides stronger weighted space-time volume growth conditions on the potential $h$, we also need suitable assumptions on the initial datum and on the weighted integrability of solutions in the space $ L^1_{\mathrm{loc}}([0,\infty);X_\alpha)$. We are now in a position to state the corresponding nonexistence result for sign-changing solutions to \eqref{s1:1}.

\begin{theorem}\label{s2:T4}
	Let $\mathcal{G}$ satisfy \eqref{s3:1}, and let $h: \mathcal{G}\times[0,\infty)\ra\mathbb{R}$ be a positive function such that for every $R\geq R_0$,
	\begin{equation}\label{s3:15}
		\int_R^{2R} \le\{\sum_{x\in \mathcal{V}\cap B_{R}}\mu_{\mathcal{V}}(x) h^{-\frac{1}{\sigma-1}}(x,t) e^{-\alpha d(x,x_0)} \ri\}dt\leq CR^{\frac{\sigma}{\sigma-1}},	\int_R^{2R}\le\{\sum_{e\in \mathcal{E}\cap B_{R}}\int_{0}^{l_e}	h_e^{-\frac{1}{\sigma-1}}(x,t)e^{-\alpha d(x,x_0)}dx \ri\}dt\leq CR^{\frac{\sigma}{\sigma-1}},
	\end{equation}
	and
	\begin{equation}\label{s3:11}
		\int_0^{2R} \le\{\sum_{x\in\mathcal{V}\cap B_{R}^c}\mu_{\mathcal{V}}(x) h^{-\frac{1}{\sigma-1}}(x,t) e^{-\alpha d(x,x_0)}\ri\}dt\leq CR^{\frac{\sigma}{\sigma-1}},\int_0^{2R}\le\{\sum_{e\in\mathcal{E}\cap B_{R}^c}\int_{0}^{l_e}	h_e^{-\frac{1}{\sigma-1}}(x,t)e^{-\alpha d(x,x_0)}dx \ri\}dt\leq CR^{\frac{\sigma}{\sigma-1}}.
	\end{equation}
    Assume that $u$ is a very weak solution to \eqref{s1:1} with $\sigma>1$ satisfying $u\in L^1_{\text{loc}}([0,\infty); X_\alpha)$, and suppose that the initial datum $u_0\in X_\alpha$ satisfies 
	\begin{equation}\label{s3:12}
		\int_{\mathcal{G}}u_0(x)d\mu_\mathcal{G}\geq0.
	\end{equation}
	Then the corresponding problem admits no nontrivial solution.
\end{theorem}

\begin{remark}
	The exponential weight $e^{-\alpha d(x,x_0)}$ is introduced to compensate for the lack of compact support of the spatial test functions used in the study of sign-changing solutions. In particular, it guarantees the integrability of the space-time quantities arising in the weak formulation. It also ensures the validity of the integration by parts argument on the whole metric graph. The weighted integrability conditions in Theorem \ref{s2:T4} reflect a balance between the growth of the graph and the decay of the potential. Moreover, the parameter $\alpha$ depends on the scaling parameter $R$, which plays an essential role in the limiting procedure as $R\to\infty$.
\end{remark}

Next, we collect several corollaries of Theorem \ref{s2:T3} in the case where the potential $h$ takes a special form.

\begin{corollary}\label{s2:C2}
	Let $\mathcal{G}$ be a metric graph satisfying \eqref{s3:1} and $\sigma>1$, and let the potential $h: \mathcal{G}\times[0,\infty)\ra\mathbb{R}$ be a positive function such that $h(x,t)\geq g(x)$ for all $(x,t)\in \mathcal{G}\times[0,\infty)$, for some positive function $g$. Suppose that $u$ is a very weak solution of \eqref{s1:1} such that $u\in L^1_{\text{loc}}([0,\infty); X_\alpha)$, and assume that the initial datum $u_0\in X_\alpha$ satisfies \eqref{s3:12}. Under the volume growth condition
	\begin{equation}\label{s3:7}
		\sum_{x\in\mathcal{V}}\mu_{\mathcal{V}}(x) g^{-\frac{1}{\sigma-1}}(x) e^{-\alpha d(x,x_0)}\leq CR^{\frac{1}{\sigma-1}},\quad \sum_{e\in \mathcal{E}}\int_{0}^{l_e}	g_e^{-\frac{1}{\sigma-1}}(x)e^{-\alpha d(x,x_0)}dx\leq CR^{\frac{1}{\sigma-1}},\quad\forall\ R\geq R_0,
	\end{equation}
	any very weak solution of \eqref{s1:1} must be trivial.
\end{corollary}

\begin{corollary}\label{s2:C3}
	Let $\mathcal{G}$ be a metric graph satisfying \eqref{s3:1} and $\sigma>1$. Suppose that $u$ is a very weak solution of \eqref{s1:1} with $h\equiv1$ and $u\in L^1_{\text{loc}}([0,\infty); X_\alpha)$,  and that the initial datum $u_0\in X_\alpha$ satisfying \eqref{s3:12}. Moreover, assume that for every $R\geq R_0$, it holds that 
	\begin{equation}\label{s3:8}
		\mu_\mathcal{V}(B_{R+1}\setminus B_R)\leq CR^{\frac{2-\sigma}{\sigma-1}},\quad\mu_\mathcal{E}(B_{R+1}\setminus B_R)\leq CR^{\frac{2-\sigma}{\sigma-1}}.
	\end{equation}
	Then 
	\[ u\equiv0, \quad\text{on }\mathcal{G}\times[0,\infty). \]
\end{corollary}

	\subsection{Hyperbolic problem}
	
	We now turn to the hyperbolic inequality \eqref{s1:2}. For any function $u$, we denote $ u^+:=\max\{u,0\}$ and  $u^-:=\min\{u,0\}$. Throughout this subsection, we keep the notation introduced above for $R_0$, $\alpha$, and the weighted space $X_\alpha$. We first establish a Liouville-type theorem for nonnegative solutions to \eqref{s1:2} under suitable weighted volume growth assumptions on the potential.
	
		\begin{theorem}\label{s2:T1}
			Let $\mathcal{G}$ satisfy \eqref{s3:1} and let $h:\mathcal{G}\times[0,\infty)\to\mathbb{R}$ be a positive function. Assume that
			\[
			\theta_1\geq2,\quad
			\theta_2\geq2,\quad
			\theta_1\geq\theta_2/2,
			\]
			and that there exists a constant $C>0$ such that for any $R\geq R_0$, 
			\begin{equation}\label{s3:17}
				\int_0^\infty \left\{\sum_{x\in\mathcal{V}}\mu_{\mathcal{V}}(x) h^{-\frac{1}{\sigma-1}}(x,t)\textbf{1}_{E_R}(x,t)\right\}dt\leq CR^{\frac{\sigma}{\sigma-1}},\int_0^\infty\left\{\sum_{e\in \mathcal{E}}\int_{0}^{l_e}h_e^{-\frac{1}{\sigma-1}}(x,t)\textbf{1}_{E_R}(x,t)dx\right\}dt\leq CR^{\frac{\sigma}{\sigma-1}},
			\end{equation} 
			where 
			$$E_R=\le\{(x,t)\in\mathcal{G}\times[0,\infty):R^{\theta_1}\leq d(x,x_0)^{\theta_1}+t^{\theta_2}\leq 2R^{\theta_1}\ri\}.$$ 
			 If $u$ is a nonnegative very weak solution of \eqref{s1:2} satisfying
			 \begin{equation}\label{s3:13}
			 	\liminf_{R\to\infty}\left\{\int_{B_{R-j}}u_1^+(x)d\mu_{\mathcal{G}}+\int_{B_{2^{1/\theta_1}R+j}}u_1^-(x)d\mu_{\mathcal{G}}\right\}\geq0,
			 \end{equation}
			then \[ u\equiv0, \quad\text{on }\mathcal{G}\times[0,\infty). \]
	\end{theorem}
	
	\begin{remark}
		The parameters $\theta_1$ and $\theta_2$ characterize the relative scaling between the spatial and temporal variables in the space-time region $E_R$. The assumptions $\theta_1\geq2$ and $\theta_2\geq1$ (respectively $\theta_2\geq2$ in the hyperbolic case) are required in order to control the spatial Laplacian term and the time derivative terms in the test function argument. The additional relations $\theta_1\geq\theta_2$ and $\theta_1\geq\theta_2/2$ reflect the different scaling structures of the parabolic and hyperbolic problems.
	\end{remark}
	
	\begin{remark}
		The hyperbolic problem exhibits substantially different analytical features from the parabolic one due to the presence of the second-order time derivative. In particular, the weak formulation contains additional boundary terms involving the initial velocity $u_1$, and therefore suitable sign conditions on $u_1$ naturally appear in the nonexistence results for nonnegative solutions.
	\end{remark}
	
	By choosing $\theta_1=2$ and $\theta_2=4$, we immediately obtain the following corollary.
	\begin{corollary}\label{s2:C5}		
		Let $\mathcal{G}$ satisfy \eqref{s3:1} and $\sigma>1$. Assume either $h\geq g(x)>0$ with growth condition
		\begin{equation*}
		\sum_{x\in\mathcal{V} \cap B_{R}}\mu_{\mathcal{V}}(x)g^{-\frac{1}{\sigma-1}}(x)\leq CR^{\frac{\sigma+1}{2(\sigma-1)}},\quad\sum_{e\in \mathcal{E}\cap B_R}\int_{0}^{l_e}g_e^{-\frac{1}{\sigma-1}}(x)dx\leq CR^{\frac{\sigma+1}{2(\sigma-1)}},\quad\forall\ R\geq R_0,
	\end{equation*} 	
		or $h\equiv1$ with volume growth
			\begin{equation*}
			\mu_{\mathcal{V}}(B_R)\leq CR^{\frac{\sigma+1}{2(\sigma-1)}},\quad\mu_{\mathcal{E}}(B_R)\leq CR^{\frac{\sigma+1}{2(\sigma-1)}},\quad\forall\ R\geq R_0.
		\end{equation*} 	
		If $u$ is a nonnegative very weak solution to \eqref{s1:2} satisfying 
		\begin{equation*}
			\liminf_{R\to\infty}\left\{\int_{B_{R-j}}u_1^+(x)d\mu_{\mathcal{G}}+\int_{B_{\sqrt{2}R+j}}u_1^-(x)d\mu_{\mathcal{G}}\right\}\geq0,
		\end{equation*}
		then $u$ is trivial.
	\end{corollary}
	
	Finally, we conclude by stating the last main result of this paper, in which no sign restriction is made on the solution to the hyperbolic problem \eqref{s1:2}.
	
			\begin{theorem}\label{s2:T2}
		Let $\mathcal{G}$ be a metric graph satisfying \eqref{s3:1} and $\sigma>1$. Assume that $h: \mathcal{G}\times[0,\infty)\ra\mathbb{R}$ is a positive function such that for every $R\geq R_0$,
		\begin{equation}\label{s3:2}
			\int_{R^{\frac{1}{2}}}^{2R^{\frac{1}{2}}} \le\{\sum_{x\in \mathcal{V}\cap B_{R}}\mu_{\mathcal{V}}(x) h^{-\frac{1}{\sigma-1}}(x,t) e^{-\alpha d(x,x_0)} \ri\}dt\leq CR^{\frac{\sigma}{\sigma-1}},
		\end{equation}
		\begin{equation}\label{s3:3}
		\int_{R^{\frac{1}{2}}}^{2R^{\frac{1}{2}}}\le\{\sum_{e\in \mathcal{E}\cap B_{R}}\int_{0}^{l_e}	h_e^{-\frac{1}{\sigma-1}}(x,t)e^{-\alpha d(x,x_0)}dx \ri\}dt\leq CR^{\frac{\sigma}{\sigma-1}},
		\end{equation}
		and
		\begin{equation}\label{s3:4}
			\int_0^{2R^{\frac{1}{2}}} \le\{\sum_{x\in\mathcal{V}\cap B_{R}^c}\mu_{\mathcal{V}}(x) h^{-\frac{1}{\sigma-1}}(x,t) e^{-\alpha d(x,x_0)}\ri\}dt\leq CR^{\frac{\sigma}{\sigma-1}},
		\end{equation}
		\begin{equation}\label{s3:5}
			\int_0^{2R^{\frac{1}{2}}}\le\{\sum_{e\in\mathcal{E}\cap B_{R}^c}\int_{0}^{l_e}	h_e^{-\frac{1}{\sigma-1}}(x,t)e^{-\alpha d(x,x_0)}dx \ri\}dt\leq CR^{\frac{\sigma}{\sigma-1}}.
		\end{equation}
	Suppose that $u:\mathcal{G}\times[0,\infty)\to\mathbb{R}$ is a very weak solution to \eqref{s1:2},
		\begin{equation}\label{s3:6}
		\int_{\mathcal{G}}u_1(x)d\mu_\mathcal{G}\geq0,
	\end{equation}
	 $u_0,u_1\in X_\alpha$, and $u\in L^1_{\text{loc}}([0,\infty); X_\alpha)$, then $u\equiv0$ on $\mathcal{G}\times[0,\infty)$.
	\end{theorem}
	
	\begin{corollary}\label{s2:C6}
		Under the hypotheses of Theorem \ref{s2:T2}, assume either $h\geq g(x)>0$ with growth condition
		\begin{equation*}
			\sum_{x\in\mathcal{V}}\mu_{\mathcal{V}}(x) g^{-\frac{1}{\sigma-1}}(x) e^{-\alpha d(x,x_0)}\leq CR^{\frac{\sigma+1}{2(\sigma-1)}},\quad \sum_{e\in \mathcal{E}}\int_{0}^{l_e}	g_e^{-\frac{1}{\sigma-1}}(x)e^{-\alpha d(x,x_0)}dx\leq CR^{\frac{\sigma+1}{2(\sigma-1)}},\quad\forall\ R\geq R_0,
		\end{equation*}
		or $h\equiv1$ with annular growth
		\begin{equation*}
			\mu_\mathcal{V}(B_{R+1}\setminus B_R)\leq CR^{\frac{\sigma+1}{2(\sigma-1)}-1},\quad\mu_\mathcal{E}(B_{R+1}\setminus B_R)\leq CR^{\frac{\sigma+1}{2(\sigma-1)}-1},\quad\forall\ R\geq R_0.
		\end{equation*}
		Then any very weak solution $u$ is trivial.
	\end{corollary}
	
	\begin{remark}
		The main results of this section show that suitable weighted space-time volume growth conditions on the potential and on the underlying metric graph prevent the existence of global very weak solutions. In this sense, the geometry of the graph and the behavior of the potential jointly determine the validity of Liouville-type phenomena.
	\end{remark}

\section{Proofs of the main results for the parabolic problem}\label{3}

In this section, we introduce a modified distance function to handle non-differentiability of the standard distance along edges and establish a priori estimates for solutions to \eqref{s1:1}. This will allow us to prove Theorems \ref{s2:T3} and \ref{s2:T4}.

\subsection{Modified distance functions}

Fix a vertex $x_0\in\mathcal{V}$. In order to construct suitable test functions, we first analyze the behavior of the distance function on the metric graph. In what follows, we identify points on edge $e$ with elements of interval $I_e$, using these two sets of points interchangeably. For any edge $e\in\mathcal{E}$, as a point $x$ travels along $e$ from $i(e)$ to $j(e)$, the distance to $x_0$ may behave in one of the following ways:

\begin{enumerate}[(a)]
	\item $d(x,x_0) = d\big(i(e),x_0\big) + x$;
	\item $d(x,x_0) = d\big(j(e),x_0\big) + l_e - x$;
	\item there exists a point $q_e\in e$ such that
	\[
	d\big(i(e),x_0\big) + q_e
	= d\big(j(e),x_0\big) + l_e - q_e,
	\]
	and consequently,
	\[
	d(x,x_0)=
	\begin{cases}
		d\big(i(e),x_0\big) + x, & x\in[i(e),q_e],\\
		d\big(j(e),x_0\big) + l_e - x, & x\in[q_e,j(e)].
	\end{cases}
	\]
\end{enumerate}
In Case $(c)$, the distance function fails to be differentiable at the interior point $q_e$, since the left-hand derivative at these points equals $1$ while the right-hand derivative equals $-1$. For a concrete illustration, we refer the reader to Figure $1$ in our previous work \cite{L-L-Z}. In the mollification procedure introduced in Subsection 4.1 of \cite{L-L-Z}, the singularity was shifted to the midpoint of the edge, but it was not completely removed, and its location still had to be analyzed in the proof. To overcome this difficulty, we introduce here a more regular distance function.

More precisely, in Case $(c)$, we redefine the distance function by
$$d(x,x_0):=d(i(e),x_0)+\f{\le(d(j(e),x_0)-d(i(e),x_0)\ri)}{l_e}x, \quad x\in[0,l_e].$$
The distance function remains unchanged in Cases $(a)$ and $(b)$. With this modification, the distance function becomes monotone (or constant) along each edge, depending solely on whether $d(i(e),x_0)$ is larger, smaller, or equal to $d(j(e),x_0)$.

Although the singularity has now been removed, the integration by parts formula on edges still produces boundary terms involving outward derivatives at the vertices. 
To eliminate these contributions, we further smooth the distance function near the endpoints of each edge, ensuring that the one-sided derivatives vanish there.

Let $\eta:[0,1]\to[0,1]$ be a $C^2$ transition function such that
\begin{equation*}
	\begin{cases}
		\eta(0)=0,\ \eta(1)=1,\ \eta^\prime_+(0)=\eta^\prime_-(1)=0;\\
		\eta^\prime(x)\geq0, \text{ for all}\ x\in (0, 1);\\
		\text{there exists}\ C > 0\ \text{such that}\ |\eta^\prime(x)|\leq C\ \text{and}\  |\eta^{\prime\prime}(x)|\leq C, \text{ for all}\ x\in (0, 1).
	\end{cases}
\end{equation*}
We then define the modified distance function $\tilde{d}:\mathcal{G}\to\mathbb{R}_+\cup\{0\}$ as
\begin{equation}\label{s4:1}
	\tilde{d}=\bigoplus_{e\in\mathcal{E}}\tilde{d}_e,
\end{equation}
where, for every edge $e\in\mathcal E$:
\begin{enumerate}[(i)]
	\item if $d_e$ satisfies Case $(a)$, 
	\begin{equation}\label{s4:2}
		\tilde{d}_e(x,x_0)=d(i(e),x_0)+l_e\eta\le(\frac{x}{l_e}\ri), \quad x\in[0,l_e];
	\end{equation}
	\item if $d_e$ satisfies Case $(b)$, 
	\begin{equation}\label{s4:4}
		\tilde{d}_e(x,x_0)=d(j(e),x_0)+l_e-l_e\eta\le(\frac{x}{l_e}\ri), \quad x\in[0,l_e];
	\end{equation}
	\item  if $d_e$ satisfies Case $(c)$, 
	\begin{equation}\label{s4:3}
		\tilde{d}_e(x,x_0)=d(i(e),x_0)+\le(d(j(e),x_0)-d(i(e),x_0)\ri)\eta\le(\frac{x}{l_e}\ri), \quad x\in[0,l_e].
	\end{equation}
\end{enumerate}
By construction, $\tilde d$ coincides with the original distance function at every vertex, while its one-sided derivatives vanish (as opposed to the original $\pm1$) at the endpoints of each edge. This property will be crucial in the integration by parts arguments developed later. Moreover,
\begin{equation}\label{s4:16}
	|d(x,x_0)-\tilde{d}(x,x_0)|\leq j,\quad\forall\ x\in\mathcal{G}.
\end{equation}

\begin{figure}[htbp]
	\centering
	\begin{tikzpicture}
		\begin{scope}[xshift=0cm, local bounding box=base]
			\def\lE{3.0}     
			\def\dU{1}     
			\def\dV{2}     
			
			\def\qE{2} 
			\def\maxD{3}
			
			\draw[->, thick] (-0.5,0) -- (\lE + 0.5,0) node[below] {$x$};
			\draw[->, thick] (0,-0.5) -- (0,\maxD + 0.5) node[left] {$d(x, x_0)$};
			
			\coordinate (startPoint) at (0, \dU);
			\coordinate (cutPoint) at (\qE, \maxD);
			\coordinate (endPoint) at (\lE, \dV);
			
			\draw[dashed,very thick] (startPoint) -- (cutPoint) -- (endPoint);
			\draw[very thick] (startPoint) -- (endPoint);
			
			\node[below left] at (0,0) {$0$};
			\node[below] at (\lE,0) {$l_e$};
			\node[left] at (0,\dU) {$d(i(e),x_0)$};
			\draw[dotted] (endPoint) -- (0,\dV) node[left] {$d(j(e),x_0)$};
			\draw[dotted] (endPoint) -- (\lE,0);
			
			\draw[dotted] (cutPoint) -- (\qE, 0) node[below] {$q_e$};
			\draw[dotted] (cutPoint) -- (0, \maxD) node[left] {$d(q_e,x_0)$};
			
			
			\node[below] at (\lE/2.0, -0.4) {(a)};
		\end{scope}
		
		\begin{scope}[xshift=5.5cm, local bounding box=smooth]
			\def\lE{2.5}     
			\def\dU{0.8}     
			\def\dV{2.8}     
			
			
			\def\maxD{3}
			
			\draw[->, thick] (-0.5,0) -- (\lE + 0.5,0) node[below] {$x$};
			\draw[->, thick] (0,-0.5) -- (0,\maxD + 0.5) node[left] {$\tilde d(x, x_0)$};
			
			\coordinate (startPoint) at (0, \dU);
			\coordinate (endPoint) at (\lE, \dV);
			
			\draw[dashed,very thick] (startPoint) -- (endPoint);
			\draw[very thick] (startPoint) to[out=0, in=180] (endPoint);
			\node[below left] at (0,0) {$0$};
			\node[below] at (\lE,0) {$l_e$};
			\node[left] at (0,\dU) {$d(i(e),x_0)$};
			\draw[dotted] (endPoint) -- (0,\dV) node[left] {$d(j(e),x_0)$};
			\draw[dotted] (endPoint) -- (\lE,0);
			
			
			\node[below] at (\lE/2.0, -0.4) {(b)};
		\end{scope}
		
		\begin{scope}[xshift=10.5cm, local bounding box=pseudo]
			\def\lE{3.0}     
			\def\dU{1.5}     
			\def\dV{1.5}     
			
			\def\qE{1.5} 
			\def\maxD{3}
			
			\draw[->, thick] (-0.5,0) -- (\lE + 0.5,0) node[below] {$x$};
			\draw[->, thick] (0,-0.5) -- (0,\maxD + 0.5) node[left] {$\tilde d(x, x_0)$};
			
			\coordinate (startPoint) at (0, \dU);
			\coordinate (cutPoint) at (\qE, \maxD);
			\coordinate (endPoint) at (\lE, \dV);
			
			\draw[dashed,very thick] (startPoint) -- (cutPoint) -- (endPoint);
			\draw[very thick] (startPoint) -- (endPoint);
			
			\node[below left] at (0,0) {$0$};
			\node[below] at (\lE,0) {$l_e$};
			\node[left] at (0,\dU) {$d(i(e),x_0)$};
			\draw[dotted] (endPoint) -- (\lE,0);
			
			\draw[dotted] (cutPoint) -- (\qE, 0) node[below] {$q_e$};
			\draw[dotted] (cutPoint) -- (0, \maxD) node[left] {$d(q_e,x_0)$};
			
			
			\node[below] at (\lE/2.0, -0.4) {(c)};
		\end{scope}
		
	\end{tikzpicture}
	\caption{We first redefine the distance function to avoid the singularity $q_e$ shown in (a), and then smooth it as illustrated in (b). For edges where $d(i(e),x_0)=d(j(e),x_0)$, the modified distance function becomes degenerate, resulting in a pseudo metric as shown in (c).}
	\label{fig:distanceFunctions}
\end{figure}
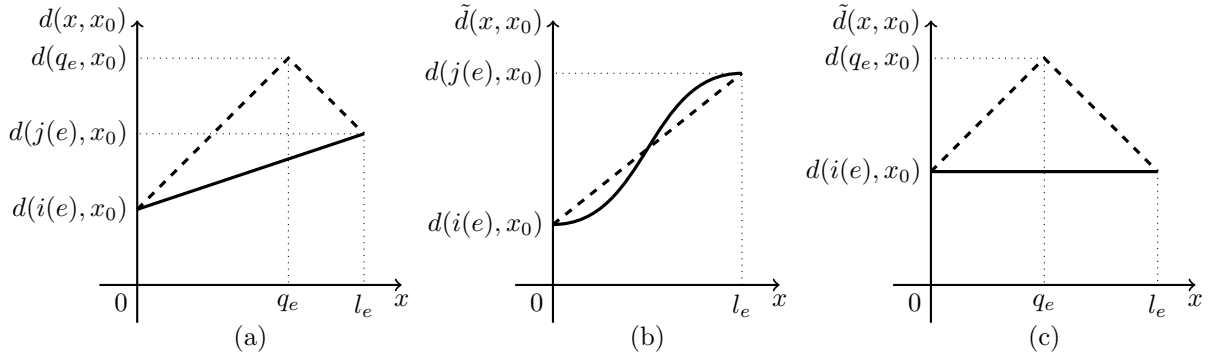

In fact, this modification naturally induces a global pseudo-metric $\rho_{x_0}(x,y):=\inf\limits_{\gamma:x\to y}\int_\gamma|\tilde d'(\gamma(s),x_0)|ds$ on $\mathcal G$, where the infimum is taken over all paths $\gamma$ connecting $x$ and $y$. Specifically,
\begin{itemize}
	\item $\rho_{x_0}:\mathcal{G}\times\mathcal{G}\to [0,\infty)$ is a symmetric map with a zero diagonal;
	\item the triangle inequality holds: for all $x,y,z\in \mathcal G$ and any paths $\gamma_1$ connecting $x$ to $z$ and $\gamma_2$ connecting $z$ to $y$, concatenating them yields a path from $x$ to $y$, which gives
	$$\rho_{x_0}(x,y)\leq \int_{\gamma_1}|\tilde d'(\gamma_1(s),x_0)|ds+\int_{\gamma_2}|\tilde d'(\gamma_2(s),x_0)|ds.$$
	Taking the infimum of the right-hand side over all such paths $\gamma_1$ and $\gamma_2$ yields $\rho_{x_0}(x,y) \leq \rho_{x_0}(x,z) + \rho_{x_0}(z,y)$;
	\item there may exist distinct points $x,y \in \mathcal{G}$ such that $\rho_{x_0}(x,y) = 0$ (for instance, any two points on the degenerate segment corresponding to Case (c) in Figure~\ref{fig:distanceFunctions}).
\end{itemize}
While our subsequent analysis primarily relies on the single-variable function $\tilde d(\cdot, x_0)$ to construct test functions, it is intrinsically rooted in this pseudo-metric structure, as it can be naturally recovered via $\tilde{d}(x,x_0)=\rho_{x_0}(x,x_0)$.

We next summarize the basic properties of the modified distance function.

\begin{proposition}\label{s4:P1}
	For every edge $e\in\mathcal{E}$, the following statements hold:
	\begin{enumerate}[(i)]
		\item if $x=i(e)$ or $x=j(e)$, then $\tilde{d}_e^\prime(x,x_0)=0$;
		\item there exists a constant $C>0$ such that for any $x\in(0,l_e)$,
		\begin{equation}\label{s4:5}
			\le|\tilde{d}_e^\prime(x,x_0)\ri|\leq C,\quad \le|\tilde{d}_e^{\prime\prime}(x,x_0)\ri|\leq C.
		\end{equation}
	\end{enumerate}
\end{proposition}
\begin{proof}
	(i) For any $x\in\mathcal{V}$ and $e\ni x$, we first calculate the right derivative of $\tilde{d}_e$ at the endpoint $x=i(e)$. By $\eta(0)=0$ and $\eta'_+(0)=0$,
	\begin{equation*}
		\tilde{d}_e'(x^+,x_0)
		= \lim_{x\to 0^+} \frac{\tilde{d}_e(x,x_0)-\tilde{d}_e(0,x_0)}{x}=
		\begin{cases}
			\lim_{x\to 0^+}\frac{l_e\eta\le(\frac{x}{l_e}\ri)}{x}=0,&\text{if}\ \eqref{s4:2}\ \text{holds},\\[1.5ex]
			\lim_{x\to 0^+}-\frac{l_e\eta\le(\frac{x}{l_e}\ri)}{x}=0,&\text{if}\ \eqref{s4:4}\ \text{holds},\\[1.5ex]
			\lim_{x\to 0^+}\frac{\le(d(j(e),x_0)-d(i(e),x_0)\ri)\eta\le(\frac{x}{l_e}\ri)}{x}=0,&\text{if}\ \eqref{s4:3}\ \text{holds}.
		\end{cases}
	\end{equation*}
	The reasoning is symmetric at the endpoint $x=j(e)$, so we omit the details here. Thus, the modified distance function has a zero one-sided derivative at both endpoints.
	
	(ii) For any $e\in\mathcal{E}$, we obtain from \eqref{s4:2} and \eqref{s4:4} that 
	$$\le|\tilde{d}_e^\prime(x,x_0)\ri|
	=\le|\pm\eta^\prime\le(\frac{x}{l_e}\ri)\ri|\leq C,\quad	\le|\tilde{d}_e^{\prime\prime}(x,x_0)\ri|=\le|\frac{\pm1}{l_e}\eta^{\prime\prime}\le(\frac{x}{l_e}\ri)\ri|\leq \le|\frac{1}{r}\eta^{\prime\prime}\le(\frac{x}{l_e}\ri)\ri|\leq C,\quad\forall\ x\in\le(0,l_e\ri),$$ 
	and from \eqref{s4:3} that for any $x\in(0,l_e)$, 
	$$\le|\tilde{d}_e^\prime(x,x_0)\ri|
	=\le|\frac{d(j(e),x_0)-d(i(e),x_0)}{l_e}\ri|\le|\eta^\prime\le(\frac{x}{l_e}\ri)\ri|\leq \le|\pm1\ri|\le|\eta^\prime\le(\frac{x}{l_e}\ri)\ri|\leq C,$$
	$$\le|\tilde{d}_e^{\prime\prime}(x,x_0)\ri|
	=\le|\frac{d(j(e),x_0)-d(i(e),x_0)}{l_e}\ri|\le|\frac{1}{l_e}\eta^{\prime\prime}\le(\frac{x}{l_e}\ri)\ri|\leq \le|\pm1\ri|\le|\frac{1}{r}\eta^{\prime\prime}\le(\frac{x}{l_e}\ri)\ri|\leq C,$$
	where we have used $\inf_{e\in\mathcal{E}}l_e>r$ and $|\eta^\prime|\leq C$, $|\eta^{\prime\prime}|\leq C$ on $(0,1)$. Hence, this completes the proof.
\end{proof}

\subsection{Nonexistence for nonnegative global solutions}

Let $\phi\in C^2([0, \infty))$ be a cut-off function on $[0, \infty)$, which satisfies the following conditions
\begin{equation}\label{s4:6}
	\begin{cases}
		\phi\equiv0\ \text{on}\ [2,\infty)\ \text{and}\ \phi\equiv1\ \text{on}\ [0,1]; \\
		\phi^\prime(x)\leq0,\ \text{for every}\ x\in [0, \infty); \\
		\text{there exists}\ C > 0\ \text{such that}\ |\phi^\prime(x)|\leq C\ \text{and}\  |\phi^{\prime\prime}(x)|\leq C, \text{ for all}\ x\in [0, \infty).
	\end{cases}
\end{equation}
Given $R\geq R_0=\max\{1,2j\}$, we define a test function
\begin{equation}\label{s4:7}
	\tau(x,t):=\phi\le(\frac{\tilde{d}(x,x_0)^{\theta_1}+t^{\theta_2}}{R^{\theta_1}}\ri),\quad \forall\ (x,t)\in\mathcal{G}\times[0,\infty),
\end{equation}
where $\tilde{d}$ is a modified distance function specified in \eqref{s4:1}, and $\theta_1,\theta_2>0$ are constants. Let $\theta_1\geq2$, $\theta_2\geq 1$ be fixed constants satisfying $\theta_1\geq\theta_2$. Define
\begin{equation}\label{s4:27}
	\tilde{E}_R=\left\{(x,t)\in\mathcal{G}\times[0,\infty):R^{\theta_1}\leq \tilde{d}(x,x_0)^{\theta_1}+t^{\theta_2}\leq 2R^{\theta_1}\right\},
\end{equation}
\begin{equation}\label{s4:28}
	M_R=\left\{(x,t)\in\mathcal{G}\times[0,\infty):(R/2)^{\theta_1}\leq d(x,x_0)^{\theta_1}+t^{\theta_2}\leq (4R)^{\theta_1}\right\}.
\end{equation} 

\begin{lemma}\label{s4:L1}
	Let $R\geq 4j$ and $\theta_1\geq2$, $\theta_2\geq1$. Then $\tilde{E}_R\subset M_R$.
\end{lemma}
\begin{proof}
	Let $(x,t)\in \tilde{E}_R$. Then
	$$
	R^{\theta_1}\leq \tilde{d}(x,x_0)^{\theta_1}+t^{\theta_2}\leq 2R^{\theta_1}.
	$$
	We first prove the upper bound in the definition of $M_R$. By \eqref{s4:16},
	$$
	d(x,x_0)\leq \tilde{d}(x,x_0)+j.
	$$
	Hence,
	\begin{align*}
		d(x,x_0)^{\theta_1}+t^{\theta_2}
		&\leq (\tilde{d}(x,x_0)+j)^{\theta_1}+t^{\theta_2}\\
		&\leq 2^{\theta_1-1}\bigl(\tilde{d}(x,x_0)^{\theta_1}+j^{\theta_1}\bigr)+t^{\theta_2}\\
		&\leq 2^{\theta_1-1}\cdot 2R^{\theta_1}
		+2^{\theta_1-1}j^{\theta_1}\\
		&\leq\left(2^{\theta_1}+\frac{1}{2^{\theta_1+1}}\right)R^{\theta_1}\\
		&\leq (4R)^{\theta_1},
	\end{align*}
	where we used the fact that $R\geq4j$.

Next we prove the lower bound. We distinguish two cases.

\medskip
\noindent
\textit{Case 1: $\tilde{d}(x,x_0)\geq 2j$.}

In this case,
$$
d(x,x_0)\geq \tilde{d}(x,x_0)-j\geq \frac12\tilde{d}(x,x_0),
$$
which yields
$$
d(x,x_0)^{\theta_1}
\geq \frac{1}{2^{\theta_1}}\tilde{d}(x,x_0)^{\theta_1}.
$$
Therefore,
\begin{align*}
	d(x,x_0)^{\theta_1}+t^{\theta_2}
	&\geq \frac{1}{2^{\theta_1}}
	\bigl(\tilde{d}(x,x_0)^{\theta_1}+t^{\theta_2}\bigr)\geq \left(\frac{R}{2}\right)^{\theta_1}.
\end{align*}

\medskip
\noindent
\textit{Case 2: $\tilde{d}(x,x_0)< 2j$.}

Since $(x,t)\in \tilde{E}_R$, we have
$$
t^{\theta_2}
\geq R^{\theta_1}-\tilde{d}(x,x_0)^{\theta_1}
> R^{\theta_1}-(2j)^{\theta_1}.
$$
Using $R\geq4j$, we obtain
$$
(2j)^{\theta_1}\leq \left(\frac{R}{2}\right)^{\theta_1}.
$$
Hence,
\begin{align*}
	d(x,x_0)^{\theta_1}+t^{\theta_2}
	\geq t^{\theta_2}\geq R^{\theta_1}-\left(\frac{R}{2}\right)^{\theta_1}\geq\frac{R^{\theta_1}}{2^{\theta_1-1}}-\le(\frac{R}{2}\ri)^{\theta_1}\geq\le(\frac{R}{2}\ri)^{\theta_1},
\end{align*}
where we used $\theta_1\geq2$. 

Combining the above estimates, we conclude that
$$
\left(\frac{R}{2}\right)^{\theta_1}
\leq d(x,x_0)^{\theta_1}+t^{\theta_2}
\leq (4R)^{\theta_1},
$$
which implies $(x,t)\in M_R$. Therefore,
$$
\tilde{E}_R\subset M_R.
$$
This completes the proof.
\end{proof}

\begin{remark}
	The stronger condition $R\geq4j$ is only needed for the inclusion $\tilde{E}_R\subset M_R$ in the nonnegative case. All other estimates in the paper remain valid under the weaker condition
	$R\geq 2j$.
\end{remark}

We next show the following upper bounds.
\begin{lemma}\label{s5:L1}
	For any $R\geq R_0$, there exists a constant $C>0$ such that 
	
	(i)	For each edge $e\in\mathcal{E}$,
	\begin{equation}\label{s4:14}
		|\tau_e^{\prime\prime}(x,t)|\leq \frac{C}{R}\textbf{1}_{\tilde{E}_R}(x,t), \quad\forall \ x\in(0,l_e),\ t\in[0,\infty).
	\end{equation}
	
	(ii) For every vertex $x\in\mathcal{V}$,
	\begin{equation}\label{s4:13}
		-\Delta_\mathcal{V} \tau(x,t)\leq \frac{C}{R}\textbf{1}_{M_R}(x,t), \quad\forall \ t\in[0,\infty).
	\end{equation}
	
	(iii) For any $x\in\mathcal{G}$ and $t\in[0,\infty)$,
	\begin{equation}\label{s5:2}
		|\tau_{t}(x,t)|\leq \frac{C}{R}\textbf{1}_{\tilde{E}_R}(x,t).
	\end{equation}
\end{lemma}
\begin{proof}
	We prove only (i) and (iii). For (ii), we refer to [\cite{M-P-S3}, Section 3], specifically the case $\alpha=0$. The proof carries over verbatim since the modified distance function $\tilde d$ satisfies the same regularity and bounded increment properties.
	
	(i) Using the chain rule, for any $e\in\mathcal{E}$, $x\in(0,l_e)$ and $t\in[0,\infty)$, we obtain from \eqref{s4:5} that 
	\begin{align*}
		\le|\tau_e^{\prime\prime}(x,t)\ri|&\leq\le|\phi_e^{\prime\prime}\le(\frac{\tilde{d}(x,x_0)^{\theta_1}+t^{\theta_2}}{R^{\theta_1}}\ri)\frac{\theta_1^2\tilde{d}_e(x,x_0)^{2\theta_1-2}}{R^{2\theta_1}}(\tilde{d}_e^\prime(x,x_0))^2\ri|\\
		&\quad+\le|\phi_e^\prime\le(\frac{\tilde{d}(x,x_0)^{\theta_1}+t^{\theta_2}}{R^{\theta_1}}\ri)\frac{\theta_1(\theta_1-1)\tilde{d}_e(x,x_0)^{\theta_1-2}}{R^{\theta_1}}(\tilde{d}_e^\prime(x,x_0))^2\ri|\\
		&\quad+\le|\phi_e^\prime\le(\frac{\tilde{d}(x,x_0)^{\theta_1}+t^{\theta_2}}{R^{\theta_1}}\ri)\frac{\theta_1\tilde{d}_e(x,x_0)^{\theta_1-1}}{R^{\theta_1}}\tilde{d}_e^{\prime\prime}(x,x_0)\ri|.
	\end{align*}
	Since the support of $\phi^\prime$ and $\phi^{\prime\prime}$ is contained in $\tilde E_R$, and since $\phi^\prime$ and $\phi^{\prime\prime}$ are bounded, we deduce that
	$$\le|\tau_e^{\prime\prime}(x,t)\ri|\leq \frac{C}{R^{2\theta_1}}\tilde{d}_e(x,x_0)^{2\theta_1-2}\textbf{1}_{\tilde{E}_R}(x,t)+\frac{C}{R^{\theta_1}}\tilde{d}_e(x,x_0)^{\theta_1-2}\textbf{1}_{\tilde{E}_R}(x,t)+\frac{C}{R^{\theta_1}}\tilde{d}_e(x,x_0)^{\theta_1-1}\textbf{1}_{\tilde{E}_R}(x,t).$$
	
	 For  $(x,t)\in \tilde{E}_R$, we have $\tilde{d}_e(x,x_0)\leq 2^{\frac{1}{\theta_1}}R$. Using  $R\geq1$ and $\theta_1\geq2$, it follows that
	\begin{equation*}
		\le|\tau_e^{\prime\prime}(x,t)\ri|\leq \frac{C}{R^{2\theta_1}}R^{2\theta_1-2}\textbf{1}_{\tilde{E}_R}(x,t)+ \frac{C}{R^{\theta_1}}R^{\theta_1-2}\textbf{1}_{\tilde{E}_R}(x,t)+ \frac{C}{R^{\theta_1}}R^{\theta_1-1}\textbf{1}_{\tilde{E}_R}(x,t)\leq \frac{C}{R}\textbf{1}_{\tilde{E}_R}(x,t).
	\end{equation*}
	
	(iii)  Observe that $\theta_2\geq1$, and that if $(x,t)\in \tilde{E}_R$, there holds $t\leq CR^{\frac{\theta_1}{\theta_2}}$. Hence, for any $x\in\mathcal{G}$ and $t\in[0,\infty)$,
	\begin{equation}\label{s4:29}
		|\tau_{t}(x,t)|=\le|\phi^\prime\le(\frac{\tilde{d}(x,x_0)^{\theta_1}+t^{\theta_2}}{R^{\theta_1}}\ri)\ri|\frac{\theta_2t^{\theta_2-1}}{R^{\theta_1}}\leq\frac{C}{R^{\theta_1}}R^{\theta_1-\frac{\theta_1}{\theta_2}}\textbf{1}_{\tilde{E}_R}(x,t)\leq \frac{C}{R^{\frac{\theta_1}{\theta_2}}}\textbf{1}_{\tilde{E}_R}(x,t).
	\end{equation}
	Since $\theta_1\geq\theta_2\geq1$, so
	$$	|\tau_{t}(x,t)|\leq \frac{C}{R}\textbf{1}_{\tilde{E}_R}(x,t).$$
	This completes the proof.
\end{proof}

Denote
$$\tilde{B}_R:=\{x\in\mathcal{G}:\tilde{d}(x,x_0)<R\}.$$
Since the support of the test function $\tau$ is contained in
$$\overline{\tilde{B}_{2^{\frac{1}{\theta_1}}R}}\times[0,2^{\frac{1}{\theta_2}}R^{\frac{\theta_1}{\theta_2}}],$$
all sums over vertices $x\in\mathcal{V}$ and edges $e\in\mathcal{E}$ are finite, and all integrals in the time variable are effectively taken over bounded intervals. For any  $t\in(0,2^{\frac{1}{\theta_2}}R^{\frac{\theta_1}{\theta_2}})$, let $\text{supp}(\tau(\cdot,t)) = \{ x \in \mathcal{G} : \tilde{d}(x, x_0)^{\theta_1} \leq 2R^{\theta_1}-t^{\theta_2} \}$ denote the support of $\tau(\cdot,t)$. Since $\tilde d$ is affine on each edge, the intersection $\operatorname{supp}(\tau(\cdot,t))\cap [0,l_e]$ is always an interval. Moreover, since $\operatorname{supp}(\tau(\cdot,t))$ intersects only finitely many edges, we denote these edges by $\mathcal{E}_{\tau(\cdot,t)} = \{ \mathcal{G}_{e_1}, \mathcal{G}_{e_2}, \dots, \mathcal{G}_{e_N} \}$. The intersection of $\text{supp}(\tau(\cdot,t))$ with each such edge $e$ is a line segment $[a_e^{(t)}, b_e^{(t)}] \subset [0, l_e]$, where $0 \leq a_e^{(t)}< b_e^{(t)} \leq l_e$. More precisely, the following there situations may occur:
\begin{itemize}
	\item the entire edge is contained in $\text{supp}(\tau(\cdot,t))$, i.e., $[a_e^{(t)}, b_e^{(t)}] = [0, l_e]$;
	\item only a terminal segment of the edge is contained in $\text{supp}(\tau(\cdot,t))$, i.e., $[a_e^{(t)}, b_e^{(t)}] = [0, b_e^{(t)}]$;
	\item only a terminal segment of the edge is contained in $\text{supp}(\tau(\cdot,t))$, i.e., $[a_e^{(t)}, b_e^{(t)}] = [a_e^{(t)}, l_e]$.
\end{itemize}
For simplicity, we use $[a_e^{(t)}, b_e^{(t)}]$ to denote all such intervals without specifying their exact form. We also use $\mathcal{E}_{\tau(\cdot,t)}^1$, $\mathcal{E}_{\tau(\cdot,t)}^2$ and $\mathcal{E}_{\tau(\cdot,t)}^3$ to denote these edges, respectively. 

On metric graphs, the edge-wise integration by parts formula for $\tau$ (with the standard distance $d$) generally produces boundary terms involving outward normal derivatives at vertices, unlike combinatorial graphs, which only consider vertex-based formulas.  The use of the modified distance $\tilde d$ ensures that the resulting test function has vanishing vertex contributions in the integration by parts formula below.

\begin{lemma}\label{s5:L5}
	Let $s>\max\{2,\sigma/(\sigma-1)\}$ with $\sigma>1$, and let $R\geq R_0$. Suppose that $u$ is a very weak nonnegative solution to \eqref{s1:1}, and $\tau:\mathcal{G}\times[0,\infty)\to\mathbb{R}$ is defined in \eqref{s4:7}. Then 
	\begin{equation}\label{s4:9}
			\int_0^\infty\mathcal{L}_{\Delta_{\mathcal{G}}u}(\tau^s)dt=\int_0^\infty\mathcal{L}_{\Delta_{\mathcal{G}}\tau^s}(u)dt
	\end{equation}
	where $\mathcal{L}_{\Delta_{\mathcal{G}}\cdot}$ is defined in \eqref{s2:1}.
\end{lemma}
\begin{proof}
	Recalling the definitions of $\Delta_\mathcal{G}u$ and $\mathcal L_{\Delta_{\mathcal G}u}$
	in \eqref{s2:14} and \eqref{s2:1}, we have
	$$\Delta_{\mathcal G}u=\Delta_{\mathcal V}ud\mu_{\mathcal{V}}+\Delta_{\mathcal E}ud\mu_{\mathcal{E}},$$
	and
	$$	\int_0^\infty\mathcal{L}_{\Delta_{\mathcal{G}}u}(\tau^s(x,t))dt=\int_{0}^\infty \le\{\sum_{x\in\mathcal{V}}\mu_{\mathcal{V}}(x)\tau^s(x,t)\Delta_{\mathcal{V}} u(x,t)\ri\}dt+	\int_0^\infty\le\{\sum_{e\in\mathcal{E}}\int_{0}^{l_e}u_e^{\prime\prime}(x,t)\tau_e^s(x,t)dx \ri\}dt,$$
	it suffices to prove the identity separately for the vertex part and the edge part.
	
	First, since $\tau(\cdot,t)$ has finite support on $\mathcal V$ for every fixed $t$, all vertex sums below are finite. By the symmetry of the vertex Laplacian $\Delta_{\mathcal V}$ with respect to the measure $\mu_{\mathcal V}$, we have
	\begin{equation*}
		\int_{0}^\infty \le\{\sum_{x\in\mathcal{V}}\mu_{\mathcal{V}}(x)\tau^s(x,t)\Delta_{\mathcal{V}} u(x,t)\ri\}dt=\int_{0}^\infty \le\{\sum_{x\in\mathcal{V}}\mu_{\mathcal{V}}(x)u(x,t)\Delta_{\mathcal{V}} \tau^s(x,t)\ri\}dt.
	\end{equation*}
	The proof is standard and therefore omitted.
	
	Next we consider the edge part.  Fix $t\in(0,2^{\frac{1}{\theta_2}}R^{\frac{\theta_1}{\theta_2}})$.  Let
	\begin{equation*}
		\mathcal{V}_1^{(t)}=\{x\in\mathcal{V}:d(x, x_0)^{\theta_1} \leq 2R^{\theta_1}-t^{\theta_2}\},\quad\Gamma^{(t)}=\{x\in\mathcal{G}\setminus\mathcal V:\tilde{d}(x,x_0)^{\theta_1}=2R^{\theta_1}-t^{\theta_2}\}.
	\end{equation*}
	Observe that $\Gamma^{(t)}$ consists of the cut points generated by the boundary of $\operatorname{supp}\tau(\cdot,t)$, which may lie in the interior of edges. 
	
	We first show that
	\begin{equation}\label{s4:new1}
		[\mathcal{K}(\tau)](x,t)\equiv0,\quad\forall\ x\in  \mathcal{V}_1^{(t)}.
	\end{equation}
	In fact, let $x\in\mathcal{V}_1^{(t)}$, consider any edge $e\ni x$ with $e\in \mathcal{E}_{\tau(\cdot,t)}$. By part $(i)$ of Proposition \ref{s4:P1}, $\tilde{d}^\prime(x,x_0)=0$. It thus follows that
	$$\tau^\prime_e(x,t)=\phi_e^\prime\le(\frac{\tilde{d}(x,x_0)^{\theta_1}+t^{\theta_2}}{R^{\theta_1}}\ri)\frac{\theta_1\tilde{d}(x,x_0)^{\theta_1-1}}{R^{\theta_1}}\tilde{d}^{\prime}(x,x_0)=0,\quad\forall x\in \mathcal{V}_1^{(t)}.$$
	Therefore
	\[
	\frac{d\tau_e(x,t)}{dn}=0
	\]
	for every edge incident to \(x\). Since \(x\) has finite degree, it follows from \eqref{s2:3} and \eqref{s2:4} that
	\[
	[\mathcal K(\tau)](x,t)
	=
	\sum_{e\ni x}
	\frac{d\tau_e(x,t)}{dn}
	=
	0.
	\]
	This proves \eqref{s4:new1}.
	
	Applying integration by parts twice, we obtain
	\begin{align}\label{s4:10}
		\nonumber	&\sum_{e\in\mathcal{E}}\int_{0}^{l_e}u_e^{\prime\prime}(x,t)\tau_e^s(x,t)dx\\
		\nonumber&=\sum_{e\in\mathcal{E}_{\tau(\cdot,t)}}\int_{a_e^{(t)}}^{b_e^{(t)}}u_e^{\prime\prime}(x,t)\tau_e^s(x,t)dx\\
		\nonumber&= -\sum_{e\in{\mathcal{E}_{\tau(\cdot,t)}}}\int_{a_e^{(t)}}^{b_e^{(t)}} u_e^{\prime}(x,t)(\tau_e^s)^{\prime}(x,t) dx+\sum_{e\in{\mathcal{E}_{\tau(\cdot,t)}}}\le(u_e^\prime(b_e^{(t)},t)\tau_e^s(b_e^{(t)},t)-u_e^\prime(a_e^{(t)},t)\tau^s_e(a_e^{(t)},t)\ri)\\
		\nonumber&=\sum_{e\in{\mathcal{E}_{\tau(\cdot,t)}}}\int_{a_e^{(t)}}^{b_e^{(t)}} u_e(x,t)(\tau_e^s)^{\prime\prime}(x,t) dx\\
		\nonumber&\quad-\sum_{e\in{\mathcal{E}_{\tau(\cdot,t)}}}\le(u_e(b_e^{(t)},t)(\tau_e^s)^\prime(b_e^{(t)},t)-u_e(a_e^{(t)},t)(\tau_e^s)^\prime(a_e^{(t)},t)\ri)\\
		&\quad+\sum_{e\in{\mathcal{E}_{\tau(\cdot,t)}}}\le(u_e^\prime(b_e^{(t)},t)\tau_e^s(b_e^{(t)},t)-u_e^\prime(a_e^{(t)},t)\tau^s_e(a_e^{(t)},t)\ri).
	\end{align}
	Denote the last two terms by $B_1$ and $B_2$, respectively. We first deal with $B_1$. The boundary contributions arise from vertices in $\mathcal V_1^{(t)}$ and cut points in $\Gamma^{(t)}$. If $x\in\Gamma^{(t)}$, then
	$$
    \frac{\widetilde d(x,x_0)^{\theta_1}+t^{\theta_2}}{R^{\theta_1}}=2.
	$$
	Since \(\phi\equiv0\) on \([2,\infty)\), we have $\phi'(2)=0$, and consequently, $\tau_e'(x,t)=0$.
	Hence all cut-point contributions vanish. Noting that each vertex is of finite degree, we may transform the sum over edge endpoints to one over the adjacent edges of each vertex. It follows from \eqref{s4:new1} that
	\begin{align*}
		B_1
		&=\sum_{e\in{\mathcal{E}_{\tau(\cdot,t)}^1}}\le(su_e(l_e,t)\tau_e^{s-1}(l_e,t)\tau_e^\prime(l_e,t)-su_e(0,t)\tau_e^{s-1}(0,t)\tau_e^\prime(0,t)\ri)\\
		&\quad+\sum_{e\in{\mathcal{E}_{\tau(\cdot,t)}^2}}\le(su_e(b_e^{(t)},t)\tau_e^{s-1}(b_e^{(t)},t)\tau_e^\prime(b_e^{(t)},t)-su_e(0,t)\tau_e^{s-1}(0,t)\tau_e^\prime(0,t)\ri)\\
		&\quad+\sum_{e\in{\mathcal{E}_{\tau(\cdot,t)}^3}}\le(su_e(l_e,t)\tau_e^{s-1}(l_e,t)\tau_e^\prime(l_e,t)-su_e(a_e^{(t)},t)\tau_e^{s-1}(a_e^{(t)},t)\tau_e^\prime(a_e^{(t)},t)\ri)\\
		&=\sum_{e\in{\mathcal{E}_{\tau(\cdot,t)}^1}}\le(su_e(l_e,t)\tau_e^{s-1}(l_e,t)\frac{d\tau_e(l_e,t)}{dn}+su_e(0,t)\tau_e^{s-1}(0,t)\frac{d\tau_e(0,t)}{dn}\ri)\\
		&\quad+\sum_{e\in{\mathcal{E}_{\tau(\cdot,t)}^2}}su_e(0,t)\tau_e^{s-1}(0,t)\frac{d\tau_e(0,t)}{dn}+\sum_{e\in{\mathcal{E}_{\tau(\cdot,t)}^3}}su_e(l_e,t)\tau_e^{s-1}(l_e,t)\frac{d\tau_e(l_e,t)}{dn}\\
		&=\sum_{x\in\mathcal{V}_1^{(t)}}su(x,t)\tau^{s-1}(x,t)\sum_{e\ni x}\frac{d\tau_e(x,t)}{dn}\\
		&=\sum_{x\in\mathcal{V}^{(t)}_1}su(x,t)\tau^{s-1}(x,t)[\mathcal{K}(\tau)](x,t)\\
		&=0.
	\end{align*}
	Next we consider $B_2$. If $x\in\Gamma^{(t)}$, then
	$$
	\tau(x,t)
	=
	\phi(2)
	=
	0.
	$$
	Note that the Kirchhoff condition,
	$$
	[\mathcal K(u)](x,t)=0,
	\quad\forall\ 
	x\in\mathcal V.
	$$
	Hence
	\begin{align*}
		B_2&=\sum_{e\in{\mathcal{E}_{\tau(\cdot,t)}^1}}\le(u_e^\prime(l_e,t)\tau_e^s(l_e,t)-u_e^\prime(0,t)\tau^s_e(0,t)\ri)\\
		&\quad+\sum_{e\in{\mathcal{E}_{\tau(\cdot,t)}^2}}\le(u_e^\prime(b_e^{(t)},t)\tau_e^s(b_e^{(t)},t)-u_e^\prime(0,t)\tau^s_e(0,t)\ri)\\
		&\quad+\sum_{e\in{\mathcal{E}_{\tau(\cdot,t)}^3}}\le(u_e^\prime(l_e,t)\tau_e^s(l_e,t)-u_e^\prime(a_e^{(t)},t)\tau^s_e(a_e^{(t)},t)\ri)\\
		&=\sum_{e\in{\mathcal{E}_{\tau(\cdot,t)}^1}}\le(\tau_e^s(l_e,t)\frac{du_e(l_e,t)}{dn}+\tau^s_e(0,t)\frac{du_e(0,t)}{dn}\ri)\\
		&\quad+\sum_{e\in{\mathcal{E}_{\tau(\cdot,t)}^2}}\tau^s_e(0,t)\frac{du_e(0,t)}{dn}+\sum_{e\in{\mathcal{E}_{\tau(\cdot,t)}^3}}\tau_e^s(l_e,t)\frac{du_e(l_e,t)}{dn}\\
		&=\sum_{x\in \mathcal{V}_1^{(t)}}\tau^{s}(x,t)[\mathcal{K}(u)](x,t)\\
		&=0.
	\end{align*}
	Substituting these identities into \eqref{s4:10}, we obtain
	$$
	\sum_{e\in\mathcal E}
	\int_0^{l_e}
	u_e''(x,t)\tau_e^s(x,t)dx
	=
	\sum_{e\in\mathcal E_{\tau(\cdot,t)}}
	\int_{a_e^{(t)}}^{b_e^{(t)}}
	u_e(x,t)
	(\tau_e^s(x,t))''dx.
	$$
	
	Since $\phi\equiv0$ on $[2,\infty)$, it follows that
	$$
	(\tau_e^s(x,t))''
	=
	0,\quad\forall\ x\in	(0,l_e)
	\setminus
	[a_e^{(t)},b_e^{(t)}].
	$$
	Therefore,
	$$
	\int_{a_e^{(t)}}^{b_e^{(t)}}
	u_e(x,t)
	(\tau_e^s(x,t))''dx
	=
	\int_0^{l_e}
	u_e(x,t)
	(\tau_e^s(x,t))''dx,
	$$
and
	$$
\sum_{e\in\mathcal E}
	\int_0^{l_e}
	u_e''(x,t)\tau_e^s(x,t)dx
	=
	\sum_{e\in\mathcal E}
	\int_0^{l_e}
	u_e(x,t)
	(\tau_e^s(x,t))''dx,\quad\forall\ t\in (0,2^{\frac{1}{\theta_2}}R^{\frac{\theta_1}{\theta_2}}).$$
Since
$$\tau(\cdot,t)\equiv0,\quad\forall\ t\ge2^{\frac1{\theta_2}}R^{\frac{\theta_1}{\theta_2}},$$
both sides of the above identity vanish for $t\ge2^{1/\theta_2}R^{\theta_1/\theta_2}$, and therefore the identity holds for all $t\ge0$. Integrating with respect to $t$ over $(0,\infty)$, we derive
$$\int_0^\infty\left\{\sum_{e\in\mathcal E}\int_0^{l_e}u_e''(x,t)\tau_e^s(x,t)dx\right\}dt=
\int_0^\infty\left\{\sum_{e\in\mathcal E}\int_0^{l_e}u_e(x,t)(\tau_e^s(x,t))''dx\right\}dt.$$
The proof is complete.
\end{proof}

\begin{lemma}\label{s5:L2}
	Let $u$ be a very weak nonnegative solution to \eqref{s1:1}.  If condition \eqref{s3:10} holds, then there exists a constant $C>0$ such that 
	$$\int_0^\infty \int_{\mathcal{G}}h(x,t)u^\sigma(x,t) d\mu_\mathcal{G}dt\leq C.$$
\end{lemma}
\begin{proof}
	Let $s>\max\{2,\sigma/(\sigma-1)\}$ be a fixed constant. Let $R\geq\max\{4j,1\}$. Since $u$ is a very weak nonnegative solution to \eqref{s1:1}, testing \eqref{s2:9} with $\tau^s(x,t)$ gives
	\begin{align}\label{s5:3}
		\nonumber&\int_0^\infty \int_{\mathcal{G}}h(x,t)u^\sigma(x,t)\tau^s(x,t)d\mu_{\mathcal{G}} dt\\
		\nonumber&=\int_0^\infty\le\{\sum_{x\in\mathcal{V}}\mu_{\mathcal{V}}(x)h(x,t)u^\sigma(x,t)\tau^s(x,t)\ri\}dt+\int_0^\infty\le\{\sum_{e\in\mathcal{E}}\int_{0}^{l_e}h_e(x,t)u_e^{\sigma}(x,t)\tau_e^s(x,t)dx \ri\}dt\\
		\nonumber&\leq -\int_0^\infty \int_{\mathcal{G}}\tau^s(x,t)\Delta_{\mathcal{G}} u(x,t)d\mu_{\mathcal{G}}dt-\int_0^\infty \int_{\mathcal{G}}u(x,t)\le(\tau^s\ri)_t(x,t)d\mu_{\mathcal{G}}dt\\
		\nonumber&\quad-\sum_{x\in\mathcal{V}}\mu_{\mathcal{V}}(x)u_0(x)\tau^s(x,0)-\sum_{e\in\mathcal{E}}\int_{0}^{l_e}u_{0,e}(x)\tau_e^s(x,0)dx\\
		\nonumber&\leq -\int_{0}^\infty \le\{\sum_{x\in\mathcal{V}}\mu_{\mathcal{V}}(x)u(x,t)(\tau^s)_t(x,t)\ri\}dt-\int_0^\infty\le\{\sum_{e\in\mathcal{E}}\int_{0}^{l_e}u_e(x,t)\le(\tau_{e}^s\ri)_t(x,t)dx \ri\}dt\\
		\nonumber&\quad-\int_{0}^\infty \le\{\sum_{x\in\mathcal{V}}\mu_{\mathcal{V}}(x)\tau^s(x,t)\Delta_{\mathcal{V}} u(x,t)\ri\}dt-\int_0^\infty\le\{\sum_{e\in\mathcal{E}}\int_{0}^{l_e}u_e^{\prime\prime}(x,t)\tau_e^s(x,t)dx \ri\}dt\\
		&:=J_1+J_2+J_3+J_4.
	\end{align}
	By Lemma \ref{s4:L1}, for any $R\geq 4j$, we obtain from \eqref{s5:2} that 
	\begin{align}\label{s5:4}
		\nonumber |J_1+J_2|&\leq s\int_{0}^\infty \le\{\sum_{x\in\mathcal{V}}\mu_{\mathcal{V}}(x)u(x,t)\tau^{s-1}(x,t)\le|\tau_t(x,t)\ri|\ri\}dt\\
		\nonumber&\quad+s\int_{0}^\infty \le\{ \sum_{e\in \mathcal{E}}\int_{0}^{l_e}u_e(x,t)\tau_{e}^{s-1}(x,t)\le|\tau_{e,t}(x,t)\ri|dx\ri\}dt\\
		\nonumber&\leq \frac{C}{R}\int_{0}^\infty \le\{\sum_{x\in\mathcal{V}}\mu_{\mathcal{V}}(x)u(x,t)\tau^{s-1}(x,t)\textbf{1}_{\tilde{E}_R}(x,t)\ri\}dt\\
		\nonumber&\quad+\frac{C}{R}\int_0^\infty\le\{\sum_{e\in \mathcal{E}}\int_{0}^{l_e} u_e(x,t)\tau_e^{s-1}(x,t)\textbf{1}_{\tilde{E}_R}(x,t)dx\ri\}dt\\
		\nonumber&\leq \frac{C}{R}\int_{0}^\infty \le\{\sum_{x\in\mathcal{V}}\mu_{\mathcal{V}}(x)u(x,t)\tau^{s-1}(x,t)\textbf{1}_{M_R}(x,t)\ri\}dt\\
		&\quad+\frac{C}{R}\int_0^\infty\le\{\sum_{e\in \mathcal{E}}\int_{0}^{l_e} u_e(x,t)\tau_e^{s-1}(x,t)\textbf{1}_{M_R}(x,t)dx\ri\}dt.
	\end{align}
	Since 
	\begin{equation*}
		(\tau_e^s(x,t))^{\prime\prime}=s(s-1)\tau_e^{s-2}(x,t)\le(\tau_e^\prime(x,t)\ri)^2+s\tau_e^{s-1}(x,t)\tau_e^{\prime\prime}(x,t),
	\end{equation*}
	and $s(s-1)\tau_e^{s-2}(x,t)\le(\tau_e^\prime(x,t)\ri)^2\leq0$ with $s>2$, we deduce that
	\begin{equation}\label{s5:9}
		-(\tau_e^s(x,t))^{\prime\prime}\leq-s\tau_e^{s-1}(x,t)\tau_e^{\prime\prime}(x,t),\quad\forall\ e\in \mathcal{E}, x\in (0,l_e), t\in[0,\infty).
	\end{equation}
	Noting that $r\mapsto r^s$ is convex on [0,1], one has
	$$a^s-b^s\geq sb^{s-1}(a-b)$$
	for all $a, b\in[0,1]$. It then follows that
	\begin{align}\label{s4:15}
		\nonumber-\Delta_{\mathcal{V}} \tau^s(x,t)&=-\frac{1}{\mu_{\mathcal{V}}(x)}\sum_{y\sim x}\o(x,y)\le(\tau^s(y,t)-\tau^s(x,t)\ri)\\
		\nonumber&\leq-\frac{1}{\mu_{\mathcal{V}}(x)}\sum_{y\sim x}\o(x,y)s\tau^{s-1}(x,t)\le(\tau(y,t)-\tau(x,t)\ri)\\
		&=-s\tau^{s-1}(x,t)\Delta_{\mathcal{V}} \tau(x,t).
	\end{align} 
	Using Lemmas \ref{s4:L1} and \ref{s5:L1},
	\[
	-\Delta_{\mathcal V}\tau(x,t)
	\leq
	\frac{C}{R}\mathbf 1_{M_R}(x,t),
	\qquad
	-\tau_e''(x,t)
	\leq
	\frac{C}{R}\mathbf 1_{M_R}(x,t),
	\]
	for all $R\ge \max\{4j,1\}$. By Lemma \ref{s5:L5}, in view of \eqref{s5:9} and \eqref{s4:15}, we have 
	\begin{align}\label{s5:11}
		\nonumber &J_3+J_4\\
		\nonumber&=-\int_{0}^\infty \le\{\sum_{x\in\mathcal{V}}\mu_{\mathcal{V}}(x)u(x,t)\Delta_{\mathcal{V}} \tau^s(x,t)\ri\}dt-\int_0^\infty\left\{\sum_{e\in\mathcal E}\int_0^{l_e}u_e(x,t)(\tau_e^s(x,t))''dx\right\}dt\\
		\nonumber&\leq -\int_{0}^\infty \le\{\sum_{x\in\mathcal{V}}\mu_{\mathcal{V}}(x)u(x,t)s\tau^{s-1}(x,t)\Delta_{\mathcal{V}} \tau(x,t)\ri\}dt-\int_0^\infty\le\{\sum_{e\in\mathcal E}\int_0^{l_e}u_e(x,t)s\tau_e^{s-1}(x,t)\tau_e^{\prime\prime}(x,t)dx \ri\}dt\\
		&\leq \frac{C}{R}\int_{0}^\infty \le\{\sum_{x\in\mathcal{V}}\mu_{\mathcal{V}}(x)u(x,t)\tau^{s-1}(x,t)\textbf{1}_{M_R}(x,t)\ri\}dt+\frac{C}{R}\int_0^\infty\le\{\sum_{e\in\mathcal{E}}\int_{0}^{l_e}u_e(x,t)\tau_e^{s-1}(x,t)\textbf{1}_{M_R}(x,t)dx \ri\}dt.
	\end{align}
	Combining \eqref{s5:4} and \eqref{s5:11}, and using Young's inequality, for any $\epsilon>0$, we have
	\begin{align}\label{s5:6}
		\nonumber&J_1+J_2+J_3+J_4\\
		\nonumber&\leq  \frac{C}{R}\int_{0}^\infty \le\{\sum_{x\in\mathcal{V}}\mu_{\mathcal{V}}(x)u(x,t)\tau^{s-1}(x,t)\textbf{1}_{M_R}(x,t)dt\ri\}dt+\frac{C}{R}\int_0^\infty\le\{\sum_{e\in \mathcal{E}}\int_{0}^{l_e} u_e(x,t)\tau_e^{s-1}(x,t)\textbf{1}_{M_R}(x,t)dx\ri\}dt\\
		\nonumber&\leq \epsilon \int_0^\infty \le\{\sum_{x\in\mathcal{V}}\mu_{\mathcal{V}}(x) h(x,t)\tau^{s}(x,t)u^\sigma(x,t)\textbf{1}_{M_R}(x,t)\ri\}dt+\frac{C}{R^{\frac{\sigma}{\sigma-1}}}\int_0^\infty \le\{\sum_{x\in\mathcal{V}}\mu_{\mathcal{V}}(x) h^{-\frac{1}{\sigma-1}}(x,t)\textbf{1}_{M_R}(x,t)\ri\}dt\\
		\nonumber&\quad+\epsilon\int_0^\infty \le\{\sum_{e\in \mathcal{E}}\int_{0}^{l_e} h_e(x,t)\tau_e^{s}(x,t)u^\sigma_e(x,t)\textbf{1}_{M_R}(x,t)dx\ri\}dt+\frac{C}{R^{\frac{\sigma}{\sigma-1}}}\int_0^\infty \le\{\sum_{e\in \mathcal{E}}\int_{0}^{l_e} h_e^{-\frac{1}{\sigma-1}}(x,t)\textbf{1}_{M_R}(x,t)dx\ri\}dt\\
		\nonumber&\leq \epsilon \int_0^\infty \int_{\mathcal{G}}h(x,t)u^\sigma(x,t)\tau^s(x,t)d\mu_{\mathcal{G}} dt+\frac{C}{R^{\frac{\sigma}{\sigma-1}}}\int_0^\infty \le\{\sum_{x\in\mathcal{V}}\mu_{\mathcal{V}}(x) h^{-\frac{1}{\sigma-1}}(x,t)\textbf{1}_{M_R}(x,t)\ri\}dt\\
		&\quad+\frac{C}{R^{\frac{\sigma}{\sigma-1}}}\int_0^\infty \le\{\sum_{e\in \mathcal{E}}\int_{0}^{l_e} h_e^{-\frac{1}{\sigma-1}}(x,t)\textbf{1}_{M_R}(x,t)dx\ri\}dt.
	\end{align}
	In view of \eqref{s5:3} and \eqref{s5:6},  for every $R\geq \max\{4j,1\}$, we get
	\begin{align*}
		&\int_0^\infty \int_{\mathcal{G}}h(x,t)u^\sigma(x,t)\tau^s(x,t)d\mu_{\mathcal{G}} dt\\
		\nonumber&\leq \epsilon\int_0^\infty \int_{\mathcal{G}}h(x,t)u^\sigma(x,t)\tau^s(x,t)d\mu_{\mathcal{G}} dt+\frac{C}{R^{\frac{\sigma}{\sigma-1}}}\int_0^\infty \le\{\sum_{x\in\mathcal{V}}\mu_{\mathcal{V}}(x) h^{-\frac{1}{\sigma-1}}(x,t)\textbf{1}_{M_R}(x,t)\ri\}dt\\
		\nonumber&\quad +\frac{C}{R^{\frac{\sigma}{\sigma-1}}}\int_0^\infty \le\{\sum_{e\in \mathcal{E}}\int_{0}^{l_e} h_e^{-\frac{1}{\sigma-1}}(x,t)\textbf{1}_{M_R}(x,t)dx\ri\}dt.
	\end{align*}    
	Let $k\in\mathbb{N}$ be sufficiently large such that $k\geq3\theta_1-1$, then for each $R\geq \max\{4j,1\}$, 
	\begin{equation}\label{s5:7}
		M_R\subset\bigcup_{j=0}^k E_{2^{\frac{j}{\theta_1}-1}R}.
	\end{equation}
	where $E_R=\le\{(x,t)\in\mathcal{G}\times[0,\infty):R^{\theta_1}\leq d(x,x_0)^{\theta_1}+t^{\theta_2}\leq 2R^{\theta_1}\ri\}$.  By virtue of this decomposition, we deduce from \eqref{s3:10} that
	\begin{align}\label{s5:24}
		\nonumber&\int_0^\infty \int_{\mathcal{G}}h(x,t)u^\sigma(x,t)\tau^s(x,t)d\mu_{\mathcal{G}} dt\\
		\nonumber&\leq \epsilon \int_0^\infty \int_{\mathcal{G}}h(x,t)u^\sigma(x,t)\tau^s(x,t)d\mu_{\mathcal{G}} dt+\frac{C}{R^{\frac{\sigma}{\sigma-1}}}\sum_{j=0}^k\int_0^\infty \le\{\sum_{x\in\mathcal{V}}\mu_{\mathcal{V}}(x) h^{-\frac{1}{\sigma-1}}(x,t)\textbf{1}_{E_{2^{\frac{j}{\theta_1}-1}R}}(x,t)\ri\}dt\\
		\nonumber&\quad+\frac{C}{R^{\frac{\sigma}{\sigma-1}}}\sum_{j=0}^k\int_0^\infty \le\{\sum_{e\in \mathcal{E}}\int_{0}^{l_e} h_e^{-\frac{1}{\sigma-1}}(x,t)\textbf{1}_{E_{2^{\frac{j}{\theta_1}-1}R}}(x,t)dx\ri\}dt\\
		\nonumber&\leq \epsilon \int_0^\infty \int_{\mathcal{G}}h(x,t)u^\sigma(x,t)\tau^s(x,t)d\mu_{\mathcal{G}} dt+\frac{C}{R^{\frac{\sigma}{\sigma-1}}}\sum_{j=0}^k\le(2^{\frac{j}{\theta_1}-1}R\ri)^{\frac{\sigma}{\sigma-1}}\\
		&\leq \epsilon \int_0^\infty \int_{\mathcal{G}}h(x,t)u^\sigma(x,t)\tau^s(x,t)d\mu_{\mathcal{G}} dt+C.
	\end{align}
	Thus, for sufficiently small $\epsilon>0$, we have
	\begin{align*}
		\int_0^\infty \int_{\mathcal{G}}h(x,t)u^\sigma(x,t)\tau^s(x,t)d\mu_{\mathcal{G}} dt&\leq C.
	\end{align*}
	Considering a bounded domain 
	\begin{equation}\label{s5:25}
	F_R=\le\{(x,t)\in\mathcal{G}\times[0,\infty):\tilde{d}(x,x_0)^{\theta_1}+t^{\theta_2}\leq R^{\theta_1}\ri\},
	\end{equation}
	noting that $\tau\equiv1$ on $F_R$, we thus get
	\begin{align*}
		\int_0^\infty \int_{\mathcal{G}}h(x,t)u^\sigma(x,t)\textbf{1}_{F_R}(x,t)d\mu_\mathcal{G}dt\leq 	\int_0^\infty \int_{\mathcal{G}}h(x,t)u^\sigma(x,t)\tau^s(x,t)d\mu_{\mathcal{G}} dt
		\leq C,
	\end{align*}
	where $C$ is a constant independent of $R$. Finally, letting  $R\ra\infty$, and applying the monotone convergence theorem, we conclude
	$$\int_0^\infty\int_\mathcal{G}h(x,t)u^\sigma(x,t) d\mu_\mathcal{G}dt\leq C,$$
	which completes the proof of this lemma.
\end{proof}

\begin{remark}
	The estimate in Lemma \ref{s5:L2} shows that
	\[
	h(x,t)u^\sigma(x,t)
	\in
	L^1(\mathcal G\times[0,\infty)).
	\]
	This integrability property will be crucial in the proof of
	Theorem \ref{s2:T3}, where it allows us to pass to the limit
	on the annular regions $M_R$ and conclude that
	\[
	\int_{M_R}
	h(x,t)u^\sigma(x,t)\,d\mu_{\mathcal G}dt
	\to0
	\quad\text{as }R\to\infty.
	\]
\end{remark}

Finally, we present the proofs of Theorem \ref{s2:T3} and Corollaries \ref{s2:C1} and \ref{s2:C4}. For the sake of brevity, overlapping parts are omitted here.\\

\noindent$\textbf{\emph{Proof of Theorem \ref{s2:T3}.}}$  We only indicate the main steps, since several arguments are analogous to those in the proof of Lemma \ref{s5:L2}. By \eqref{s5:6}, we obtain 
\begin{align}\label{s4:8}
	\nonumber\int_0^\infty \int_{\mathcal{G}}h(x,t)u^\sigma(x,t)\tau^s(x,t)d\mu_{\mathcal{G}} dt& \leq \frac{C}{R}\int_{0}^\infty \le\{\sum_{x\in\mathcal{V}}\mu_{\mathcal{V}}(x)u(x,t)\tau^{s-1}(x,t)\textbf{1}_{M_R}(x,t)dt\ri\}dt\\
	&\quad+\frac{C}{R}\int_0^\infty\le\{\sum_{e\in \mathcal{E}}\int_{0}^{l_e} u_e(x,t)\tau_e^{s-1}(x,t)\textbf{1}_{M_R}(x,t)dx\ri\}dt.
\end{align}
By applying H\"{o}lder's inequality, from \eqref{s3:10} and \eqref{s5:7}, we have
\begin{align*}
	&\frac{C}{R}\int_{0}^\infty \le\{\sum_{x\in\mathcal{V}}\mu_{\mathcal{V}}(x)u(x,t)\tau^{s-1}(x,t)\textbf{1}_{M_R}(x,t)dt\ri\}dt\\
	&\leq\frac{C}{R}\le(\int_0^\infty \le\{\sum_{x\in\mathcal{V}}\mu_{\mathcal{V}}(x) h(x,t)\tau^{s}(x,t)u^\sigma(x,t)\textbf{1}_{M_R}(x,t)\ri\}dt\ri)^{\frac{1}{\sigma}}\le(\int_0^\infty \le\{\sum_{x\in\mathcal{V}}\mu_{\mathcal{V}}(x) h^{-\frac{1}{\sigma-1}}(x,t)\textbf{1}_{M_R}(x,t)\ri\}dt\ri)^{\frac{\sigma-1}{\sigma}}\\
	&\leq\frac{C}{R}\le(\int_0^\infty \le\{\sum_{x\in\mathcal{V}}\mu_{\mathcal{V}}(x) h(x,t)u^\sigma(x,t)\textbf{1}_{M_R}(x,t)\ri\}dt\ri)^{\frac{1}{\sigma}}\le(\sum_{j=0}^k\le(2^{\frac{j}{\theta_1}-1}R\ri)^{\frac{\sigma}{\sigma-1}}\ri)^{\frac{\sigma-1}{\sigma}}\\
	&\leq C\le(\int_0^\infty \le\{\sum_{x\in\mathcal{V}}\mu_{\mathcal{V}}(x) h(x,t)u^\sigma(x,t)\textbf{1}_{M_R}(x,t)\ri\}dt\ri)^{\frac{1}{\sigma}}.
\end{align*}
Similarly, 
\begin{align*}
	&\frac{C}{R}\int_0^\infty\le\{\sum_{e\in \mathcal{E}}\int_{0}^{l_e}u_e(x,t)\tau_e^{s-1}(x,t)\textbf{1}_{M_R}(x,t) dx\ri\}dt\\
	&\leq \frac{C}{R} \sum_{e\in \mathcal{E}}\int_{0}^{l_e}\le\{\le(\int_0^\infty h_e(x,t)\tau_e^{s}(x,t)u_e^\sigma(x,t)\textbf{1}_{M_R}(x,t)dt\ri)^{\frac{1}{\sigma}}\le(\int_0^\infty 	\tau_e^{s-\frac{\sigma-1}{\sigma}}(x,t)h_e^{-\frac{1}{\sigma-1}}(x,t)\textbf{1}_{M_R}(x,t)dt\ri)^{\frac{\sigma-1}{\sigma}}\ri\}dx\\
	&\leq \frac{C}{R} \sum_{e\in\mathcal{E}}\le(\int_{0}^{l_e}\int_0^\infty h_e(x,t)\tau_e^{s}(x,t)u_e^\sigma(x,t)\textbf{1}_{M_R}(x,t)dt dx\ri)^\frac{1}{\sigma}\le(\int_{0}^{l_e}\int_0^\infty\tau_e^{s-\frac{\sigma-1}{\sigma}}(x,t)h_e^{-\frac{1}{\sigma-1}}(x,t)\textbf{1}_{M_R}(x,t)dtdx\ri)^\frac{\sigma-1}{\sigma}\\
	&\leq \frac{C}{R} \le\{\sum_{e\in\mathcal{E}}\int_{0}^{l_e}\le(\int_0^\infty h_e(x,t)\tau_e^{s}(x,t)u_e^\sigma(x,t)\textbf{1}_{M_R}(x,t)dt\ri) dx\ri\}^\frac{1}{\sigma}\\
	&\quad\cdot\le\{\sum_{e\in\mathcal{E}}\int_{0}^{l_e}\le(\int_0^\infty\tau_e^{s-\frac{\sigma-1}{\sigma}}(x,t)h_e^{-\frac{1}{\sigma-1}}(x,t)\textbf{1}_{M_R}(x,t)dt\ri)dx\ri\}^\frac{\sigma-1}{\sigma}\\
	&\leq \frac{C}{R} \le\{\int_0^\infty\le( \sum_{e\in\mathcal{E}}\int_{0}^{l_e}h_e(x,t)u_e^\sigma(x,t)\textbf{1}_{M_R}(x,t)dx\ri) dt\ri\}^\frac{1}{\sigma}\le\{\int_0^\infty\le(\sum_{e\in\mathcal{E}}\int_{0}^{l_e} h_e^{-\frac{1}{\sigma-1}}(x,t)\textbf{1}_{M_R}(x,t)dx\ri)dt\ri\}^\frac{\sigma-1}{\sigma}\\
	&\leq \frac{C}{R} \le\{\int_0^\infty\le( \sum_{e\in\mathcal{E}}\int_{0}^{l_e}h_e(x,t)u_e^\sigma(x,t)\textbf{1}_{M_R}(x,t)dx\ri) dt\ri\}^\frac{1}{\sigma}\le(\sum_{j=0}^k\le(2^{\frac{j}{\theta_1}-1}R\ri)^{\frac{\sigma}{\sigma-1}}\ri)^{\frac{\sigma-1}{\sigma}}\\
	& \leq C\le(\int_0^\infty\le( \sum_{e\in\mathcal{E}}\int_{0}^{l_e}h_e(x,t)u_e^\sigma(x,t)\textbf{1}_{M_R}(x,t)dx\ri) dt\ri)^\frac{1}{\sigma}.
\end{align*}
It then follows from these estimates and \eqref{s4:8} that
\begin{align}\label{s4:11}
	\nonumber\int_0^\infty \int_{\mathcal{G}}h(x,t)u^\sigma(x,t)\textbf{1}_{F_R}(x,t)d\mu_\mathcal{G}dt&
	\nonumber\leq C\le(\int_0^\infty \le\{\sum_{x\in\mathcal{V}}\mu_{\mathcal{V}}(x) h(x,t)u^\sigma(x,t)\textbf{1}_{M_R}(x,t)\ri\}dt\ri)^{\frac{1}{\sigma}}\\
	&\quad+C\le(\int_0^\infty\le( \sum_{e\in\mathcal{E}}\int_{0}^{l_e}h_e(x,t)u_e^\sigma(x,t)\textbf{1}_{M_R}(x,t)dx\ri) dt\ri)^{\frac{1}{\sigma}}.
\end{align}
By Lemma \ref{s5:L2},
\[
\int_0^\infty
\left\{
\sum_{x\in\mathcal V}
\mu_{\mathcal V}(x)
h(x,t)u^\sigma(x,t)
\right\}dt<\infty,\quad
\int_0^\infty
\left\{
\sum_{e\in\mathcal E}
\int_0^{l_e}
h_e(x,t)u_e^\sigma(x,t)dx
\right\}dt<\infty.
\]
Since
$M_R=\left\{R^{\theta_1}\leq\tilde d(x,x_0)^{\theta_1}+t^{\theta_2}\leq2R^{\theta_1}\right\}$,
for every fixed $(x,t)\in\mathcal G\times(0,\infty)$, one has $\mathbf 1_{M_R}(x,t)\to0$ as $R\to\infty$.
Moreover,
\[
0\leq
h(x,t)u^\sigma(x,t)\mathbf 1_{M_R}(x,t)
\leq
h(x,t)u^\sigma(x,t),
\]
Hence, by the dominated convergence theorem,
\[
\int_0^\infty \le\{\sum_{x\in\mathcal{V}}\mu_{\mathcal{V}}(x) h(x,t)u^\sigma(x,t)\textbf{1}_{M_R}(x,t)\ri\}dt
\to0, \quad \int_0^\infty\le( \sum_{e\in\mathcal{E}}\int_{0}^{l_e}h_e(x,t)u_e^\sigma(x,t)\textbf{1}_{M_R}(x,t)dx\ri) dt
\to0,
\quad R\to\infty.
\]
Finally, as $F_R\uparrow\mathcal{G}\times(0,\infty)$, it follows from \eqref{s4:11} and the monotone convergence theorem that
$$\int_0^\infty \int_{\mathcal{G}}h(x,t)u^\sigma(x,t)d\mu_\mathcal{G}dt\leq0.$$
Since $h>0$ and $u\geq0$, it follows that $u(x,t)\equiv0$ for all $(x,t)\in\mathcal{G}\times[0,\infty)$. This completes the proof. $\hfill\Box$\\

\noindent$\textbf{\emph{Proof of Corollary \ref{s2:C1}.}}$  (i) Setting $\theta_1=\theta_2=2$,  we see that the set $E_R$ satisfies
$$E_R\subset B_{\sqrt{2}R}\times[0,\sqrt{2}R],$$
since $d(x,x_0)^2+t^2\leq 2R^2$ for $(x,t)\in E_R$. Hence, for all sufficiently large $R\geq \max\{4j,1\}$, by \eqref{s3:9}, we have
\begin{align*}
	\int_0^\infty \le\{\sum_{x\in\mathcal{V}}\mu_{\mathcal{V}}(x) h^{-\frac{1}{\sigma-1}}(x,t)\textbf{1}_{E_R}(x,t)\ri\}dt&\leq\le(\int_0^{\sqrt{2}R}f^{-\frac{1}{\sigma-1}}(t)dt\ri)\le(\sum_{x\in\mathcal{V} \cap B_{\sqrt{2}R}}\mu_{\mathcal{V}}(x)g^{-\frac{1}{\sigma-1}}(x)\ri)\\
	&\leq C(\sqrt{2}R)^{\xi_1}(\sqrt{2}R)^{\xi_2}\\
	&\leq CR^{\frac{\sigma}{\sigma-1}},
\end{align*}
and
\begin{align*}
	\int_0^\infty\le\{\sum_{e\in \mathcal{E}}\int_{0}^{l_e}h_e^{-\frac{1}{\sigma-1}}(x,t)\textbf{1}_{E_R}(x,t)dx\ri\}dt&\leq 	\int_0^\infty\le\{\sum_{e\in \mathcal{E}}\int_{0}^{l_e}f_e^{-\frac{1}{\sigma-1}}(t)g_e^{-\frac{1}{\sigma-1}}(x)\textbf{1}_{E_R}(x,t)dx\ri\}dt\\
	&\leq\le(\int_0^{\sqrt{2}R}f_e^{-\frac{1}{\sigma-1}}(t)dt\ri)\le(\sum_{e\in \mathcal{E}\cap B_{\sqrt{2}R}}\int_{0}^{l_e}g_e^{-\frac{1}{\sigma-1}}(x)dx\ri)\\
	&\leq C(\sqrt{2}R)^{\xi_1}(\sqrt{2}R)^{\xi_2}\\
	&\leq CR^{\frac{\sigma}{\sigma-1}},
\end{align*}
where $C$ is a constant independent of $R$, and we have used $\xi_1+\xi_2\leq \sigma/(\sigma-1)$. Hence, by Theorem \ref{s2:T3}, we deduce that $u\equiv0$.

(ii) By choosing $f\equiv1$ and noting condition \eqref{s3:14}, we set $\xi_1=1$ and $\xi_2=1/(\sigma-1)$, so that $\xi_1+\xi_2=\sigma/(\sigma-1)$. The conclusion then follows immediately from part (i). $\hfill\Box$\\

\noindent$\textbf{\emph{Proof of Corollary \ref{s2:C4}.}}$ This is a consequence of part (ii) in Corollary \ref{s2:C1}, obtained by setting $g\equiv1$. $\hfill\Box$

\subsection{Nonexistence for sign-changing global solutions}\label{3.3}

Fix $x_0\in\mathcal{V}$ and let $R\geq R_0$, where $R_0:=\max\{2j,1\}$. Let 
$$\psi:\le[-\frac{j}{R}, \infty\ri)\to(0,\infty)$$
be a function satisfying 
\begin{itemize}
	\item $\psi\in C^2([-j/R, \infty))$;
	\item $\psi\equiv1$ on $[-j/R,1]$ and $\psi(r)=e^{-\delta r}$ on $[2,\infty)$ for some  $\delta>0$;
	\item $\psi^\prime(x)\leq0$ for all $x\in [-j/R, \infty)$.
\end{itemize}
Then there exist constants $C_1>0$, $C_2 > 0$, independent of $R$, such that
\begin{equation}\label{s4:17}
	0<C_2e^{-\delta r}\leq\psi(r)\leq C_1e^{-\delta r},\quad\forall\ r\in \le[-\frac{j}{R}, \infty\ri),
\end{equation}
and
\begin{equation}\label{s4:18}
	|\psi^\prime(r)|\leq C_1 e^{-\delta r},\quad|\psi^{\prime\prime}(r)|\leq C_1e^{-\delta r},\quad\forall\ r\in \le[-\frac{j}{R}, \infty\ri).
\end{equation}

Next, we introduce the test function
\begin{equation}\label{s5:17}
	\Phi(x,t):=\phi^s\le(\frac{t}{R}\ri)\psi\le(\frac{\tilde{d}(x,x_0)-j}{R}\ri),
\end{equation}
where $\phi\in C^2([0,\infty))$ is the cut-off function defined in \eqref{s4:6} and $s>\sigma/(\sigma-1)$. 
Let 
$$\tilde{B}_{R+j}^c=\{x\in\mathcal{G}:\tilde{d}(x,x_0)\geq R+j\},\quad B_R^c=\{x\in\mathcal{G}:d(x,x_0)\geq R\}.$$
The following lemma provides the key estimates for $\phi$.

\begin{lemma}\label{s5:L3}
	For any $R\geq R_0$, there exists a constant $C>0$ such that
	
	(i)	For each edge $e\in\mathcal{E}$, all $x\in (0,l_e)$, and all $t\in[0,\infty)$, it holds
	\begin{equation}\label{s4:24}
		|\Phi_e^{\prime\prime}(x,t)|\leq \frac{C}{R}\phi^s\le(\frac{t}{R}\ri)e^{\frac{-\delta \tilde{d}(x,x_0)}{R}}\textbf{1}_{\tilde{B}_{R+j}^c}(x).
	\end{equation}
	
	(ii) For any $x\in\mathcal{V}$ and $t\in[0,\infty)$, it holds
	\begin{equation}\label{s4:25}
		\le|\Delta_{\mathcal{V}}\Phi(x,t)\ri|\leq \frac{C}{R}\phi^s\le(\frac{t}{R}\ri)e^{\frac{-\delta d(x,x_0)}{R}}\textbf{1}_{B_{R}^c}(x).
	\end{equation}
	
	(iii) For any $(x,t)\in\mathcal{G}\times[0,\infty)$, it holds
	\begin{equation}\label{s4:26}
		|\Phi_{t}(x,t)|\leq \frac{C}{R}\phi^{s-1}\le(\frac{t}{R}\ri)e^{\frac{-\delta d(x,x_0)}{R}}\textbf{1}_{Q_R}(x,t),
	\end{equation}
	where  $Q_R:=\mathcal{G}\times[R,2R]$.
\end{lemma}
\begin{proof}
	(i) For each $e\in\mathcal{E}$, noting that 
	$$\psi_e^{\prime}\le(\frac{\tilde{d}(x,x_0)-j}{R}\ri)=\psi_e^{\prime\prime}\le(\frac{\tilde{d}(x,x_0)-j}{R}\ri)=0,$$
	 for every $x\in\tilde{B}_{R+j}$, by Proposition \ref{s4:P1},  we obtain that for $R\geq R_0$
	\begin{align*}
		\le|\Phi_e^{\prime\prime}(x,t)\ri|&=\le|\phi^s\le(\frac{t}{R}\ri)\psi_e^{\prime\prime}\le(\frac{\tilde{d}(x,x_0)-j}{R}\ri)\frac{1}{R^2}\le(\tilde{d}_e^{\prime}(x,x_0)\ri)^2+\phi^s\le(\frac{t}{R}\ri)\psi_e^\prime\le(\frac{\tilde{d}(x,x_0)-j}{R}\ri)\frac{1}{R}\tilde{d}_e^{\prime\prime}(x,x_0)\ri|\\
		&\leq\frac{C}{R^2}\phi^s\le(\frac{t}{R}\ri)e^{-\delta\frac{\tilde{d}(x,x_0)-j}{R}}\textbf{1}_{\tilde{B}_{R+j}^c}(x)+\frac{C}{R}\phi^s\le(\frac{t}{R}\ri)e^{-\delta\frac{\tilde{d}(x,x_0)-j}{R}}\textbf{1}_{\tilde{B}_{R+j}^c}(x)\\
		&\leq \frac{C}{R}\phi^s\le(\frac{t}{R}\ri)e^{\frac{-\delta \tilde{d}(x,x_0)}{R}}\textbf{1}_{\tilde{B}_{R+j}^c}(x), \quad\forall\ x\in(0,l_e),\ t\in[0,\infty).
	\end{align*}
	
	(ii) For each $x\in\mathcal{V}$, there exists some $\xi$ such that
	\begin{align*}
		\Delta_{\mathcal{V}}\Phi(x,t)&=\frac{1}{\mu_{\mathcal{V}}(x)}\sum_{y\sim x}\o(x,y)\le(\Phi(y,t)-\Phi(x,t)\ri)\\
		&=\frac{1}{\mu_{\mathcal{V}}(x)}\sum_{y\sim x}\o(x,y)\phi^s\le(\frac{t}{R}\ri)\le(\psi\le(\frac{\tilde{d}(y,x_0)-j}{R}\ri)-\psi\le(\frac{\tilde{d}(x,x_0)-j}{R}\ri)\ri)\\
		&=\frac{1}{\mu_{\mathcal{V}}(x)}\sum_{y\sim x}\phi^s\le(\frac{t}{R}\ri)\o(x,y)\le(\psi\le(\frac{d(y,x_0)-j}{R}\ri)-\psi\le(\frac{d(x,x_0)-j}{R}\ri)\ri)\\
		&=\frac{1}{\mu_{\mathcal{V}}(x)}\sum_{y\sim x}\o(x,y)\phi^s\le(\frac{t}{R}\ri)\psi^\prime\le(\frac{d(x,x_0)-j}{R}\ri)\le(\frac{d(y,x_0)-d(x,x_0)}{R}\ri)\\
		&\quad+\frac{1}{2\mu_{\mathcal{V}}(x)}\sum_{y\sim x}\o(x,y)\phi^s\le(\frac{t}{R}\ri)\psi^{\prime\prime}\le(\xi\ri)\le(\frac{d(y,x_0)-d(x,x_0)}{R}\ri)^2,\quad t\in[0,\infty).
	\end{align*}
	Here, $\xi$ between $(d(y,x_0)-j)/R$ and $(d(x,x_0)-j)/R$.  
	
	If $x\in B_R\cap\mathcal V$,
	then for every $y\sim x$,
	$$\frac{d(y,x_0)-j}{R}\leq1,\quad
	\frac{d(x,x_0)-j}{R}
	\leq1,$$
	since \(d(y,x_0)\leq d(x,x_0)+j\leq R+j\). Hence, $\xi\leq1$ for all $x\in B_R\cap\mathcal{V}$. In addition, for any $x\in B_R$, by \eqref{s4:16},
	$$\tilde{d}(x,x_0)\leq d(x,x_0)+j\leq R+j,$$
	thus $x\in \tilde{B}_{R+j}$. Noting that $\psi^\prime\equiv0$ and $\psi^{\prime\prime}\equiv0$ on $\tilde{B}_{R+j}$, we have $$\Delta_{\mathcal{V}}\Phi(x,t)\equiv0,$$ 
	for all $(x,t)\in B_R\times [0,\infty)$. 
	
	For $x\in B_R^c\cap \mathcal{V}$, due to $R\geq R_0$, we get
	\begin{equation*}
		\xi\geq \min\le\{\frac{d(y,x_0)-j}{R},\frac{d(x,x_0)-j}{R}\ri\}\geq \frac{d(x,x_0)-2j}{R}\geq\frac{d(x,x_0)}{R}-1\geq0,
	\end{equation*}
	hence
	\begin{equation}\label{s4:12}
		e^{-\delta\xi}\leq e^\delta e^{-\frac{\delta d(x,x_0)}{R}}.
	\end{equation}
	Since
	\begin{align*}
		\Delta_{\mathcal{V}}d(x,x_0)&=\frac{1}{\mu_{\mathcal{V}}(x)}\sum_{y\sim x}\o(x,y)\le(d(y,x_0)-d(x,x_0)\ri)\\
		\nonumber&\leq \frac{1}{\mu_{\mathcal{V}}(x)}\sum_{y\sim x}\o(x,y)d(x,y)\\
		&\leq \frac{1}{\mu_{\mathcal{V}}(x)}\sum_{y\sim x}\o(x,y)j\leq Cj,
	\end{align*}
	by \eqref{s4:18} and \eqref{s4:12}, we obtain for any $(x,t)\in\mathcal{V}\times [0,\infty)$,
	\begin{align*}
		\le|\Delta_{\mathcal{V}}\Phi(x,t)\ri|&\leq \frac{1}{R}\le|\psi^\prime\le(\frac{d(x,x_0)-j}{R}\ri)\ri|\phi^s\le(\frac{t}{R}\ri)\frac{1}{\mu_{\mathcal{V}}(x)}\sum_{y\sim x}\o(x,y)\le(d(y,x_0)-d(x,x_0)\ri)\textbf{1}_{B_{R}^c}(x)\\
		&\quad+\frac{C}{2R^2}\phi^s\le(\frac{t}{R}\ri)\frac{1}{\mu_{\mathcal{V}}(x)}\sum_{y\sim x}\o(x,y)e^{-\frac{\delta d(x,x_0)}{R}}\le(d(y,x_0)-d(x,x_0)\ri)^2\textbf{1}_{B_{R}^c}(x)\\
		&\leq \frac{C}{R} \phi^s\le(\frac{t}{R}\ri)e^{-\delta\frac{d(x,x_0)-j}{R}} \Delta_\mathcal{V}d(x,x_0)\textbf{1}_{B_{R}^c}(x)\\
		\nonumber&\quad+\frac{C}{2R^2}\phi^s\le(\frac{t}{R}\ri)e^{-\frac{\delta d(x,x_0)}{R}}\textbf{1}_{B_{R}^c}(x)\frac{1}{\mu_{\mathcal{V}}(x)}\sum_{y\sim x}\o(x,y)j^2\\
		&\leq\le[\frac{C}{R}e^{-\delta\frac{ d(x,x_0)}{R}}+\frac{C }{2R^2}e^{-\delta\frac{d(x,x_0)}{R}}\ri]\phi^s\le(\frac{t}{R}\ri)\textbf{1}_{B_{R}^c}(x)\\
		&\leq\frac{C}{R}\phi^s\le(\frac{t}{R}\ri)e^{-\frac{\delta d(x,x_0)}{R}}\textbf{1}_{B_{R}^c}(x),
	\end{align*}
	where we have used \eqref{s3:1}-(v) and $|d(y,x_0)-d(x,x_0)|\leq d(x,y)\leq j$.  
	
	(iii) By \eqref{s4:16}, \eqref{s4:6} and \eqref{s4:17}, we compute
	\begin{align*}
		|\Phi_{t}(x,t)|&=\le|s\phi^{s-1}\le(\frac{t}{R}\ri)\phi^\prime\le(\frac{t}{R}\ri)\frac{1}{R}\psi\le(\frac{\tilde{d}(x,x_0)-j}{R}\ri)\ri|\\
		&\leq \frac{C}{R}\phi^{s-1}\le(\frac{t}{R}\ri)e^{\frac{-\delta \tilde{d}(x,x_0)+\delta j}{R}}\textbf{1}_{Q_R}(x,t)\\
		&\leq\frac{C}{R}\phi^{s-1}\le(\frac{t}{R}\ri)e^{\frac{-\delta d(x,x_0)+2\delta j}{R}}\textbf{1}_{Q_R}(x,t)\\
		&\leq \frac{C}{R}\phi^{s-1}\le(\frac{t}{R}\ri)e^{\frac{-\delta d(x,x_0)}{R}}\textbf{1}_{Q_R}(x,t),\quad\forall \ x\in\mathcal{G},\ t\in[0,\infty),
	\end{align*}
	where the factor $e^{2\delta j/R}$ has been absorbed into the constant $C$, since $R\geq R_0\geq 2j$.
	
	The proof of the lemma is complete.
\end{proof}

Although the test function \eqref{s5:17} is not compactly supported in the spatial variable, the integration by parts formula remains valid under the additional assumption $u\in L^1_{\rm loc}([0,\infty);X_\alpha)$. Since $\Phi(\cdot,t)=0$ for all $t\ge 2R$, all time integrals appearing below are effectively taken over the finite interval $[0,2R]$.

\begin{lemma}\label{s4:L2}
	Suppose that $u$ is a very weak solution to \eqref{s1:1}, and let $\Phi$ be as in \eqref{s5:17}. Assume $u\in L^1_{\text{loc}}([0,\infty); X_\alpha)$ with $\alpha=\delta/R>0$. Then we have
	\begin{align}\label{s4:20}
		\int_0^\infty\mathcal{L}_{\Delta_{\mathcal{G}}u}(\Phi)dt=\int_0^\infty\mathcal{L}_{\Delta_{\mathcal{G}}\Phi}(u)dt
	\end{align}
\end{lemma}
\begin{proof}
	We first prove that
	$$\int_0^\infty\int_{\mathcal{G}}\Phi(x,t) \Delta_{\mathcal{V}}u(x,t) d\mu_{\mathcal{V}} dt=\int_0^\infty\int_{\mathcal{G}}u(x,t) \Delta_{\mathcal{V}}\Phi(x,t) d\mu_{\mathcal{V}} dt.$$
	Indeed, by \eqref{s4:16}, for all $(x,t)\in\mathcal{G}\times[0,\infty)$
	$$|\Phi(x,t)|\leq \le|\psi\le(\frac{\tilde{d}(x,x_0)-j}{R}\ri)\ri|\leq Ce^{-\alpha\tilde{d}(x,x_0)}\leq Ce^{-\alpha d(x,x_0)}.$$
	Thus, for any $t\in(0,2R)$, by \eqref{s3:1}-(v), 
	\begin{align*}
		\le|	\int_{\mathcal{G}}\Phi(x,t) \Delta_{\mathcal{V}}u(x,t) d\mu_{\mathcal{V}}\ri|&\leq\sum_{x\in\mathcal{V}}\sum_{y\sim x}\o(x,y)\le(|u(x,t)|+|u(y,t)|\ri)|\Phi(x,t)|\\
		&\leq \sum_{x\in\mathcal{V}}\sum_{y\sim x}\o(x,y)|u(x,t)||\Phi(x,t)|+\sum_{x\in\mathcal{V}}\sum_{y\sim x}\o(x,y)|u(y,t)||\Phi(x,t)|\\
		&\leq C \sum_{x\in\mathcal{V}}\mu_{\mathcal{V}}(x)|u(x,t)||\Phi(x,t)|+C\sum_{x\in\mathcal{V}}\sum_{y\sim x}\o(x,y)|u(y,t)|e^{-\alpha d(x,x_0)}\\
		&\leq C \sum_{x\in\mathcal{V}}\mu_{\mathcal{V}}(x)|u(x,t)|e^{-\alpha d(x,x_0)}+C\sum_{x\in\mathcal{V}}\sum_{y\sim x}\o(x,y)|u(y,t)|e^{-\alpha d(y,x_0)+\frac{\delta}{2}}\\
		&\leq C \sum_{x\in\mathcal{V}}\mu_{\mathcal{V}}(x)|u(x,t)|e^{-\alpha d(x,x_0)}+C e^{\frac{\delta}{2}} \sum_{y\in\mathcal{V}}\mu_{\mathcal{V}}(y)|u(y,t)|e^{-\alpha d(y,x_0)}\\
		&\leq C \sum_{x\in\mathcal{V}}\mu_{\mathcal{V}}(x)|u(x,t)|e^{-\alpha d(x,x_0)},
	\end{align*}
	where the factor $e^{\frac{\delta}{2}}$ follows from
	$d(x,y)\leq j$ and $R\geq 2j$. Moreover, the symmetry $\omega(x,y)=\omega(y,x)$ allows us to exchange the roles of $x$ and $y$ in the double summation. Since $u\in L^1_{\text{loc}}([0,\infty); X_\alpha)$, we have that
	\begin{align*}
	  \le|\int_0^\infty\int_{\mathcal{G}}\Phi(x,t) \Delta_{\mathcal{V}}u(x,t) d\mu_{\mathcal{V}} dt\ri|&\leq \int_0^{2R}\le|\int_{\mathcal{G}}\Phi(x,t) \Delta_{\mathcal{V}}u(x,t) d\mu_{\mathcal{V}}\ri| dt\\
	  &\leq C\int_0^{2R}\le\{\sum_{x\in\mathcal{V}}\mu_{\mathcal{V}}(x)|u(x,t)|e^{-\alpha d(x,x_0)}\ri\} dt<C.
	\end{align*}
	This implies that 
	$$\int_0^\infty\int_{\mathcal{G}}\Phi(x,t) \Delta_{\mathcal{V}}u(x,t) d\mu_{\mathcal{V}} dt$$
	converges absolutely, and moreover
	$$\int_0^\infty\le\{ \sum_{x\in\mathcal{V}}\sum_{y\sim x}\o(x,y)\le(|u(x,t)|+|u(y,t)|\ri)|\Phi(x,t)|\ri\}dt<\infty.$$
	Therefore, by Fubini's theorem, the order of integration and summation may be interchanged in the computations below.
	\begin{align*}
		&\int_0^\infty\int_{\mathcal{G}}\Phi(x,t) \Delta_{\mathcal{V}}u(x,t) d\mu_{\mathcal{V}} dt\\
		&=	\int_0^\infty \le\{\sum_{x\in\mathcal{V}}\sum_{y\sim x}\o(x,y)\le(u(y,t)-u(x,t)\ri)\Phi(x,t)\ri\}dt\\
		&=\int_0^\infty \le\{\sum_{x\in\mathcal{V}}\sum_{y\sim x}\o(x,y)u(y,t)\Phi(x,t)\ri\}dt-\int_0^\infty\le\{ \sum_{x\in\mathcal{V}}\sum_{y\sim x}\o(x,y)u(x,t)\Phi(x,t)\ri\}dt\\
		&=\int_0^\infty \le\{\sum_{y\in\mathcal{V}}\sum_{x\sim y}\o(x,y)u(x,t)\Phi(y,t)\ri\}dt-\int_0^\infty\le\{ \sum_{x\in\mathcal{V}}\sum_{y\sim x}\o(x,y)u(x,t)\Phi(x,t)\ri\}dt\\
		&= \int_0^\infty\le\{ \sum_{x\in\mathcal{V}}\sum_{y\sim x}\o(x,y)u(x,t)\le(\Phi(y,t)-\Phi(x,t)\ri)\ri\}dt\\
		&=\int_0^\infty\int_{\mathcal{G}}u(x,t) \Delta_{\mathcal{V}}\Phi(x,t) d\mu_{\mathcal{V}} dt.
	\end{align*}
	
	Next, in order to establish \eqref{s4:20}, it is enough to verify
	\begin{equation}\label{s4:21}
		\int_0^\infty \le\{\sum_{e\in\mathcal{E}}\int_{0}^{l_e}\Phi_e(x,t)u_{e}^{\prime\prime}(x,t)dx \ri\}dt=\int_0^\infty \le\{\sum_{e\in\mathcal{E}}\int_{0}^{l_e}u_e(x,t)\Phi_{e}^{\prime\prime}(x,t)dx \ri\}dt.
	\end{equation}
	By \eqref{s4:16}, the assumption $u\in L^1_{\rm loc}([0,\infty);X_\alpha)$, and the fact that $R\geq 2j$, we obtain 
	\begin{equation*}
		C\geq \int_0^{2R} \le\{\sum_{e\in\mathcal{E}}\int_{0}^{l_e}|u_e(x,t)|e^{-\alpha d(x,x_0)}dx\ri\}dt\geq \int_0^{2R} \le\{\sum_{e\in\mathcal{E}}\int_{0}^{l_e}|u_e(x,t)|e^{-\alpha \tilde{d}(x,x_0)}e^{-\frac{\delta}{2}}dx\ri\}dt,
	\end{equation*}
	where we used $\tilde d(x,x_0)\leq d(x,x_0)+j$ and $\alpha j=\delta j/R\leq \delta/2$. This implies that
	\begin{equation}\label{s4:22}
		\int_0^{2R} \le\{\sum_{e\in\mathcal{E}}\int_{0}^{l_e}|u_e(x,t)|e^{-\alpha \tilde{d}(x,x_0)}dx\ri\}dt\leq C.
	\end{equation}
	By the formula for integration by parts twice, for any $t\in(0,2R)$,
	\begin{align*}
		\nonumber	&\sum_{e\in\mathcal{E}}\int_{0}^{l_e}\Phi_e(x,t)u_{e}^{\prime\prime}(x,t)dx\\
		&= -\sum_{e\in{\mathcal{E}}}\int_{0}^{l_e} u_e^{\prime}(x,t)\Phi_e^{\prime}(x,t) dx+\sum_{e\in{\mathcal{E}}}\le(u_e^\prime(l_e,t)\Phi_e(l_e,t)-u_e^\prime(0,t)\Phi_e(0,t)\ri)\\
		\nonumber&=\sum_{e\in{\mathcal{E}}}\int_{0}^{l_e} u_e(x,t)\Phi_e^{\prime\prime}(x,t) dx-\sum_{e\in{\mathcal{E}}}\le(u_e(l_e,t)\Phi_e^\prime(l_e,t)-u_e(0,t)\Phi_e^\prime(0,t)\ri)\\
		\nonumber&\quad+\sum_{e\in{\mathcal{E}}}\le(u_e^\prime(l_e,t)\Phi_e(l_e,t)-u_e^\prime(0,t)\Phi_e(0,t)\ri)\\
		&=\sum_{e\in{\mathcal{E}}}\int_{0}^{l_e} u_e(x,t)\Phi_e^{\prime\prime}(x,t) dx-\sum_{x\in\mathcal{V}}u(x,t)[\mathcal{K}(\Phi)](x,t)+\sum_{x\in\mathcal{V}}\Phi(x,t)[\mathcal{K}(u)](x,t),
	\end{align*}
	where the last equality follows from the continuity condition at vertices and the definition of the Kirchhoff operator $\mathcal K$.
	By Definition \ref{s2:D1}, we know that for each $t\in(0,2R)$ and all $x\in\mathcal{V}$, $[\mathcal{K}(u)](x,t)=0$. We claim that $[\mathcal{K}(\Phi)](x,t)=0$. In fact, for all $x\in\mathcal{V}$, for every $e\ni x$ with $e\in\mathcal{E}$, it follows from $(i)$ of Proposition \ref{s4:P1} that $\tilde{d}_e^{\prime}(x,x_0)=0$. By chain rule, we have
	\begin{equation*}
		\Phi_e^\prime(x,t)=\phi^s\le(\frac{t}{R}\ri)\psi_e^\prime\le(\frac{\tilde{d}(x,x_0)-j}{R}\ri)\frac{1}{R}\tilde{d}_e^\prime(x,x_0)=0,\quad\forall\ x\in \mathcal{V},\ e\in\mathcal{E},\ t\in(0,2R).
	\end{equation*}
	Therefore,
	$$	[\mathcal{K}(\Phi)](x,t)=\sum_{e\ni x}\frac{d\Phi_e(x,t)}{dn}=0,$$
	which yields that
	\begin{equation}\label{s4:23}
		\sum_{e\in\mathcal{E}}\int_{0}^{l_e}\Phi_e(x,t)u_{e}^{\prime\prime}(x,t)dx=\sum_{e\in{\mathcal{E}}}\int_{0}^{l_e} u_e(x,t)\Phi_e^{\prime\prime}(x,t) dx.
	\end{equation}
	It then follows from \eqref{s4:24}, \eqref{s4:22} and \eqref{s4:23} that
	\begin{align*}
		\le|	\int_0^\infty \le\{\sum_{e\in\mathcal{E}}\int_{0}^{l_e}\Phi_e(x,t)u_{e}^{\prime\prime}(x,t)dx \ri\}dt\ri|&\leq \int_0^{2R} \le\{\sum_{e\in{\mathcal{E}}}\int_{0}^{l_e} |u_e(x,t)|\le|\Phi_e^{\prime\prime}(x,t) \ri|dx\ri\}dt\\
		&\leq \frac{C}{R} \int_0^{2R} \le\{\sum_{e\in\mathcal{E}}\int_{0}^{l_e}|u_e(x,t)|\phi^s\le(\frac{t}{R}\ri)e^{\frac{-\delta \tilde{d}(x,x_0)}{R}}\textbf{1}_{\tilde{B}_{R+j}^c}(x)dx\ri\}dt\\
		&\leq \frac{C}{R_0} \int_0^{2R} \le\{\sum_{e\in\mathcal{E}}\int_{0}^{l_e}|u_e(x,t)|e^{-\alpha\tilde{d}(x,x_0)}dx\ri\}dt\\
		&\leq C.
	\end{align*}
	Hence, 
	$$	\int_0^\infty \le\{\sum_{e\in\mathcal{E}}\int_{0}^{l_e}\Phi_e(x,t)u_{e}^{\prime\prime}(x,t)dx \ri\}dt$$
	converges absolutely. The same estimate also shows that
	\[
	\int_0^\infty
	\left|
	\sum_{e\in\mathcal E}
	\int_0^{l_e}
	u_e(x,t)\Phi_e''(x,t)dx
	\right|dt<\infty,
	\]
	so both sides of \eqref{s4:23} are integrable with respect to $t$.
	Therefore, integrating \eqref{s4:23} over $(0,\infty)$ yields \eqref{s4:21}.
\end{proof}

Assume that $u_0\in X_\alpha$ and that $u\in L^1_{\text{loc}}([0,\infty); X_\alpha)$ is a very weak solution of \eqref{s1:1}. For the function $\Phi$ given in \eqref{s5:17}, $\Phi(x,\cdot)\in C^1([0,\infty))$ for each $x\in\mathcal{G}$. Since 
$$\phi\le(\frac{t}{R}\ri)=0$$
for all $t\geq 2R$, it follows that
$$
\Phi(x,t)\equiv0,
\quad\forall\ (x,t)\in\mathcal G\times[2R,\infty).
$$
Moreover, by Lemma \ref{s5:L3} and \eqref{s4:16}, there exists a constant $C>0$ such that
\begin{equation}\label{s5:12}
	\begin{cases}
			|\Phi(x,t)|\leq Ce^{-\alpha d(x,x_0)},\quad |\Phi_t(x,t)|\leq Ce^{-\alpha d(x,x_0)}, \\
			|\Delta_{\mathcal{E}}\Phi(x,t)|\leq Ce^{-\alpha d(x,x_0)},\quad |\Delta_{\mathcal{V}}\Phi(x,t)|\leq Ce^{-\alpha d(x,x_0)},
	\end{cases}
\end{equation}
for all $(x,t)\in\mathcal{G}\times[0,\infty)$. 

Although $\Phi$ is not compactly supported with respect to the spatial variable $x$, it can still be employed as a test function in \eqref{s2:9}. 
\begin{lemma}\label{s4:L3}
	For any $R\geq R_0$, let $u$ be a very weak solution of \eqref{s1:1}, and let $\Phi$ be given in \eqref{s5:17}. Then $\Phi$ satisfies
		\begin{align}\label{s4:30}
		\nonumber&\int_0^\infty \int_{\mathcal{G}}u(x,t)\Phi_t(x,t)d\mu_{\mathcal{G}}dt+\int_0^\infty \int_{\mathcal{G}}\Phi(x,t)\Delta_{\mathcal{G}} u(x,t)d\mu_{\mathcal{G}}dt\\
		&\quad+\int_\mathcal{G}u_0(x)\Phi(x,0)d\mu_{\mathcal{G}}+\int_0^\infty \int_{\mathcal{G}}h(x,t)|u(x,t)|^\sigma\Phi(x,t)d\mu_{\mathcal{G}} dt\leq0.
	\end{align} 
\end{lemma}
\begin{proof}
	To justify this, let
	\[
	\eta_k(x) := \phi\Big(\frac{\tilde d(x,x_0)}{k}\Big),\quad x\in \mathcal G,
	\]
	where $\phi$ is the $C^2$ cut-off function defined in \eqref{s4:6}. Then 
	$$0\leq \eta_k\leq 1,\quad \eta_k\equiv 1,\ \text{on}\ \tilde B_k,\quad \eta_k\equiv 0, \text{ on } \mathcal G\setminus \tilde B_{2k}.$$ 
	Define
	\[
	\Phi_k(x,t) := \eta_k(x)\Phi(x,t).
	\]
	By construction, each $\Phi_k$ is an admissible test function in the sense of Definition \ref{s2:D1}. For this sequence, we have
	\begin{align}\label{s4:19}
		\nonumber&\int_0^\infty \int_{\tilde B_{2k}}u(x,t)\Phi_{k,t}(x,t)d\mu_{\mathcal{G}}dt+ \int_0^\infty \int_{\tilde B_{2k}}\Phi_k(x,t)\Delta_{\mathcal{G}} u(x,t)d\mu_{\mathcal{G}}dt\\
		&\quad+\int_{\tilde B_{2k}}u_0(x)\Phi_k(x,0)d\mu_{\mathcal{G}}	+\int_0^\infty \int_{\tilde B_{2k}}h(x,t)|u(x,t)|^\sigma\Phi_k(x,t)d\mu_{\mathcal{G}} dt\leq0,
	\end{align}
	for every $k\in\mathbb{N}$. 
	
	Since $0\leq \eta_k\leq 1$ and $\eta_k(x)\to1$ for every $x\in\mathcal G$, we have
	$$
	\Phi_k(x,t)\to\Phi(x,t),
	\quad
	\Phi_{k,t}(x,t)\to\Phi_t(x,t),\quad k\to\infty,
	$$
	pointwise on $\mathcal G\times[0,\infty)$. Moreover, by the decay estimates \eqref{s5:12},
	$$
	|u(x,t)\Phi_{k,t}(x,t)|
	\leq
	C|u(x,t)|e^{-\alpha d(x,x_0)},$$
	$$|u_0(x)\Phi_k(x,0)|
	\leq
	C|u_0(x)|e^{-\alpha d(x,x_0)}.
	$$
	Since $u\in L^1_{\mathrm{loc}}([0,\infty);X_\alpha)$, $u_0\in X_\alpha$, and $\Phi(\cdot,t)\equiv0$ for $t\ge 2R$, the first and third terms are dominated by integrable functions independent of $k$. 
	
	For the second term, using Lemma \ref{s4:L2} and \eqref{s5:12}, we obtain
	\begin{align*}
		\int_0^\infty\int_{\mathcal G}|\Phi(x,t)\Delta_{\mathcal G}u(x,t)|d\mu_{\mathcal G}dt&\leq\int_0^{2R}\int_{\mathcal G}|u(x,t)||\Delta_{\mathcal G}\Phi(x,t)|d\mu_{\mathcal G}dt\\
		&\leq C \int_0^{2R}\int_{\mathcal G} |u(x,t)|e^{-\alpha d(x,x_0)}d\mu_{\mathcal G}dt\\
		&\leq C,
	\end{align*}
	that is, 
	\[
	|\Phi \Delta_{\mathcal G} u| \in L^1(\mathcal G \times [0,2R]),
	\]
	and as $|\Phi_k \Delta_{\mathcal G} u| \leq |\Phi \Delta_{\mathcal G} u|$, the integrand of the second term is also dominated by an integrable function independent of $k$. 
	
	Finally, observing that the fourth term is non-decreasing, the Monotone Convergence Theorem yields
	$$
	\lim_{k\to\infty} \int_0^\infty \int_{\tilde B_{2k}}h(x,t)|u(x,t)|^\sigma\Phi_k(x,t)d\mu_{\mathcal{G}} dt
	= \int_0^\infty \int_{\mathcal G} h(x,t)|u(x,t)|^\sigma \Phi(x,t)\, d\mu_{\mathcal G} dt.
	$$
	Consequently, passing to the limit as $k\to\infty$ in \eqref{s4:19}, we obtain \eqref{s4:30}. In particular, since the first three terms are finite, the fourth term is finite as well. Therefore, we can choose $\Phi$ as the test function in \eqref{s2:9}, which satisfies \eqref{s4:30}.
\end{proof}

In view of the preceding key estimates, we are now ready to present the proof of Theorem \ref{s2:T4}.\\

\noindent$\textbf{\emph{Proof of Theorem \ref{s2:T4}.}}$  Let $u\in L^1_{\text{loc}}([0,\infty); X_\alpha)$ be a very weak solution of \eqref{s1:1}. By Lemma \ref{s4:L3}, one has
\begin{align*}
  \int_0^\infty \int_{\mathcal{G}}h(x,t)|u(x,t)|^\sigma\Phi(x,t)d\mu_{\mathcal{G}} &\leq-\int_0^\infty\int_{\mathcal{G}}\Phi(x,t)\Delta_{\mathcal{G}} u(x,t)d\mu_{\mathcal{G}} dt
  -\int_0^\infty \int_{\mathcal{G}}u(x,t)\Phi_t(x,t)d\mu_{\mathcal{G}}dt\\
  &\quad-\int_\mathcal{G}u_0(x)\Phi(x,0)d\mu_{\mathcal{G}}.
\end{align*}
It then follows from Lemma \ref{s4:L2} that
\begin{align}\label{s5:22}
	\nonumber&\int_0^\infty \int_{\mathcal{G}}h(x,t)|u(x,t)|^\sigma\Phi(x,t)d\mu_{\mathcal{G}}\\ \nonumber&\leq-\int_0^\infty\int_{\mathcal{G}}u(x,t)\Delta_{\mathcal{G}} \Phi(x,t)d\mu_{\mathcal{G}} dt
	-\int_0^\infty \int_{\mathcal{G}}u(x,t)\Phi_t(x,t)d\mu_{\mathcal{G}}dt-\int_\mathcal{G}u_0(x)\Phi(x,0)d\mu_{\mathcal{G}}\\
	\nonumber&= -\int_0^\infty\le\{\sum_{x\in\mathcal{V}}\mu_{\mathcal{V}}(x)u(x,t) \Delta_{\mathcal{V}}\Phi(x,t) \ri\}dt-\int_0^\infty \le\{\sum_{e\in\mathcal{E}}\int_{0}^{l_e}u_e(x,t)\Phi_{e}^{\prime\prime}(x,t)dx \ri\}dt\\
	&\quad-\int_0^\infty \int_{\mathcal{G}}u(x,t)\Phi_t(x,t)d\mu_{\mathcal{G}}dt-\int_\mathcal{G}u_0(x)\Phi(x,0)d\mu_{\mathcal{G}}.
\end{align}

 We next estimate each term on the right-hand side of \eqref{s5:22}. Using \eqref{s4:17}, \eqref{s4:24} and \eqref{s4:25}, applying Young’s inequality with a sufficiently small $\epsilon>0$, we obtain from \eqref{s3:11} that
\begin{align}\label{s5:23}
	\nonumber\le|-\int_0^\infty\le\{\sum_{x\in\mathcal{V}}\mu_{\mathcal{V}}(x)u(x,t) \Delta_{\mathcal{V}}\Phi(x,t) \ri\}dt\ri|&\leq \int_0^\infty\le\{\sum_{x\in\mathcal{V}}\mu_{\mathcal{V}}(x)|u(x,t)| \le|\Delta_{\mathcal{V}}\Phi(x,t) \ri|\ri\}dt\\
	\nonumber&\leq \frac{C}{R}\int_0^\infty\le\{ \sum_{x\in\mathcal{V}}\mu_{\mathcal{V}}(x)|u(x,t)| \phi^s\le(\frac{t}{R}\ri)e^{\frac{-\delta d(x,x_0)}{R}}\textbf{1}_{ B_{R}^c}(x)\ri\}dt \\
	\nonumber&\leq \epsilon\int_0^\infty\le\{\sum_{x\in\mathcal{V}\cap B_{R}^c}\mu_{\mathcal{V}}(x) |u(x,t)|^\sigma h(x,t) \phi^s\le(\frac{t}{R}\ri)e^{\frac{-\delta d(x,x_0)}{R}}\ri\}dt \\
	\nonumber&\quad+\frac{C_\epsilon}{R^{\frac{\sigma}{\sigma-1}}}\int_0^\infty \le\{\sum_{x\in\mathcal{V}\cap B_{R}^c}\mu_{\mathcal{V}}(x) h^{-\frac{1}{\sigma-1}}(x,t) \phi^s\le(\frac{t}{R}\ri)e^{\frac{-\delta d(x,x_0)}{R}}\ri\}dt \\
	\nonumber&\leq  \epsilon C\int_0^\infty \le\{\sum_{x\in\mathcal{V}\cap B_{R}^c}\mu_{\mathcal{V}}(x) |u(x,t)|^\sigma h(x,t) \phi^s\le(\frac{t}{R}\ri)\psi\le(\frac{d(x,x_0)-j}{R}\ri)\ri\}dt \\
	\nonumber&\quad+\frac{C_\epsilon}{R^{\frac{\sigma}{\sigma-1}}}\int_0^{2R} \le\{\sum_{x\in\mathcal{V}\cap B_{R}^c}\mu_{\mathcal{V}}(x) h^{-\frac{1}{\sigma-1}}(x,t) e^{\frac{-\delta d(x,x_0)}{R}}\ri\}dt \\
	&\leq \frac{1}{4}\int_0^\infty\le\{\sum_{x\in\mathcal{V}}\mu_{\mathcal{V}}(x) |u(x,t)|^\sigma h(x,t) \Phi(x,t)\ri\}dt+C, 
\end{align}
and that
\begin{align}\label{s5:14}
	\nonumber&\le|-\int_0^\infty \le\{\sum_{e\in\mathcal{E}}\int_{0}^{l_e}u_e(x,t)\Phi_{e}^{\prime\prime}(x,t)dx \ri\}dt\ri|\\
	\nonumber&\leq 	 \frac{C}{R} \int_0^\infty\le\{\sum_{e\in\mathcal{E}}\int_{0}^{l_e} |u_e(x,t)| \phi^s\le(\frac{t}{R}\ri)e^{\frac{-\delta \tilde{d}(x,x_0)}{R}}\textbf{1}_{ \tilde{B}_{R+j}^c}(x)dx\ri\}dt\\
	\nonumber&\leq \epsilon \int_0^\infty\le\{\sum_{e\in \mathcal{E}}\int_{0}^{l_e} |u_e(x,t)|^\sigma h_e(x,t) \phi^s\le(\frac{t}{R}\ri)e^{\frac{-\delta \tilde{d}(x,x_0)}{R}}\textbf{1}_{ \tilde{B}_{R+j}^c}(x)dx\ri\}dt\\
	\nonumber&\quad +\frac{C_\epsilon}{R^{\frac{\sigma}{\sigma-1}}}\int_0^{\infty}\le\{\sum_{e\in\mathcal{E}}\int_{0}^{l_e} h_e^{-\frac{1}{\sigma-1}}(x,t)\phi^{s}\le(\frac{t}{R}\ri)e^{\frac{-\delta \tilde{d}(x,x_0)}{R}}\textbf{1}_{ \tilde{B}_{R+j}^c}(x)dx\ri\}dt\\
	\nonumber&\leq  \epsilon  \int_0^\infty\le\{\sum_{e\in \mathcal{E}}\int_{0}^{l_e} |u_e(x,t)|^\sigma h_e(x,t) \phi^s\le(\frac{t}{R}\ri)e^{ \frac{-\delta \tilde{d}(x,x_0)}{R}}\textbf{1}_{ \tilde{B}_{R+j}^c}(x)dx\ri\}dt\\
	\nonumber&\quad +\frac{C_\epsilon}{R^{\frac{\sigma}{\sigma-1}}}\int_0^{\infty}\le\{\sum_{e\in\mathcal{E}}\int_{0}^{l_e} h_e^{-\frac{1}{\sigma-1}}(x,t)\phi^{s}\le(\frac{t}{R}\ri)e^{-\delta\frac{ d(x,x_0)-j}{R}}\textbf{1}_{ B_{R}^c}(x)dx\ri\}dt\\
	\nonumber&\leq \epsilon C \int_0^\infty\le\{\sum_{e\in\mathcal{E}}\int_{0}^{l_e} |u_e(x,t)|^\sigma h_e(x,t) \phi^s\le(\frac{t}{R}\ri) \psi_e\le(\frac{\tilde{d}(x,x_0)-j}{R}\ri)\textbf{1}_{ \tilde{B}_{R+j}^c}(x)dx\ri\}dt\\
	\nonumber&\quad+\frac{C_\epsilon}{R^{\frac{\sigma}{\sigma-1}}}\int_0^{2R}\le\{\sum_{e\in\mathcal{E}\cap B_{R}^c}\int_{0}^{l_e} h_e^{-\frac{1}{\sigma-1}}(x,t)e^{\frac{-\delta d(x,x_0)}{R}}dx\ri\}dt\\
	&\leq \frac{1}{4}\int_0^\infty\le\{\sum_{e\in\mathcal{E}}\int_{0}^{l_e} |u_e(x,t)|^\sigma h_e(x,t) \Phi_e(x,t)dx\ri\}dt+C,
\end{align}
where we have used \eqref{s4:16} and the fact that $\textbf{1}_{ \tilde{B}_{R+j}^c}\leq\textbf{1}_{ B_{R}^c}$. Indeed, if $x\in \tilde{B}_{R+j}^c$, then $\tilde{d}(x,x_0)\geq R+j$, and hence
$$d(x,x_0)\geq \tilde{d}(x,x_0)-j\geq R.$$
Thus, $x\in B_{R}^c$.

By \eqref{s4:26}, using Young’s inequality, we have 
\begin{align*}
	\nonumber&\le|-\int_0^\infty \int_{\mathcal{G}}u(x,t)\Phi_t(x,t)d\mu_{\mathcal{G}}dt\ri|\\
	\nonumber&\leq \frac{C}{R}\int_0^\infty \int_{\mathcal{G}}|u(x,t)| \phi^{s-1}\le(\frac{t}{R}\ri)e^{\frac{-\delta d(x,x_0)}{R}}\textbf{1}_{Q_R}(x,t) d\mu_{\mathcal{G}}dt\\
	\nonumber&\leq \epsilon \int_R^{2R} \le\{\sum_{x\in\mathcal{V}}\mu_{\mathcal{V}}(x)|u(x,t)|^\sigma h(x,t) \phi^{s}\le(\frac{t}{R}\ri)e^{\frac{-\delta d(x,x_0)}{R}}\ri\}dt\\
	\nonumber&\quad+ \epsilon\int_R^{2R}\le\{\sum_{e\in\mathcal{E}}\int_{0}^{l_e}|u_e(x,t)|^\sigma h_e(x,t)\phi^{s}\le(\frac{t}{R}\ri)e^{\frac{-\delta d(x,x_0)}{R}}dx \ri\}dt\\
	\nonumber&\quad+\frac{C_\epsilon}{R^{\frac{\sigma}{\sigma-1}}}\int_R^{2R} \le\{\sum_{x\in\mathcal{V}}\mu_{\mathcal{V}}(x) h^{-\frac{1}{\sigma-1}}(x,t) \phi^{s-\frac{\sigma}{\sigma-1}}\le(\frac{t}{R}\ri)e^{\frac{-\delta d(x,x_0)}{R}} \ri\} dt\\
	\nonumber&\quad+\frac{C_\epsilon}{R^{\frac{\sigma}{\sigma-1}}}\int_R^{2R}\le\{\sum_{e\in\mathcal{E}}\int_{0}^{l_e}h_e^{-\frac{1}{\sigma-1}}(x,t)\phi^{s-\frac{\sigma}{\sigma-1}}\le(\frac{t}{R}\ri)e^{\frac{-\delta d(x,x_0)}{R}}dx \ri\}dt.
	\end{align*}
	Noting that $s>\sigma/(\sigma-1)$, we further obtain from \eqref{s3:15}, \eqref{s3:11} and \eqref{s4:17} that
	\begin{align}\label{s5:13}
			\nonumber&\le|-\int_0^\infty \int_{\mathcal{G}}u(x,t)\Phi_t(x,t)d\mu_{\mathcal{G}}dt\ri|\\
	\nonumber&\leq \epsilon \int_R^{2R} \le\{\sum_{x\in\mathcal{V}}\mu_{\mathcal{V}}(x)|u(x,t)|^\sigma h(x,t) \phi^{s}\le(\frac{t}{R}\ri)e^{-\delta \frac{d(x,x_0)-j}{R}}\ri\}dt\\
	\nonumber&\quad+ \epsilon\int_R^{2R}\le\{\sum_{e\in\mathcal{E}}\int_{0}^{l_e}|u_e(x,t)|^\sigma h_e(x,t)\phi^{s}\le(\frac{t}{R}\ri)e^{-\delta \frac{\tilde{d}(x,x_0)-j}{R}}dx \ri\}dt\\
	\nonumber&\quad+\frac{C_\epsilon}{R^{\frac{\sigma}{\sigma-1}}}\int_R^{2R} \le\{\sum_{x\in\mathcal{V}}\mu_{\mathcal{V}}(x) h^{-\frac{1}{\sigma-1}}(x,t) e^{\frac{-\delta d(x,x_0)}{R}} \ri\}dt+\frac{C_\epsilon}{R^{\frac{\sigma}{\sigma-1}}}\int_R^{2R}\le\{\sum_{e\in\mathcal{E}}\int_{0}^{l_e}h_e^{-\frac{1}{\sigma-1}}(x,t)e^{\frac{-\delta d(x,x_0)}{R}}dx \ri\}dt\\
	\nonumber&\leq \epsilon C \int_R^{2R} \le\{\sum_{x\in\mathcal{V}}\mu_{\mathcal{V}}(x)|u(x,t)|^\sigma h(x,t) \phi^{s}\le(\frac{t}{R}\ri)\psi\le(\frac{d(x,x_0)-j}{R}\ri)\ri\}dt\\
	&\nonumber\quad+ \epsilon C\int_R^{2R}\le\{\sum_{e\in\mathcal{E}}\int_{0}^{l_e}|u_e(x,t)|^\sigma h_e(x,t)\phi^{s}\le(\frac{t}{R}\ri)\psi_e\le(\frac{\tilde{d}(x,x_0)-j}{R}\ri)dx \ri\}dt\\
	\nonumber&\quad+\frac{C_\epsilon}{R^{\frac{\sigma}{\sigma-1}}}\int_R^{2R} \le\{\sum_{x\in\mathcal{V}\cap B_R}\mu_{\mathcal{V}}(x) h^{-\frac{1}{\sigma-1}}(x,t) e^{\frac{-\delta d(x,x_0)}{R}}\ri\} dt+\frac{C_\epsilon}{R^{\frac{\sigma}{\sigma-1}}}\int_0^{2R} \le\{\sum_{x\in\mathcal{V}\cap B_R^c}\mu_{\mathcal{V}}(x) h^{-\frac{1}{\sigma-1}}(x,t) e^{\frac{-\delta d(x,x_0)}{R}}\ri\}dt\\
	\nonumber&\quad+\frac{C_\epsilon}{R^{\frac{\sigma}{\sigma-1}}}\int_R^{2R}\le\{\sum_{e\in\mathcal{E}\cap B_{R}}\int_{0}^{l_e}h_e^{-\frac{1}{\sigma-1}}(x,t)e^{\frac{-\delta d(x,x_0)}{R}}dx \ri\}dt+\frac{C_\epsilon}{R^{\frac{\sigma}{\sigma-1}}}\int_0^{2R}\le\{\sum_{e\in \mathcal{E}\cap B_{R}^c}\int_{0}^{l_e}h_e^{-\frac{1}{\sigma-1}}(x,t)e^{\frac{-\delta d(x,x_0)}{R}}dx \ri\}dt\\
	&\leq \frac{1}{4} \int_R^{2R}\int_\mathcal{G} |u(x,t)|^\sigma h(x,t)\Phi(x,t)d\mu_{\mathcal{G}}dt+C,
\end{align}
for $\epsilon>0$ sufficiently small.  

Finally, by the definition of $\psi$, we get
\begin{align}\label{s5:15}
	\nonumber&-\int_\mathcal{G}u_0(x)\Phi(x,0)d\mu_{\mathcal{G}}\\
	\nonumber&=-\sum_{x\in\mathcal{V}}\mu_{\mathcal{V}}(x)u_0(x)\Phi(x,0)-\sum_{e\in\mathcal{E}}\int_{0}^{l_e}u_{0,e}(x)\Phi_e(x,0)dx\\
	\nonumber&=-\sum_{x\in\mathcal{V}}\mu_{\mathcal{V}}(x)u_0(x)\psi\le(\frac{d(x,x_0)-j}{R}\ri)-\sum_{e\in\mathcal{E}}\int_{0}^{l_e}u_{0,e}(x)\psi_e\le(\frac{\tilde{d}(x,x_0)-j}{R}\ri)dx\\
	\nonumber&= -\sum_{x\in\mathcal{V}}\mu_{\mathcal{V}}(x)u_0^+(x)\psi\le(\frac{d(x,x_0)-j}{R}\ri)-\sum_{x\in\mathcal{V}}\mu_{\mathcal{V}}(x)u_0^-(x)\psi\le(\frac{d(x,x_0)-j}{R}\ri)\\
	\nonumber&\quad-\sum_{e\in\mathcal{E}}\int_{0}^{l_e}u_{0,e}^{+}(x)\psi_e\le(\frac{\tilde{d}(x,x_0)-j}{R}\ri)dx-\sum_{e\in\mathcal{E}}\int_{0}^{l_e}u_{0,e}^{-}(x)\psi_e\le(\frac{\tilde{d}(x,x_0)-j}{R}\ri)dx\\
	\nonumber&\leq -\sum_{x\in\mathcal{V}\cap  B_{R+j}}\mu_{\mathcal{V}}(x)u_0^+(x)-\sum_{x\in\mathcal{V}}\mu_{\mathcal{V}}(x)u_0^-(x)-\sum_{e\in\mathcal{E}\cap \tilde{B}_{R+j}}\int_{0}^{l_e}u_{0,e}^{+}(x)dx-\sum_{e\in\mathcal{E}}\int_{0}^{l_e}u_{0,e}^{-}(x)dx\\
	\nonumber&\leq-\sum_{x\in\mathcal{V}\cap  B_{R}}\mu_{\mathcal{V}}(x)u_0^+(x)-\sum_{e\in\mathcal{E}\cap B_{R}}\int_{0}^{l_e}u_{0,e}^{+}(x)dx-\int_\mathcal{G}u_0^-(x)d\mu_{\mathcal{G}}\\
	&=-\int_{B_{R}}u_0^+(x)d\mu_{\mathcal{G}}-\int_\mathcal{G}u_0^-(x)d\mu_{\mathcal{G}},
\end{align}
where we have used $\psi\leq1$ and $\psi\equiv1$ on $\tilde{B}_{R+j}$. Combining the above estimates with  \eqref{s5:22}, we deduce  that 
\begin{align*}
	&\int_0^\infty \int_{\mathcal{G}}h(x,t)|u(x,t)|^\sigma\Phi(x,t)d\mu_{\mathcal{G}} dt\\
	&\leq \frac{1}{4}\int_0^\infty \int_{\mathcal{G}}h(x,t)|u(x,t)|^\sigma\Phi(x,t)d\mu_{\mathcal{G}} dt+\frac{1}{4} \int_R^{2R}\int_\mathcal{G} |u(x,t)|^\sigma h(x,t)\Phi(x,t)d\mu_{\mathcal{G}}dt\\
	&\quad+C-\int_{B_{R}}u_0^+(x)d\mu_{\mathcal{G}}-\int_\mathcal{G}u_0^-(x)d\mu_{\mathcal{G}}.
\end{align*}
Consequently, 
\begin{align*}
	\int_0^R \int_{\tilde{B}_{R+j}}h(x,t)|u(x,t)|^\sigma d\mu_{\mathcal{G}} dt&\leq \int_0^\infty \int_{\mathcal{G}}h(x,t)|u(x,t)|^\sigma\Phi(x,t)d\mu_{\mathcal{G}} dt\\
	&\leq C-C\le(\int_{B_{R}}u_0^+(x)d\mu_{\mathcal{G}}+\int_\mathcal{G}u_0^-(x)d\mu_{\mathcal{G}}\ri),
\end{align*}
where $C>0$ is a constant independent of $R$. Since $u_0$ satisfies \eqref{s3:12}, we take the limit as $R\ra\infty$. By Monotone Convergence Theorem, we conclude that
\begin{equation}\label{s5:16}
	\int_0^\infty \int_\mathcal{G}h(x,t)|u(x,t)|^{\sigma}d\mu_\mathcal{G}dt\leq C-C\int_\mathcal{G}u_0(x)d\mu_\mathcal{G}\leq C.
\end{equation}

It remains to show that $u\equiv 0$. To this end, we revisit the estimates above and pass to the limit as $R\to\infty$ in \eqref{s5:22}. By H\"{o}lder's inequality, and using \eqref{s3:15}, \eqref{s3:11} and \eqref{s4:26}, we get from \eqref{s5:13} that
\begin{align*}
	&\le|-\int_0^\infty\le\{\sum_{e\in\mathcal{E}}\int_{0}^{l_e}u_e(x,t)\Phi_{e,t}(x,t)dx \ri\}dt\ri|\\
	&\leq  \frac{C}{R} \sum_{e\in\mathcal{E}}\int_{0}^{l_e}\le\{\int_R^{2R}|u_e(x,t)|\phi^{s-1}\le(\frac{t}{R}\ri)e^{\frac{-\delta d(x,x_0)}{R}}dt \ri\}dx\\
	&\leq \frac{C}{R} \sum_{e\in\mathcal{E}}\int_{0}^{l_e}\le\{\int_R^{2R} 	|u_e(x,t)|^\sigma h_e(x,t)\phi^{s}\le(\frac{t}{R}\ri)e^{\frac{-\delta d(x,x_0)}{R}} dt\ri\}^{\frac{1}{\sigma}}\le\{\int_R^{2R}h_e^{-\frac{1}{\sigma-1}}(x,t)\phi^{s-\frac{\sigma}{\sigma-1}}\le(\frac{t}{R}\ri)e^{\frac{-\delta d(x,x_0)}{R}}dt\ri\}^{\frac{\sigma-1}{\sigma}}dx\\
	&\leq \frac{C}{R} \sum_{e\in\mathcal{E}} \le\{\int_{0}^{l_e} \int_R^{2R} 	|u_e(x,t)|^\sigma h_e(x,t)\phi^{s}\le(\frac{t}{R}\ri)e^{\frac{-\delta d(x,x_0)}{R}} dtdx\ri\}^\frac{1}{\sigma}\le\{\int_{0}^{l_e} \int_R^{2R}h_e^{-\frac{1}{\sigma-1}}(x,t)\phi^{s-\frac{\sigma}{\sigma-1}}\le(\frac{t}{R}\ri)e^{\frac{-\delta d(x,x_0)}{R}}dtdx\ri\}^\frac{\sigma-1}{\sigma}\\
	&\leq \frac{C}{R} \le\{\sum_{e\in\mathcal{E}}\int_{0}^{l_e} \le(\int_R^{2R} 	|u_e(x,t)|^\sigma h_e(x,t)\phi^{s}\le(\frac{t}{R}\ri)e^{\frac{-\delta d(x,x_0)}{R}} dt\ri)dx\ri\}^{\frac{1}{\sigma}}\\
	\nonumber&\quad\cdot\le\{\sum_{e\in\mathcal{E}}\int_{0}^{l_e} \le(\int_R^{2R} h_e^{-\frac{1}{\sigma-1}}(x,t)\phi^{s-\frac{\sigma}{\sigma-1}}\le(\frac{t}{R}\ri)e^{\frac{-\delta d(x,x_0)}{R}}dt\ri)dx \ri\}^{\frac{\sigma-1}{\sigma}}\\
	&\leq \frac{C}{R}\le\{\int_R^{2R}\le(\sum_{e\in\mathcal{E}}\int_{0}^{l_e}	|u_e(x,t)|^\sigma h_e(x,t)e^{\frac{-\delta d(x,x_0)}{R}}dx \ri)dt\ri\}^{\frac{1}{\sigma}}\le\{\int_R^{2R}\le(\sum_{e\in\mathcal{E}}\int_{0}^{l_e}	h_e^{-\frac{1}{\sigma-1}}(x,t)e^{\frac{-\delta d(x,x_0)}{R}}dx \ri)dt\ri\}^{\frac{\sigma-1}{\sigma}}\\
	&\leq C \le(\int_R^{2R}\le\{\sum_{e\in\mathcal{E}}\int_{0}^{l_e}	|u_e(x,t)|^\sigma h_e(x,t)dx \ri\}dt\ri)^{\frac{1}{\sigma}}.
\end{align*}
Similarly,
\begin{align*}
	\le|-\int_0^\infty\le\{\sum_{x\in\mathcal{V}}\mu_{\mathcal{V}}(x)u(x,t)\Phi_{t}(x,t)\ri\}dt\ri|\leq C \le(\int_R^{2R}\le\{\sum_{x\in\mathcal{V}}\mu_{\mathcal{V}}(x)h(x,t)|u(x,t)|^{\sigma}\ri\}dt\ri)^{\frac{1}{\sigma}}.
\end{align*}
In view of \eqref{s3:11},  we have from \eqref{s5:14} that
\begin{align*}
	&\le|-\int_0^\infty\le\{\sum_{e\in\mathcal{E}}\int_{0}^{l_e}u_e(x,t)\Phi_e^{\prime\prime}(x,t)dx \ri\}dt\ri|\\
	&\leq 	\frac{C}{R} \int_0^\infty\le\{ \sum_{e\in\mathcal{E}}\int_{0}^{l_e} |u_e(x,t)| \phi^s\le(\frac{t}{R}\ri)e^{\frac{-\delta \tilde{d}(x,x_0)}{R}}\textbf{1}_{\tilde{B}_{R+j}^c}(x)dx\ri\}dt\\
	&\leq \frac{C}{R}\le\{\int_0^{\infty}\le(\sum_{e\in \mathcal{E}\cap \tilde{B}_{R+j}^c}\int_{0}^{l_e}	|u_e(x,t)|^\sigma h_e(x,t)\phi^s\le(\frac{t}{R}\ri)e^{\frac{-\delta \tilde{d}(x,x_0)}{R}}dx \ri)dt\ri\}^{\frac{1}{\sigma}}\\
	\nonumber&\quad\cdot\le\{\int_0^{\infty}\le(\sum_{e\in \mathcal{E}\cap\tilde{B}_{R+j}^c}\int_{0}^{l_e}	h_e^{-\frac{1}{\sigma-1}}(x,t)\phi^s\le(\frac{t}{R}\ri)e^{\frac{-\delta \tilde{d}(x,x_0)}{R}}dx \ri)dt\ri\}^{\frac{\sigma-1}{\sigma}}\\
	&\leq \frac{C}{R}\le\{\int_0^{\infty}\le(\sum_{e\in \mathcal{E}\cap \tilde{B}_{R+j}^c}\int_{0}^{l_e}	|u_e(x,t)|^\sigma h_e(x,t)dx \ri)dt\ri\}^{\frac{1}{\sigma}}\le\{\int_0^{2R}\le(\sum_{e\in \mathcal{E}\cap B_{R}^c}\int_{0}^{l_e}	h_e^{-\frac{1}{\sigma-1}}(x,t)e^{\frac{-\delta d(x,x_0)}{R}}dx \ri)dt\ri\}^{\frac{\sigma-1}{\sigma}}\\
	&\leq C \le(\int_0^{\infty}\le\{\sum_{e\in \mathcal{E}\cap B_{R}^c}\int_{0}^{l_e}	|u_e(x,t)|^\sigma h_e(x,t)dx \ri\}dt\ri)^{\frac{1}{\sigma}},
\end{align*}
and from \eqref{s5:23} that
$$\le|-\int_0^\infty\int_{\mathcal{G}}u(x,t) \Delta_{\mathcal{V}}\Phi(x,t) d\mu_{\mathcal{V}} dt\ri|\leq C  \le(\int_0^{\infty}\le\{\sum_{x\in\mathcal{V}\cap B_{R}^c} \mu_{\mathcal{V}}(x)	h(x,t)|u(x,t)|^\sigma \ri\}dt\ri)^{\frac{1}{\sigma}},$$
respectively. By these estimates, together with \eqref{s5:15}, we deduce from \eqref{s5:22} that
\begin{align*}
	&\int_0^R \int_{\tilde{B}_{R+j}}h(x,t)|u(x,t)|^\sigma d\mu_{\mathcal{G}} dt\\
	&\leq C \le(\int_R^{2R}\le\{\sum_{x\in\mathcal{V}}\mu_{\mathcal{V}}(x)h(x,t)|u(x,t)|^{\sigma}\ri\}dt\ri)^{\frac{1}{\sigma}}+C \le(\int_R^{2R}\le\{\sum_{e\in\mathcal{E}}\int_{0}^{l_e}	|u_e(x,t)|^\sigma h_e(x,t)dx \ri\}dt\ri)^{\frac{1}{\sigma}}\\
	&\quad+C  \le(\int_0^{\infty}\le\{\sum_{x\in\mathcal{V}\cap B_{R}^c} \mu_{\mathcal{V}}(x)	h(x,t)|u(x,t)|^\sigma \ri\}dt\ri)^{\frac{1}{\sigma}}+C \le(\int_0^{\infty}\le\{\sum_{e\in \mathcal{E}\cap B_{R}^c}\int_{0}^{l_e}	|u_e(x,t)|^\sigma h_e(x,t)dx \ri\}dt\ri)^{\frac{1}{\sigma}}\\
	&\quad-\int_{B_{R}}u_0^+(x)d\mu_{\mathcal{G}}-\int_\mathcal{G}u_0^-(x)d\mu_{\mathcal{G}},
\end{align*}
which implies that
\begin{equation*}
	\int_0^\infty \int_\mathcal{G}h(x,t)|u(x,t)|^{\sigma}d\mu_\mathcal{G}dt\leq -\int_\mathcal{G}u_0(x)d\mu_\mathcal{G}\leq 0,
\end{equation*}
by \eqref{s3:12} and \eqref{s5:16}. Since the left-hand side is nonnegative, we must have
$$\int_0^\infty \int_\mathcal{G}h(x,t)|u(x,t)|^{\sigma}d\mu_\mathcal{G}dt=0.$$
Hence, $u\equiv0$ on $\mathcal{G}\times[0,\infty)$.  $\hfill\Box$\\

\noindent$\textbf{\emph{Proof of Corollary \ref{s2:C2}.}}$ Assume that condition \eqref{s3:7} holds. Since $h(x,t)\geq g(x)$ for all $(x,t)\in \mathcal{G}\times[0,\infty)$, we have $$h^{-\frac{1}{\sigma-1}}(x,t)  \leq g^{-\frac{1}{\sigma-1}}(x).$$
Then for any $R\geq R_0$, we estimate
\begin{equation*}
	\int_R^{2R} \le\{\sum_{x\in \mathcal{V}\cap B_{R}}\mu_{\mathcal{V}}(x) h^{-\frac{1}{\sigma-1}}(x,t) e^{-\alpha d(x,x_0)} \ri\}dt\leq \int_R^{2R} \le\{\sum_{x\in\mathcal{V}}\mu_{\mathcal{V}}(x) g^{-\frac{1}{\sigma-1}}(x) e^{-\alpha d(x,x_0)} \ri\}dt\leq CR^{\frac{1}{\sigma-1}}R\leq CR^{\frac{\sigma}{\sigma-1}},
\end{equation*}
and
\begin{equation*}
	\int_R^{2R}\le\{\sum_{e\in \mathcal{E}\cap B_{R}}\int_{0}^{l_e}	h_e^{-\frac{1}{\sigma-1}}(x,t)e^{-\alpha d(x,x_0)}dx \ri\}dt\leq  \int_R^{2R}\le\{\sum_{e\in \mathcal{E}}\int_{0}^{l_e}	g_e^{-\frac{1}{\sigma-1}}(x)e^{-\alpha d(x,x_0)}dx \ri\}dt\leq CR^{\frac{1}{\sigma-1}}R\leq CR^{\frac{\sigma}{\sigma-1}}.
\end{equation*}
Similarly, for the integrations on $B_R^c\times[0,2R]$, we obtain the same bound.
The result then follows directly from Theorem \ref{s2:T4}. $\hfill\Box$\\

\noindent$\textbf{\emph{Proof of Corollary \ref{s2:C3}.}}$  Apply Corollary \ref{s2:C2} with $g \equiv 1$. By \eqref{s3:8}, for sufficiently large $k_0>R_0$ and any $R\geq k_0$, we have
\begin{align*}
	\sum_{x\in\mathcal{V}}\mu_{\mathcal{V}}(x) g^{-\frac{1}{\sigma-1}}(x) e^{-\alpha d(x,x_0)}&=\sum_{x\in\mathcal{V}}\mu_{\mathcal{V}}(x) e^{-\alpha d(x,x_0)}\\
	&\leq \sum_{k=k_0}^{\infty}\sum_{x\in\mathcal{V}\cap \{B_{k+1}\setminus B_k\}}\mu_{\mathcal{V}}(x)e^{-\alpha d(x,x_0)}+\sum_{x\in\mathcal{V}\cap B_{k_0}}\mu_{\mathcal{V}}(x)e^{-\alpha d(x,x_0)}\\
	&\leq C\sum_{k=k_0}^{\infty}e^{-\delta\frac{k}{R}}\mu_\mathcal{V}(B_{k+1}\setminus B_k)+\mu_\mathcal{V}(B_{k_0})\\
	&\leq C\sum_{k=k_0}^{\infty}e^{-\delta\frac{k}{R}}k^{\frac{2-\sigma}{\sigma-1}}+C\\
	&\leq C\int_{R_0}^\infty e^{-\delta\frac{r}{R}}r^{\frac{2-\sigma}{\sigma-1}}dr+C\\
	&\leq CR^{\frac{1}{\sigma-1}},
\end{align*} 
and
\begin{align*}
	\sum_{e\in\mathcal{E}}\int_0^{l_e}g^{-\frac{1}{\sigma-1}}(x)e^{-\alpha d(x,x_0)}dx&=\sum_{e\in\mathcal{E}}\int_0^{l_e}e^{-\alpha d(x,x_0)}dx\\
	&\leq \sum_{k=k_0}^{\infty}\sum_{e\in\mathcal{E}\cap \{B_{k+1}\setminus B_k\}}\int_{0}^{l_e}e^{-\alpha d(x,x_0)}dx+\sum_{e\in\mathcal{E}\cap B_{k_0}}\int_{0}^{l_e}e^{-\alpha d(x,x_0)}dx\\
	&\leq C\sum_{k=k_0}^{\infty}e^{-\delta\frac{k}{R}}\mu_\mathcal{E}(B_{k+1}\setminus B_k)+\mu_\mathcal{E}(B_{k_0})\\
	&\leq CR^{\frac{1}{\sigma-1}},
\end{align*} 
where we have used $d(x,x_0)\geq k$ for $x\in B_{k+1}\setminus B_k$ and $\alpha=\delta/R$. The conclusion follows immediately from Corollary \ref{s2:C2}. $\hfill\Box$

\section{Proofs of the main results for the hyperbolic problem}\label{4}

In this section, we study the hyperbolic inequality \eqref{s1:2} and prove Theorems \ref{s2:T1} and \ref{s2:T2}. Since the arguments are similar to those used in Section \ref{3}, we only sketch the proofs.

\noindent$\textbf{\emph{Proof of Theorem \ref{s2:T1}.}}$ Let $\tau:\mathcal{G}\times[0,\infty)\to\mathbb{R}$ be defined by \eqref{s4:7} satisfying \eqref{s4:6}, and let $s>2\sigma/(\sigma-1)$. Suppose that $u:\mathcal{G}\times[0,\infty)\to\mathbb{R}$ is a nonnegative very weak solution to \eqref{s1:2}. Observe that 
$$\le(\tau^s\ri)_t(x,0)=0,\quad\forall\ x\in\mathcal{G}.$$ 
Testing \eqref{s2:10} by $\tau^s$, it then follows from Lemma \ref{s5:L5} that
\begin{align}\label{s5:19}
	\nonumber&\int_0^\infty \int_{\mathcal{G}}h(x,t)u^\sigma(x,t)\tau^s(x,t)d\mu_{\mathcal{G}} dt\\
	\nonumber&\leq
-\int_0^\infty \int_{\mathcal{G}}\tau^s(x,t)\Delta_{\mathcal{G}} u(x,t)d\mu_{\mathcal{G}}dt+\int_0^\infty \int_{\mathcal{G}}\le(\tau^s\ri)_{tt}(x,t) u(x,t)d\mu_{\mathcal{G}}dt\\
\nonumber&\quad-\int_\mathcal{G}u_1(x)\tau^s(x,0)d\mu_{\mathcal{G}}+\int_\mathcal{G}u_0(x)\le(\tau^s\ri)_t(x,0)d\mu_{\mathcal{G}}\\
\nonumber&=-\int_{0}^\infty \le\{\sum_{x\in\mathcal{V}}\mu_{\mathcal{V}}(x)u(x,t)\Delta_{\mathcal{V}} \tau^s(x,t)\ri\}dt-\int_0^\infty\le\{\sum_{e\in\mathcal{E}}\int_{0}^{l_e}u_e(x,t)\le(\tau_e^s(x,t)\ri)^{\prime\prime}dx \ri\}dt\\
\nonumber&\quad-\int_\mathcal{G}u_1(x)\tau^s(x,0)d\mu_{\mathcal{G}}+\int_0^\infty \int_{\mathcal{G}}\le(\tau^s\ri)_{tt}(x,t) u(x,t)d\mu_{\mathcal{G}}dt\\
&:=I_1+I_2+I_3+I_4.
\end{align}
We estimate the four terms separately. Note that Lemma \ref{s4:L1} remains valid for $\theta_1\geq2$ and $\theta_2\geq2$ as well. Proceeding as in the derivation of \eqref{s5:11}, for any $R\geq \max\{4j,1\}$, by a convexity argument, we obtain from \eqref{s4:14} and \eqref{s4:13} that
\begin{align}\label{s5:1}
	\nonumber I_1+I_2&\leq \frac{C}{R}\int_{0}^\infty \le\{\sum_{x\in\mathcal{V}}\mu_{\mathcal{V}}(x)u(x,t)\tau^{s-1}(x,t)\textbf{1}_{M_R}(x,t)\ri\}dt\\
	&\quad+\frac{C}{R}\int_0^\infty\le\{\sum_{e\in\mathcal{E}}\int_{0}^{l_e}u_e(x,t)\tau_e^{s-1}(x,t)\textbf{1}_{M_R}(x,t)dx \ri\}dt,
\end{align}
where $M_R$ is given in \eqref{s4:28}. For the term $I_3$, observing that $\phi\equiv1$ on $[0,1]$ and $0\leq\phi\leq1$ on $[0,\infty)$, by \eqref{s4:16}, we have
\begin{align}\label{s5:5}
	\nonumber I_3&=-\sum_{x\in\mathcal{V}}\mu_{\mathcal{V}}(x)u_1(x)\tau^s(x,0)-\sum_{e\in\mathcal{E}}\int_{0}^{l_e}u_{1,e}(x)\tau_e^s(x,0)dx\\
	\nonumber&=-\sum_{x\in\mathcal{V}\cap \tilde{B}_{2^{1/\theta_1}R}}\mu_{\mathcal{V}}(x)u_1(x)\phi^s\le(\frac{d(x,x_0)^{\theta_1}}{R^{\theta_1}}\ri)-\sum_{e\in\mathcal{E}\cap \tilde{B}_{2^{1/\theta_1}R}}\int_{0}^{l_e}u_{1,e}(x)\phi_e^s\le(\frac{\tilde{d}(x,x_0)^{\theta_1}}{R^{\theta_1}}\ri)dx\\
	\nonumber&=-\sum_{x\in\mathcal{V}\cap B_{2^{1/\theta_1}R}}\mu_{\mathcal{V}}(x)u_1^+(x)\phi^s\le(\frac{d(x,x_0)^{\theta_1}}{R^{\theta_1}}\ri)-\sum_{x\in\mathcal{V}\cap B_{2^{1/\theta_1}R}}\mu_{\mathcal{V}}(x)u_1^-(x)\phi^s\le(\frac{d(x,x_0)^{\theta_1}}{R^{\theta_1}}\ri)\\
	\nonumber&\quad-\sum_{e\in\mathcal{E}\cap \tilde{B}_{2^{1/\theta_1}R}}\int_{0}^{l_e}u^+_{1,e}(x)\phi_e^s\le(\frac{\tilde{d}(x,x_0)^{\theta_1}}{R^{\theta_1}}\ri)dx-\sum_{e\in\mathcal{E}\cap \tilde{B}_{2^{1/\theta_1}R}}\int_{0}^{l_e}u^-_{1,e}(x)\phi_e^s\le(\frac{\tilde{d}(x,x_0)^{\theta_1}}{R^{\theta_1}}\ri)dx\\
	\nonumber&\leq -\sum_{x\in\mathcal{V}\cap B_{R}}\mu_{\mathcal{V}}(x)u_1^+(x)-\sum_{x\in\mathcal{V}\cap B_{2^{1/\theta_1}R}}\mu_{\mathcal{V}}(x)u_1^-(x)\\
	\nonumber&\quad-\sum_{e\in\mathcal{E}\cap \tilde{B}_{R}}\int_{0}^{l_e}u^+_{1,e}(x)dx-\sum_{e\in\mathcal{E}\cap \tilde{B}_{2^{1/\theta_1}R}}\int_{0}^{l_e}u^-_{1,e}(x)dx\\
	\nonumber&\leq -\sum_{x\in\mathcal{V}\cap B_{R-j}}\mu_{\mathcal{V}}(x)u_1^+(x)-\sum_{x\in\mathcal{V}\cap B_{2^{1/\theta_1}R+j}}\mu_{\mathcal{V}}(x)u_1^-(x)\\
	\nonumber&\quad-\sum_{e\in\mathcal{E}\cap B_{R-j}}\int_{0}^{l_e}u^+_{1,e}(x)dx-\sum_{e\in\mathcal{E}\cap B_{2^{1/\theta_1}R+j}}\int_{0}^{l_e}u^-_{1,e}(x)dx\\
	&=-\int_{B_{R-j}}u_1^+(x)d\mu_{\mathcal{G}}-\int_{B_{2^{1/\theta_1}R+j}}u_1^-(x)d\mu_{\mathcal{G}}.
\end{align}
It remains to estimate the term $I_4$.
\begin{equation}\label{s5:10}
	I_4=\int_0^\infty \int_{\mathcal{G}}s(s-1)\tau^{s-2}(x,t) \tau_t^2(x,t)u(x,t)d\mu_{\mathcal{G}}dt+\int_0^\infty \int_{\mathcal{G}}s\tau^{s-1}(x,t)\tau_{tt}(x,t) u(x,t)d\mu_{\mathcal{G}}dt.
\end{equation}
Recalling the definition of $\tilde{E}_R$ in \eqref{s4:27}, we first establish the following estimate: for any $x\in\mathcal{G}$ and $t\in[0,\infty)$,
\begin{equation}\label{s5:8}
	|\tau_{t}(x,t)|\leq \frac{C}{R^{\frac{1}{2}}}\textbf{1}_{\tilde{E}_R}(x,t),\quad	|\tau_{tt}(x,t)|\leq \frac{C}{R}\textbf{1}_{\tilde{E}_R}(x,t).
\end{equation}
Indeed, by $\theta_1\geq\theta_2/2$, we know from \eqref{s4:29}  that
$$	|\tau_{t}(x,t)|\leq \frac{C}{R^{\frac{\theta_1}{\theta_2}}}\textbf{1}_{\tilde{E}_R}(x,t)\leq\frac{C}{R^{\frac{1}{2}}}\textbf{1}_{\tilde{E}_R}(x,t).$$
Similarly, it follows from \eqref{s4:6} that
\begin{align*}
	|\tau_{tt}(x,t)|&=\le|\phi^{\prime\prime}\le(\frac{\tilde{d}(x,x_0)^{\theta_1}+t^{\theta_2}}{R^{\theta_1}}\ri)\frac{\theta_2^2t^{2\theta_2-2}}{R^{2\theta_1}}+\phi^\prime\le(\frac{\tilde{d}(x,x_0)^{\theta_1}+t^{\theta_2}}{R^{\theta_1}}\ri)\frac{\theta_2(\theta_2-1)t^{\theta_2-2}}{R^{\theta_1}}\ri|\\
	&\leq \frac{C}{R^{\frac{2\theta_1}{\theta_2}}}\textbf{1}_{\tilde{E}_R}(x,t)+\frac{C}{R^{\frac{2\theta_1}{\theta_2}}}\textbf{1}_{\tilde{E}_R}(x,t)\\
	&\leq \frac{C}{R}\textbf{1}_{\tilde{E}_R}(x,t),
\end{align*}
where we have used the facts $\theta_1\geq\theta_2/2$ and $t\leq CR^{\frac{\theta_1}{\theta_2}}$ whence $(x,t)\in \tilde{E}_R$. Combining \eqref{s5:10} and \eqref{s5:8}, and using the fact that  $s>2\sigma/(\sigma-1)>2$, by Lemma \ref{s4:L1}, we have
 \begin{align}\label{s5:18}
 	\nonumber I_4&\leq \frac{C}{R}\int_0^\infty \int_{\mathcal{G}}\tau^{s-2}(x,t) u(x,t)\textbf{1}_{\tilde{E}_R}(x,t)d\mu_{\mathcal{G}}dt+\frac{C}{R}\int_0^\infty \int_{\mathcal{G}}\tau^{s-1}(x,t) u(x,t)\textbf{1}_{\tilde{E}_R}(x,t)d\mu_{\mathcal{G}}dt\\
 	\nonumber&\leq  \frac{C}{R}\int_0^\infty \int_{\mathcal{G}}\tau^{s-2}(x,t) u(x,t)\textbf{1}_{\tilde{E}_R}(x,t)d\mu_{\mathcal{G}}dt\\
 	&\leq \frac{C}{R}\int_{0}^\infty \le\{\sum_{x\in\mathcal{V}}\mu_{\mathcal{V}}(x)u(x,t)\tau^{s-2}(x,t)\textbf{1}_{M_R}(x,t)\ri\}dt+\frac{C}{R}\int_0^\infty\le\{\sum_{e\in\mathcal{E}}\int_{0}^{l_e}u_e(x,t)\tau_e^{s-2}(x,t)\textbf{1}_{M_R}(x,t)dx \ri\}dt.
 \end{align}
 As $0\leq\tau\leq1$ and $s>2$, substituting \eqref{s5:1}, \eqref{s5:5} and \eqref{s5:18} into \eqref{s5:19} yields
 \begin{align*}
 	&\int_0^\infty \int_{\mathcal{G}}h(x,t)u^\sigma(x,t)\tau^s(x,t)d\mu_{\mathcal{G}} dt\\
 	&\leq\frac{C}{R}\int_{0}^\infty \le\{\sum_{x\in\mathcal{V}}\mu_{\mathcal{V}}(x)u(x,t)\tau^{s-2}(x,t)\textbf{1}_{M_R}(x,t)\ri\}dt+\frac{C}{R}\int_0^\infty\le\{\sum_{e\in\mathcal{E}}\int_{0}^{l_e}u_e(x,t)\tau_e^{s-2}(x,t)\textbf{1}_{M_R}(x,t)dx \ri\}dt\\
 	&\quad -\int_{B_{R-j}}u_1^+(x)d\mu_{\mathcal{G}}-\int_{B_{2^{1/\theta_1}R+j}}u_1^-(x)d\mu_{\mathcal{G}}.
 \end{align*}
 By applying the Young’s inequality, we get
 \begin{align*}
 	&\int_0^\infty \int_{\mathcal{G}}h(x,t)u^\sigma(x,t)\tau^s(x,t)d\mu_{\mathcal{G}} dt\\
 	&\leq \epsilon \int_0^\infty \le\{\sum_{x\in\mathcal{V}}\mu_{\mathcal{V}}(x) h(x,t)\tau^{s}(x,t)u^\sigma(x,t)\textbf{1}_{M_R}(x,t)\ri\}dt\\
 	\nonumber&\quad+\frac{C}{R^{\frac{\sigma}{\sigma-1}}}\int_0^\infty \le\{\sum_{x\in\mathcal{V}}\mu_{\mathcal{V}}(x)\tau^{s-\frac{2\sigma}{\sigma-1}}(x,t) h^{-\frac{1}{\sigma-1}}(x,t)\textbf{1}_{M_R}(x,t)\ri\}dt\\
 	\nonumber&\quad+\epsilon\int_0^\infty \le\{\sum_{e\in \mathcal{E}}\int_{0}^{l_e} h_e(x,t)\tau_e^{s}(x,t)u^\sigma_e(x,t)\textbf{1}_{M_R}(x,t)dx\ri\}dt\\
 	\nonumber&\quad+\frac{C}{R^{\frac{\sigma}{\sigma-1}}}\int_0^\infty \le\{\sum_{e\in \mathcal{E}}\int_{0}^{l_e}\tau_e^{s-\frac{2\sigma}{\sigma-1}}(x,t) h_e^{-\frac{1}{\sigma-1}}(x,t)\textbf{1}_{M_R}(x,t)dx\ri\}dt\\
 	&\quad -\int_{B_{R-j}}u_1^+(x)d\mu_{\mathcal{G}}-\int_{B_{2^{1/\theta_1}R+j}}u_1^-(x)d\mu_{\mathcal{G}}\\
 	&\leq \epsilon\int_0^\infty \int_{\mathcal{G}}h(x,t)u^\sigma(x,t)\tau^s(x,t)d\mu_{\mathcal{G}} dt+\frac{C}{R^{\frac{\sigma}{\sigma-1}}}\int_0^\infty \le\{\sum_{x\in\mathcal{V}}\mu_{\mathcal{V}}(x) h^{-\frac{1}{\sigma-1}}(x,t)\textbf{1}_{M_R}(x,t)\ri\}dt\\
 	&\quad+\frac{C}{R^{\frac{\sigma}{\sigma-1}}}\int_0^\infty \le\{\sum_{e\in \mathcal{E}}\int_{0}^{l_e} h_e^{-\frac{1}{\sigma-1}}(x,t)\textbf{1}_{M_R}(x,t)dx\ri\}dt -\int_{B_{R-j}}u_1^+(x)d\mu_{\mathcal{G}}-\int_{B_{2^{1/\theta_1}R+j}}u_1^-(x)d\mu_{\mathcal{G}}.
 \end{align*}
 Using \eqref{s5:7}, following the same procedure of \eqref{s5:24}, we deduce from \eqref{s3:17} that
 \begin{align*}
 	\int_0^\infty \int_{\mathcal{G}}h(x,t)u^\sigma(x,t) \textbf{1}_{F_R}(x,t)d\mu_{\mathcal{G}} dt&\leq\int_0^\infty \int_{\mathcal{G}}h(x,t)u^\sigma(x,t)\tau^s(x,t)d\mu_{\mathcal{G}} dt\\
 	&\leq C-C\le(\int_{B_{R-j}}u_1^+(x)d\mu_{\mathcal{G}}+\int_{B_{2^{1/\theta_1}R+j}}u_1^-(x)d\mu_{\mathcal{G}}\ri),
 \end{align*}
 where $F_R$ is a bounded domain defined by \eqref{s5:25}. Thus, by \eqref{s3:13}, we derive that
 $$\int_0^\infty \int_{\mathcal{G}}h(x,t)u^\sigma(x,t)d\mu_{\mathcal{G}}dt\leq C.$$
  The remainder of the proof follows exactly as in Theorem \ref{s2:T3} by applying H\"{o}lder’s inequality; the details are omitted here. $\hfill\Box$\\

 \noindent$\textbf{\emph{Proof of Theorem \ref{s2:T2}.}}$ Suppose that $u$ is a very weak solution to \eqref{s1:2}. We first define the test function $\Psi:\mathcal{G}\times[0,\infty)\to\mathbb{R}$ by 
 \begin{equation*}
 		\Psi(x,t):=\phi^s\le(\frac{t}{R^{\frac{1}{2}}}\ri)\psi\le(\frac{\tilde{d}(x,x_0)-j}{R}\ri),
 \end{equation*}
 	where $\phi$ and $\psi$ are defined in \eqref{s4:6} and Subsection \ref{3.3}, respectively, $s>2\sigma/(\sigma-1)$, and $R\geq \max\{1,2j\}$. Proceeding as in Lemma \ref{s5:L3}, we have the following estimates:
 	
 	(i)	For each edge $e\in\mathcal{E}$, all $x\in (0,l_e)$, and all $t\in[0,\infty)$, 
 	\begin{equation}\label{s5:26}
 		|\Psi_e^{\prime\prime}(x,t)|\leq \frac{C}{R}\phi^s\le(\frac{t}{R^\frac{1}{2}}\ri)e^{\frac{-\delta \tilde{d}(x,x_0)}{R}}\textbf{1}_{\tilde{B}_{R+j}^c}(x).
 	\end{equation}
 	
 	(ii) For any $x\in\mathcal{V}$ and $t\in[0,\infty)$, 
 	\begin{equation}\label{s5:27}
 		\le|\Delta_{\mathcal{V}}\Psi(x,t)\ri|\leq \frac{C}{R}\phi^s\le(\frac{t}{R^\frac{1}{2}}\ri)e^{\frac{-\delta d(x,x_0)}{R}}\textbf{1}_{B_{R}^c}(x).
 	\end{equation}
 	
 	(iii) For any $(x,t)\in\mathcal{G}\times[0,\infty)$, 
 	\begin{equation}\label{s5:28}
 		|\Psi_{t}(x,t)|\leq \frac{C}{R^\frac{1}{2}}\phi^{s-1}\le(\frac{t}{R^\frac{1}{2}}\ri)e^{\frac{-\delta d(x,x_0)}{R}}\textbf{1}_{P_R}(x,t), \quad 	|\Psi_{tt}(x,t)|\leq \frac{C}{R}\phi^{s-2}\le(\frac{t}{R^\frac{1}{2}}\ri)e^{\frac{-\delta d(x,x_0)}{R}}\textbf{1}_{P_R}(x,t),
 	\end{equation}
 	where  $P_R:=\mathcal{G}\times[R^\frac{1}{2},2R^\frac{1}{2}]$.
 	
 	Next, since $u\in L^1_{\text{loc}}([0,\infty); X_\alpha)$ with $\alpha=\delta/R>0$, by the argument in Lemma \ref{s4:L2}, we also have
 	 \begin{align*}
 		\nonumber\int_0^\infty\int_{\mathcal{G}}\Psi(x,t) \Delta_{\mathcal{G}}u(x,t) d\mu_{\mathcal{G}} dt=\int_0^\infty\int_{\mathcal{G}}u(x,t) \Delta_{\mathcal{G}}\Psi(x,t) d\mu_{\mathcal{G}} dt.
 	\end{align*} 
Moreover, $\Psi$ is admissible as a test function; we omit the verification to avoid redundancy.

We now estimate each term on the right-hand side of
\begin{align*}
		\nonumber&\int_0^\infty \int_{\mathcal{G}}h(x,t)|u(x,t)|^\sigma\Psi(x,t)d\mu_{\mathcal{G}} dt\\
	\nonumber&\leq
	-\int_0^\infty \int_{\mathcal{G}} u(x,t) \Delta_{\mathcal{G}}\Psi(x,t)d\mu_{\mathcal{G}}dt+\int_0^\infty \int_{\mathcal{G}}\Psi_{tt}(x,t) u(x,t)d\mu_{\mathcal{G}}dt\\
	\nonumber&\quad-\int_\mathcal{G}u_1(x)\Psi(x,0)d\mu_{\mathcal{G}}+\int_\mathcal{G}u_0(x)\Psi_t(x,0)d\mu_{\mathcal{G}}.
\end{align*}
We denote these terms by $L_1,L_2,L_3,L_4$, respectively.
By \eqref{s3:4} and \eqref{s3:5}, it follows from \eqref{s5:26} and \eqref{s5:27} that
\begin{align*}
|L_1|&\leq\int_0^\infty\le\{\sum_{x\in\mathcal{V}}\mu_{\mathcal{V}}(x)|u(x,t)| \le|\Delta_{\mathcal{V}}\Psi(x,t) \ri|\ri\}dt+\int_0^\infty \le\{\sum_{e\in\mathcal{E}}\int_{0}^{l_e}|u_e(x,t)||\Psi_{e}^{\prime\prime}(x,t)|dx \ri\}dt\\
	\nonumber&\leq \frac{C}{R}\int_0^\infty\le\{ \sum_{x\in\mathcal{V}}\mu_{\mathcal{V}}(x)|u(x,t)| \phi^s\le(\frac{t}{R^\frac{1}{2}}\ri)e^{\frac{-\delta d(x,x_0)}{R}}\textbf{1}_{ B_{R}^c}(x)\ri\}dt\\
	&\quad +\frac{C}{R} \int_0^\infty\le\{\sum_{e\in\mathcal{E}}\int_{0}^{l_e} |u_e(x,t)| \phi^s\le(\frac{t}{R^\frac{1}{2}}\ri)e^{\frac{-\delta \tilde{d}(x,x_0)}{R}}\textbf{1}_{ B_{R}^c}(x)dx\ri\}dt\\
		\nonumber&\leq  \epsilon C\int_0^\infty \le\{\sum_{x\in\mathcal{V}}\mu_{\mathcal{V}}(x) |u(x,t)|^\sigma h(x,t) \phi^s\le(\frac{t}{R^\frac{1}{2}}\ri)\psi\le(\frac{d(x,x_0)-j}{R}\ri)\textbf{1}_{ B_{R}^c}(x)\ri\}dt \\
		&\quad+\frac{C_\epsilon}{R^{\frac{\sigma}{\sigma-1}}}\int_0^\infty \le\{\sum_{x\in\mathcal{V}}\mu_{\mathcal{V}}(x) h^{-\frac{1}{\sigma-1}}(x,t) \phi^s\le(\frac{t}{R^\frac{1}{2}}\ri)e^{\frac{-\delta d(x,x_0)}{R}}\textbf{1}_{ B_{R}^c}(x)\ri\}dt \\
	\nonumber&\quad+ \epsilon C \int_0^\infty\le\{\sum_{e\in\mathcal{E}}\int_{0}^{l_e} |u_e(x,t)|^\sigma h_e(x,t) \phi^s\le(\frac{t}{R^\frac{1}{2}}\ri) \psi_e\le(\frac{\tilde{d}(x,x_0)-j}{R}\ri)\textbf{1}_{ B_{R}^c}(x)dx\ri\}dt\\
	&\quad +\frac{C_\epsilon}{R^{\frac{\sigma}{\sigma-1}}}\int_0^{\infty}\le\{\sum_{e\in\mathcal{E}}\int_{0}^{l_e} h_e^{-\frac{1}{\sigma-1}}(x,t)\phi^{s}\le(\frac{t}{R^\frac{1}{2}}\ri)e^{-\delta\frac{ d(x,x_0)}{R}}\textbf{1}_{ B_{R}^c}(x)dx\ri\}dt\\
	\nonumber&\leq  \frac{1}{4}\int_0^\infty \int_{\mathcal{G}}h(x,t)|u(x,t)|^\sigma\Psi(x,t)\textbf{1}_{ B_{R}^c}(x)d\mu_{\mathcal{G}} dt\\
	&\quad+\frac{C}{R^{\frac{\sigma}{\sigma-1}}}\int_0^{2R^\frac{1}{2}} \le\{\sum_{x\in\mathcal{V}\cap B_{R}^c}\mu_{\mathcal{V}}(x) h^{-\frac{1}{\sigma-1}}(x,t) e^{\frac{-\delta d(x,x_0)}{R}}\ri\}dt \\
	&\quad+\frac{C}{R^{\frac{\sigma}{\sigma-1}}}\int_0^{2R^\frac{1}{2}}\le\{\sum_{e\in\mathcal{E}\cap B_{R}^c}\int_{0}^{l_e} h_e^{-\frac{1}{\sigma-1}}(x,t)e^{\frac{-\delta d(x,x_0)}{R}}dx\ri\}dt\\
	&\leq  \frac{1}{4}\int_0^\infty \int_{\mathcal{G}}h(x,t)|u(x,t)|^\sigma\Psi(x,t)\textbf{1}_{ B_{R}^c}(x)d\mu_{\mathcal{G}} dt+C.
\end{align*}
A similar argument as in \eqref{s5:23} and \eqref{s5:14} applies here.
By \eqref{s5:28} and $s>2\sigma/(\sigma-1)$, using \eqref{s3:2} and \eqref{s3:3}, we deduce
\begin{align*}
	\nonumber\le|L_2\ri|&\leq \frac{C}{R}\int_0^\infty \int_{\mathcal{G}}|u(x,t)| \phi^{s-2}\le(\frac{t}{R^\frac{1}{2}}\ri)e^{\frac{-\delta d(x,x_0)}{R}}\textbf{1}_{P_R}(x,t) d\mu_{\mathcal{G}}dt\\
	\nonumber&\leq \epsilon \int_{R^\frac{1}{2}}^{2R^\frac{1}{2}} \le\{\sum_{x\in\mathcal{V}}\mu_{\mathcal{V}}(x)|u(x,t)|^\sigma h(x,t) \phi^{s}\le(\frac{t}{R^\frac{1}{2}}\ri)e^{\frac{-\delta d(x,x_0)}{R}}\ri\}dt\\
	&\quad+\frac{C_\epsilon}{R^{\frac{\sigma}{\sigma-1}}}\int_{R^\frac{1}{2}}^{2R^\frac{1}{2}} \le\{\sum_{x\in\mathcal{V}}\mu_{\mathcal{V}}(x) h^{-\frac{1}{\sigma-1}}(x,t) \phi^{s-\frac{2\sigma}{\sigma-1}}\le(\frac{t}{R^\frac{1}{2}}\ri)e^{\frac{-\delta d(x,x_0)}{R}} \ri\} dt\\
	\nonumber&\quad+ \epsilon\int_{R^\frac{1}{2}}^{2R^\frac{1}{2}}\le\{\sum_{e\in\mathcal{E}}\int_{0}^{l_e}|u_e(x,t)|^\sigma h_e(x,t)\phi^{s}\le(\frac{t}{R^\frac{1}{2}}\ri)e^{\frac{-\delta d(x,x_0)}{R}}dx \ri\}dt\\
	\nonumber&\quad+\frac{C_\epsilon}{R^{\frac{\sigma}{\sigma-1}}}\int_{R^\frac{1}{2}}^{2R^\frac{1}{2}}\le\{\sum_{e\in\mathcal{E}}\int_{0}^{l_e}h_e^{-\frac{1}{\sigma-1}}(x,t)\phi^{s-\frac{2\sigma}{\sigma-1}}\le(\frac{t}{R^\frac{1}{2}}\ri)e^{\frac{-\delta d(x,x_0)}{R}}dx \ri\}dt\\
	\nonumber&\leq \epsilon C \int_{R^\frac{1}{2}}^{2R^\frac{1}{2}} \le\{\sum_{x\in\mathcal{V}}\mu_{\mathcal{V}}(x)|u(x,t)|^\sigma h(x,t) \phi^{s}\le(\frac{t}{R^\frac{1}{2}}\ri)\psi\le(\frac{d(x,x_0)-j}{R}\ri)\ri\}dt\\
	&\nonumber\quad+ \epsilon C\int_{R^\frac{1}{2}}^{2R^\frac{1}{2}}\le\{\sum_{e\in\mathcal{E}}\int_{0}^{l_e}|u_e(x,t)|^\sigma h_e(x,t)\phi^{s}\le(\frac{t}{R^\frac{1}{2}}\ri)\psi_e\le(\frac{\tilde{d}(x,x_0)-j}{R}\ri)dx \ri\}dt\\
	\nonumber&\quad+\frac{C_\epsilon}{R^{\frac{\sigma}{\sigma-1}}}\int_{R^\frac{1}{2}}^{2R^\frac{1}{2}} \le\{\sum_{x\in\mathcal{V}}\mu_{\mathcal{V}}(x) h^{-\frac{1}{\sigma-1}}(x,t) e^{\frac{-\delta d(x,x_0)}{R}} \ri\}dt\\
	&\quad+\frac{C_\epsilon}{R^{\frac{\sigma}{\sigma-1}}}\int_{R^\frac{1}{2}}^{2R^\frac{1}{2}}\le\{\sum_{e\in\mathcal{E}}\int_{0}^{l_e}h_e^{-\frac{1}{\sigma-1}}(x,t)e^{\frac{-\delta d(x,x_0)}{R}}dx \ri\}dt\\
	&\leq \frac{1}{4} \int_{R^\frac{1}{2}}^{2R^\frac{1}{2}}\int_\mathcal{G} |u(x,t)|^\sigma h(x,t)\Psi(x,t)d\mu_{\mathcal{G}}dt+C.
\end{align*}
For $L_3$, since $\Phi(x,0)=\Psi(x,0)$ for all $x\in\mathcal{G}$, by \eqref{s5:15} we have
 $$L_3\leq-\int_{B_{R}}u_1^+(x)d\mu_{\mathcal{G}}-\int_\mathcal{G}u_1^-(x)d\mu_{\mathcal{G}}.$$
 Finally, noting that
 $$	\Psi_{t}(x,t)=\frac{s}{R^\frac{1}{2}}\phi^{s-1}\le(\frac{t}{R^\frac{1}{2}}\ri)\phi^\prime\le(\frac{t}{R^\frac{1}{2}}\ri)\psi\le(\frac{\tilde{d}(x,x_0)-j}{R}\ri),$$
we immediately get
 $$\Psi_{t}(x,0)=0,\quad\forall\ x\in\mathcal{G},$$
which implies that
$$L_4=0.$$

Combining the above estimates leads to
\begin{align*}
	\int_0^{R^\frac{1}{2}} \int_{\tilde{B}_{R+j}}h(x,t)|u(x,t)|^\sigma d\mu_{\mathcal{G}} dt&\leq \int_0^\infty \int_{\mathcal{G}}h(x,t)|u(x,t)|^\sigma\Psi(x,t)d\mu_{\mathcal{G}} dt\\
	&\leq C-C\le(\int_{B_{R}}u_1^+(x)d\mu_{\mathcal{G}}+\int_\mathcal{G}u_1^-(x)d\mu_{\mathcal{G}}\ri),
	\end{align*}
 which, together with \eqref{s3:6}, yields that 
 $$\int_0^\infty \int_{\mathcal{G}}h(x,t)|u(x,t)|^\sigma d\mu_{\mathcal{G}} dt\leq C.$$
The final part follows similarly to Theorem \ref{s2:T4} using H\"{o}lder’s inequality; we omit it.$\hfill\Box$\\

 \noindent$\textbf{\emph{Proof of Corollary \ref{s2:C6}.}}$ By the same argument as in Corollaries \ref{s2:C2} and \ref{s2:C3}, Corollary \ref{s2:C6} is an immediate consequence of Theorem \ref{s2:T2}. $\hfill\Box$\\

\begin{appendices}

	\noindent
	\textbf{Appendix}   \\

	For the reader’s convenience, we recall the fundamental notions and preliminaries on metric graphs, as detailed results can be found in \cite {B-K,M,M-P,P-T1,P-T2,L-L-Z}.
	
	\section{The metric graph setting}\label{A}
	
	Like combinatorial graphs, a metric graph comprises a countable set $\mathcal{V}$ of vertices and a countable set $E$ of edges. In contrast to combinatorial graphs, however, the edges are treated as intervals glued together at the vertices. Given a function $l:E\ra(0,+\infty]$, it is usually referred to as a weight. We consider the weighted graph $(\mathcal{V}, E, l)$ and regard $l(e)$ as the length of the edge $e\in E$ (denoted as $l_e$ for short). Let 
	$$\mathcal{E}:=\bigcup_{e\in E}\{e\}\times(0,l_e).$$
	We may give the following definition.
	\begin{definition}\label{s0:D1}
		The metric graph $\mathcal{G}$ over the weighted graph $(\mathcal{V}, E, l)$ is the pair $(\mathcal{V},\mathcal{E})$.
	\end{definition}
	We equip the metric graph $\mathcal{G}$ with maps $i : E \to \mathcal{V}$ assigning the initial vertex of each edge and $j : \{e \in E : l_e < +\infty\} \to \mathcal{V}$ assigning the final vertex, with these vertices collectively referred to as the endpoints of the edge. We always assume for simplicity that  $l_e<+\infty$ for all $e\in E$, and use the following notations,
	$$I_e:=(0,l_e),\quad\mathcal{G}_e:=\{e\}\times I_e,\quad\overline{\mathcal{G}}_e:=e\cup\{i(e),j(e)\}\times\overline{I}_e.$$ 
	For $e \in E$ and $v\in \mathcal{V}$, we write $e\ni v$ (or $v\in e$) if \(i(e)=v\) or \(j(e)=v\) (i.e., $v$ is an endpoint of $e$). In what follows, we adopt the notational convention of denoting points in $\mathcal{G}$ as $x \in \mathcal{G}$, where either $x = v \in \mathcal{V}$ or $x \in \mathcal{G}_e$ for some $e \in E$. For simplicity, we sometimes make no distinction between $e$ and $I_e$, performing this identification by abuse of language; accordingly, we may write $x\in e$ or $x\in I_e$ instead of $x\in\mathcal{G}_e$, and denote $e\in\mathcal{E}$ as $\mathcal{G}_e\in\mathcal{E}$, without causing confusion. Moreover, this practice introduces no ambiguity when the same notation $x, y, \dots$ is used to denote both points of the edge $e \in E$ and points of the interval $I_e \subseteq \mathbb{R}_+$. For each $e \in \mathcal{E}$, the map $\pi_e \colon \mathcal{G}_e \to I_e$ defined by $\pi_e(\{e\}, x) \equiv \pi_e(x) := x$ sets up a bijection between points of $e\in E$ and points of $I_e$. This map can be extended to a mapping from $\overline{\mathcal{G}}_e$ to $\overline{I}_e = [0, l_e]$ such that \(\pi_e(i(e)) = 0\) and \(\pi_e(j(e)) = l_e\).
	
		\begin{figure}[htbp]
		\centering
		
		\begin{tikzpicture}[scale=0.8]
			
			
			\coordinate (v1)  at (0,0);
			\coordinate (v2)  at (1.8,1.2);
			\coordinate (v3)  at (2.0,-1.4);
			\coordinate (v4)  at (4.0,0.3);
			\coordinate (v5)  at (5.6,1.8);
			\coordinate (v6)  at (6.2,-1.0);
			\coordinate (v7)  at (8.0,0.9);
			
			\coordinate (v8)  at (-1.3,1.0);
			\coordinate (v9)  at (-1.0,-0.9);
			
			\coordinate (v10) at (9.2,2.0);
			\coordinate (v11) at (9.4,-0.3);
			
			\coordinate (v12) at (5.0,-2.2);
			
			
			\draw[thick] (v1) to[bend left=10] node[above left] {$e_1$} (v2);
			
			\draw[thick] (v1) to[bend right=12] node[below left] {$e_2$} (v3);
			
			\draw[thick] (v2) to[bend left=8] node[above] {$e_3$} (v4);
			
			\draw[thick] (v3) to[bend right=10] node[below] {$e_4$} (v4);
			
			\draw[thick] (v4) to[bend left=10] node[above] {$e_5$} (v5);
			
			\draw[thick] (v4) to[bend right=10] node[below] {$e_6$} (v6);
			
			\draw[thick] (v5) to[bend left=8] node[above] {$e_7$} (v7);
			
			\draw[thick] (v6) to[bend right=8] node[below] {$e_8$} (v7);
			
			\draw[thick] (v6) to[bend right=15] node[right] {$e_9$} (v12);
			
			
			\draw[dashed] (v1)--(v8);
			\draw[dashed] (v1)--(v9);
			
			\draw[dashed] (v7)--(v10);
			\draw[dashed] (v7)--(v11);
			
			
			\draw[dashed] (v2)--++(-0.6,1.0);
			\draw[dashed] (v3)--++(-0.5,-1.0);
			
			\draw[dashed] (v5)--++(0.5,1.1);
			
			\draw[dashed] (v12)--++(0.5,-1.1);
			
			\node at (0.9,2.5) {$\cdots$};
			\node at (1.1,-2.7) {$\cdots$};
			
			\node at (6.5,3.0) {$\cdots$};
			
			\node at (6.0,-3.4) {$\cdots$};

			
			\fill (v1) circle (2.2pt);
			\fill (v2) circle (2.2pt);
			\fill (v3) circle (2.2pt);
			\fill (v4) circle (2.2pt);
			\fill (v5) circle (2.2pt);
			\fill (v6) circle (2.2pt);
			\fill (v7) circle (2.2pt);
			\fill (v12) circle (2.2pt);
			
			
			\node[left]  at (v1) {$v_1$};
			
			\node[above] at (v2) {$v_2$};
			
			\node[below] at (v3) {$v_3$};
			
			\node[above] at (v4) {$v_4$};
			
			\node[above] at (v5) {$v_5$};
			
			\node[below] at (v6) {$v_6$};
			
			\node[right] at (v7) {$v_7$};
			
			\node[below] at (v12) {$v_{12}$};
			
			\node at (-1.7,1.3) {$\cdots$};
			\node at (-1.5,-1.2) {$\cdots$};
			
			\node at (9.8,2.3) {$\cdots$};
			\node at (9.9,-0.5) {$\cdots$};
			
		\end{tikzpicture}
		
		\caption{
			An infinite locally finite metric graph \(\mathcal{G}=(\mathcal{V},\mathcal{E})\). 
			Each edge is identified with a compact interval of finite length.
		}
		
		\label{fig:metricgraph}
		
	\end{figure}
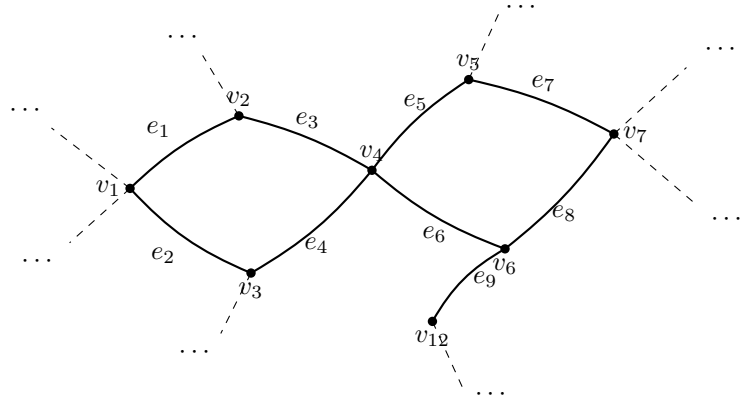
	
	\begin{definition}
		Let $\mathcal{G}$ be a metric graph.
		
		(i) A metric graph $\mathcal{G}$ is finite if both $E$ and $\mathcal{V}$ are finite sets; it is infinite otherwise.
		
		(ii) For a vertex $v\in\mathcal{V}$,  its degree $\deg_v \in \mathbb{N}$ counts the number of edges $e \ni v$. The inbound degree $\deg_v^+$ (resp. outbound degree $\deg_v^-$) refers to the number of edges with $j(e) = v$ (resp. $i(e) = v$). Obviously, $\deg_v=\deg_v^++\deg_v^-$. $\mathcal{G}$ is locally finite if $\deg_v < \infty$ for all $v \in \mathcal{V}$. 
		
		(iii) For two vertices $u, v \in \mathcal{V}$, a path connecting them is a set $\{x_1, \dots, x_n\} \subset \mathcal{G}$ ($n \in \mathbb{N}$) such that $x_1 = u$, $x_n = v$, and for each $k = 1, \dots, n-1$, there exists an edge $e_k$ where both $x_k$ and $ x_{k+1}$ lie in $ \overline{\mathcal{G}}_{e_k}$. A path is closed if its start and end vertices coincide $(u = v)$. A closed path is termed a cycle if it does not pass through the same vertex more than once.
		
		(iv) A metric graph $\mathcal{G}$ is connected if there exists a path between any two distinct vertices $v, w \in \mathcal{V}$. A connected graph with no cycles is called a tree.
		
		(v) The boundary of the metric graph is given by $\p\mathcal{G} := \{v\in\mathcal{V}\mid \deg_v=1\}$.
	\end{definition}

	\section{Two volume measures}\label{B}
	
	In traditional combinatorial analysis, concepts such as integrals, Laplacians, and Rayleigh quotients are all defined using a single volume measure. In this paper, we depart from this convention by employing two distinct volume measures.
	
	Let $\mu_{\mathcal{V}}:\mathcal{V}\ra\mathbb{R}$ be a vertex measure supported on the vertex set $\mathcal{V}$, with $\mu_{\mathcal{V}}(v) > 0$ for every $v \in \mathcal{V}$. We then define an edge measure. A connected metric graph $\mathcal{G}$ can naturally be endowed with the structure of a metric measure space. To elaborate, for any two points $x, y \in \mathcal{G}$, we may treat them as vertices of a connecting path $P$ (with $x$ and $y$ possibly added to the vertex set $\mathcal{V}$ if necessary). The length of $P$ is defined as the sum of the length of its $n$ edges $e_k$, i.e., $l(P) := \sum_{k=1}^n l_{e_k}$. The distance $ d(x, y)$ between $x$ and $y$ is then given by the infimum of the lengths of all such connecting paths:
	$$d(x, y) \coloneqq \inf \left\{ l(P) \mid P \text{ connects } x \text{ and } y \right\}.$$
	This makes $\mathcal{G}$ a metric space, which in turn induces a topological structure via the metric topology. Let $\mathcal{B} = \mathcal{B}(\mathcal{G})$ denote the Borel $\sigma$-algebra of $\mathcal{G}$. Let $B(x_0,r)$ denote the open ball on the metric graph $\mathcal{G}$ with center $x_0\in\mathcal{G}$ and radius $r>0$, which consists of all points in $\mathcal{G}$ whose distance from $x_0$ is less than $r$. If $\mathcal{G}$ is locally finite, then $B(x_0,r)$ is a union of finitely many open subintervals of edges and finitely many entire edges for $r$ small enough.  We define the edge measure $\mu_{\mathcal E}$ on $\mathcal B(\mathcal G)$ by
	\begin{equation}\label{s2:8}
	\mu_{\mathcal E}(\Omega)=\sum_{e\in\mathcal E}\lambda(I_e\cap\Omega),\quad\forall\ \O\in\mathcal{B},
    \end{equation}
	where $\lambda$ is the Lebesgue measure on each interval $I_e$. It is straightforward to verify that $\mu_{\mathcal E}$ is a Radon measure on $\mathcal G$ and satisfies $\mu_{\mathcal{E}}(v)=0$ for all $v\in\mathcal{V}$. 
	
	\section{Functional spaces on metric graphs}\label{C}
	Let
	$$
	\mathcal F:=\prod_{e\in\mathcal E}
	\{f_e:\overline I_e\to\mathbb R\}
	$$
	denote the collection of all families of functions defined on the closed edges of $\mathcal G$. For $f=\{f_e\}_{e\in\mathcal E}\in\mathcal F$, we write
	$$
	f=\bigoplus_{e\in\mathcal E}f_e.
	$$
	We say that  $f$ is compatible at the vertices if, for every vertex $v\in\mathcal V$, the values $f_e(v)$ coincide for all edges $e$ incident to $v$. We denote by $\mathcal F_0$ the subspace of all vertex-compatible families. In this case, the common value at $v$ is denoted by $f(v)$, and $f$ can be naturally identified with a function on the metric graph $\mathcal G$. Throughout this paper, all functions under consideration belong to $\mathcal F_0$. Hence, unless otherwise stated, every element of $\mathcal F_0$ will be regarded as a well-defined function on $\mathcal G$. Accordingly, we shall identify $\mathcal F_0$ with the space of functions on $\mathcal G$ and make no distinction between the two notions.
		
	We define $f^{(h)} \coloneqq \bigoplus_{e\in \mathcal{E}} f^{(h)}_e$ for $h \in \mathbb{N}$, whenever the derivative $f^{(h)}_e \coloneqq \frac{d^h f_e}{dx^h}$ exists on $I_e$ for all $e \in E$. We also adopt the notation $f^{(0)} \equiv f$, $f^{(1)} \equiv f^\prime$ and $f^{(2)} \equiv f^{\prime\prime}$. We say that $f\in\mathcal F_0$ is continuous on $\mathcal{G}$, writing $f \in C(\mathcal{G})$, if $f_e \in C(\overline{I}_e)$ for all $e \in E$. We set
	$$C^k(\mathcal{G}):=\{f\in C(\mathcal{G})\mid f_e\in C^{k}(\overline{I}_e), \forall e\in\mathcal{E}, f^{(h)}\in C(\mathcal{G}),\forall h=1,\dots,k\},\quad(k\in\mathbb{N}),$$
	and $C^0(\mathcal{G})=C(\mathcal{G})$. 
	
	Based on \eqref{s2:8}, for every function $f\in\mathcal F_0$ such that the following series and integrals are well defined, we set 
	\begin{equation*}
		\int_\mathcal{G}f(x) d\mu_{\mathcal{E}}:=\sum_{e\in \mathcal{E}}\int_0^{l_e} f_e(x)dx,\quad \int_{\Omega} f(x)d\mu_{\mathcal{E}}:=\int_\mathcal{G}f(x) \textbf{1}_{\Omega}(x)d\mu_{\mathcal{E}},\quad\forall\ \Omega\in\mathcal{B}.
	\end{equation*}
	Here, $\textbf{1}_{\O}$ denotes the characteristic function of the set $\O$, and we use the standard notation $dx \equiv d\lambda$. In analogy with the discrete graph setting, we define
	\begin{equation*}
		\int_\mathcal{G}f(x) d\mu_{\mathcal{V}}:=\sum_{x\in\mathcal{V}}\mu_{\mathcal{V}}(x)f(x),\quad \int_{\Omega} f(x)d\mu_{\mathcal{V}}:=\int_\mathcal{G}f(x) \textbf{1}_{\Omega}(x)d\mu_{\mathcal{V}},\quad\forall\ \Omega\in\mathcal{B}.
	\end{equation*}
	The vertex-based and edge-based integral forms as above determine an integral over the metric graph:
	\begin{equation}\label{s2:22}
		\int_{\mathcal{G}}f(x)d\mu_{\mathcal{G}}:=\int_\mathcal{G}f(x) d\mu_{\mathcal{V}}+	\int_\mathcal{G}f(x) d\mu_{\mathcal{E}},\quad\forall\ f\in \mathcal{F}_0.
	\end{equation}
	The measure $\mu_\mathcal G$ combines the discrete contribution of the vertices and the continuous contribution of the edges.

	For $1\leq p<\infty$, the Lebesgue spaces on the metric graph $\mathcal{G}$ (denoted $L^p(\mathcal{G})$ or $L^p(\mathcal{G}, \mu_\mathcal{G})$) take the form
	$$
	L^p(\mathcal G)
	:=
	\left\{
	f\in \mathcal F_0:
	\sum_{e\in\mathcal E}\int_0^{l_e}|f_e(x)|^p\,dx
	+
	\sum_{x\in\mathcal V}\mu_{\mathcal V}(x)|f(x)|^p
	<\infty
	\right\},
	$$
	endowed with the norm
	$$
	\|f\|_{L^p(\mathcal G)}
	=
	\left(
	\sum_{e\in\mathcal E}\int_0^{l_e}|f_e(x)|^p\,dx
	+
	\sum_{x\in\mathcal V}\mu_{\mathcal V}(x)|f(x)|^p
	\right)^{1/p}.
	$$
	For $p=\infty$, we set
	$$
\|f\|_{L^\infty(\mathcal G)}
=
\sup_{x\in\mathcal G}|f(x)|.
	$$
	
	We further let \(\varphi:\mathcal G\to\mathbb R\) be a positive continuous function. For each \(p\in[1,\infty)\), we define the weighted Lebesgue space \(L^p_\varphi(\mathcal G)\) by
	\begin{align}\label{s2:13}
		L^p_\varphi(\mathcal G)
		:=\Bigg\{
		f\in\mathcal F_0:\ 
		\sum_{x\in\mathcal V}\mu_{\mathcal V}(x)\varphi(x)|f(x)|^p +\sum_{e\in\mathcal E}\int_0^{l_e}|f_e(x)|^p\varphi_e(x)\,dx<\infty
		\Bigg\}.
	\end{align}
	It is endowed with the norm
	\[
	\|f\|_{L^p_\varphi(\mathcal G)}
	=
	\left(
	\sum_{x\in\mathcal V}\mu_{\mathcal V}(x)\varphi(x)|f(x)|^p
	+
	\sum_{e\in\mathcal E}\int_0^{l_e}|f_e(x)|^p\varphi_e(x)\,dx
	\right)^{1/p}.
	\]

	\section{The Laplacian on metric graphs}\label{D}
	
	Let $\omega:\mathcal V\times\mathcal V\to[0,\infty)$ be a symmetric weight such that $\omega(x,y)>0$ if and only if $x\sim y$. For any $f\in \mathcal{F}_0$, the vertex-based Laplacian on $\mathcal{G}$ is defined as
	\begin{equation}\label{s2:7}
		\Delta_\mathcal{V} f(x)=\frac{1}{\mu_\mathcal{V}(x)}\sum_{y\sim x}\o(x,y)(f(y)-f(x)),
	\end{equation}
	where $y\sim x$ means $y$ is adjacent to $x$, i.e., $(x,y)\in E$. This is the usual combinatorial Laplacian. It is not difficult to see that the following integration by parts formula is valid:
	\begin{equation*}
	\sum_{x\in\mathcal{V}}\mu_{\mathcal{V}}(x)f(x)\Delta_{\mathcal{V}}g(x)=-\frac{1}{2}\sum_{x\in\mathcal V}\sum_{y\sim x}\o(x,y)(f(y)-f(x))(g(y)-g(x))=\sum_{x\in\mathcal{V}}\mu_{\mathcal{V}}(x)g(x)\Delta_{\mathcal{V}}f(x),
	\end{equation*}
	provided that at least one of the functions $f,g\in \mathcal{F}_0$ has finite support.
	
	We consider a space
	\begin{equation}\label{s2:2}
		\mathcal{D}(\mathcal{G}):=\left\{f\in C(\mathcal{G}): f_e\in C^2(I_e)\cap C^1(\overline{I}_e), f_e^{\prime\prime}\in L^\infty(I_e), \forall\ e\in\mathcal{E}\right\},
	\end{equation}
	and note that if $f \in \mathcal{D}(\mathcal{G})$, then $f$ is in $C(\mathcal{G})$, but generally not in $C^1(\mathcal{G})$; specifically, for a vertex $v \in \overline{\mathcal{G}}_{e_1} \cap \overline{\mathcal{G}}_{e_2}$ (where \(e_1, e_2 \in \mathcal E\)), it may hold that $f^\prime_{e_1}(v) \neq f^\prime_{e_2}(v)$. 
	
	Within the above functional framework, we define the edge-based Laplacian, an operator that acts on $\mathcal{D}(\mathcal{G})$ in the canonical way
	\begin{equation}\label{s2:12}
		\Delta_{\mathcal{E}}  u(x):=u^{\prime\prime}_e(x),\quad\forall\ u\in \mathcal{D}(\mathcal{G}),\ e\in\mathcal{E},\ x\in I_e.
	\end{equation}
	The outer normal derivative of $u_e$ at a vertex $v \in \mathcal{V}$ is denoted by
	\begin{equation}\label{s2:3}\frac{du_e(v)}{dn}=\left\{\begin{array}{lll}
			u_e^\prime(v), &\text{if}\ j(e)=v,&\\[1ex]
			-u_e^\prime(v),&\text{if}\ i(e)=v.&\end{array}\ri.		
	\end{equation}
	For any $x\in\mathcal{V}$, we define 
	\begin{equation}\label{s2:4}
		[\mathcal{K}(u)](x):=\sum_{e\ni x}\frac{du_e(x)}{dn}.
	\end{equation}
	Throughout the paper, every function $u\in\mathcal D(\mathcal G)$
	is assumed to satisfy the Kirchhoff condition
	\[
	[\mathcal K(u)](x)=0,
	\quad x\in\mathcal V.
	\]
	For interior vertices $x\in \mathcal{G}\setminus\p\mathcal{G}$, the condition $ [\mathcal{K}(u)](x)= 0$ is referred to as the Kirchhoff transmission condition; for boundary vertices $x \in \partial\mathcal{G}$, this condition corresponds to the homogeneous Neumann boundary condition. 
	
	Since the vertex and edge contributions are associated with different measures, it is natural to combine them into a single Laplacian in the integral sense. Specifically, we need to mark functions with $d\mu_{\mathcal{V}}$ or $d\mu_{\mathcal{E}}$ to clarify how the function should be integrated against other functions. In this paper, we use a similar setting as in \cite{F-T2,Fr} and always consider the Laplacian $\Delta_\mathcal{G}$ of the form
	\begin{equation}\label{s2:14}
		\Delta_{\mathcal{G}}f:=	(\Delta_{\mathcal{V}}f)d\mu_{\mathcal{V}}+(\Delta_{\mathcal{E}}f)d\mu_{\mathcal{E}},\quad\forall f \in \mathcal{D}(\mathcal{G}),
	\end{equation}
	where $\Delta_{\mathcal{V}}$ and $\Delta_{\mathcal{E}}$ are defined in accordance with \eqref{s2:7} and \eqref{s2:12}. According to \cite{F-T1}, as an integrating factor, $\Delta_{\mathcal{G}}f$ naturally induces a linear functional $\mathcal{L}_{\Delta_{\mathcal{G}}f}$ on $\mathcal F_0$ by means of 
	\begin{equation}\label{s2:1}
		\mathcal{L}_{\Delta_{\mathcal{G}}f}(g):=\int_\mathcal{G} g \Delta_{\mathcal{G}}f=\int_\mathcal{G}g(\Delta_{\mathcal{V}}f)d\mu_{\mathcal{V}}+\int_\mathcal{G}g(\Delta_{\mathcal{E}}f)d\mu_{\mathcal{E}}.
	\end{equation}
	The quantity $\Delta_{\mathcal G}f$ is not regarded as a pointwise-defined function. Instead, it is understood through its action on test functions via the integral identity \eqref{s2:1}. 
\end{appendices}\\

\noindent
\textbf{Acknowledgements} This paper is supported by the National Natural Science Foundation of China (Grant Number: 12471088).\\

\noindent\textbf{Conflict of interest}
The author declares there are no conflicts of interest regarding the publication of this paper.
\\

\noindent\textbf{Data availability}
Data sharing not applicable as no datasets were used or analysed during the current study.
\\

\end{document}